\documentclass[a4paper,10pt]{amsart}
\usepackage{amsmath,amssymb}
\usepackage{amscd}
\newtheorem{thm}{Theorem}[section]
\newtheorem{prop}[thm]{Proposition}
\newtheorem{lemma}[thm]{Lemma}
\newtheorem{cor}[thm]{Corollary}

\theoremstyle{remark}
\newtheorem{remark}[thm]{Remark}
\newcommand{\id}{{\rm{id}}}
\newcommand{\Ad}{{\rm{Ad}}}
\newcommand{\Hom}{{\rm{Hom}}}
\newcommand{\Alg}{{\rm{Alg}}}
\newcommand{\BN}{\mathbf N}
\newcommand{\BC}{\mathbf C}
\newcommand{\BZ}{\mathbf Z}

\newtheorem{Def}{Definition}[section]

\title{The Rohlin property for coactions of finite dimensional $C^*$-Hopf algebras on unital $C^*$-algebras}
\author{Kazunori Kodaka and Tamotsu Teruya}
\address{Department of Mathematical Sciences, Faculty of Science, Ryukyu
\endgraf
University, Nishihara-cho, Okinawa, 903-0213, Japan}
\address{Faculty of education, Gunma University, 4-2 Aramaki-machi, Maebashi City, 
\endgraf Gunma, 371-8510,
Japan}
\address{\sl{E-mail address}: \rm{kodaka@math.u-ryukyu.ac.jp}}
\address{\sl{E-mail address}: \rm{teruya@edu.gunma-u.ac.jp}}

\begin{document}
\maketitle
\begin{abstract}
We shall introduce the approximate representability and the Rohlin property for coactions
of a finite dimensional $C^*$-Hopf algebra on
a unital $C^*$-algebra and discuss some basic properties of approximately representable
coactions and coactions with the Rohlin property of a finite dimensional $C^*$-Hopf algebra on
a unital $C^*$-algebra.
Also, we shall give an example of an approximately representable coaction of a finite dimensional
$C^*$-Hopf algebra on a simple unital $C^*$-algebra which has also the Rohlin property and
we shall give the 1-cohomology
vanishing theorem for coactions of a finite dimensional $C^*$-Hopf algebra on a unital $C^*$-algebra
and the 2-cohomology vanishing theorem for twisted coactions of a finite dimensional
$C^*$-Hopf algebra on a unital $C^*$-algebra.
Furthermore,  we shall introduce the notion of the approximately unitary
equivalence of coactions of a finite dimensional $C^*$-Hopf algebra
$H$ on a unital $C^*$-algebra $A$ and show that if $\rho$ and $\sigma$, coactions
of $H$ on a separable unital $C^*$-algebra $A$, which have the Rohlin property, are
approximately unitarily equivalent, then there is an approximately inner automorphism
$\alpha$ on $A$ such that
$$
\sigma=(\alpha\otimes\id)\circ\rho\circ\alpha^{-1} .
$$

\end{abstract}
 
\section{Introduction}\label{sec:intro}
Let $A$ be a unital $C^*$-algebra and $H$ a finite dimensional $C^*$-Hopf algebra with the
comultiplication $\Delta$. In this paper, we shall introduce the approximate representability and the Rohlin property for
coactions of $H$ on $A$ and discuss some basic properties of approximately representable
coactions and coactions with the Rohlin property of $H$ on $A$.
Also, we shall give an example of an approximately representable
coaction of a finite dimensional
$C^*$-Hopf algebra on a simple unital $C^*$-algebra which has also the Rohlin property and we shall give the following
1-cohomology vanishing theorem: Let $\rho$ be a coaction of $H$ on $A$ with the Rohlin property.
Let $v$ be a unitary element in $A\otimes H$ with
$$
(v\otimes 1)(\rho\otimes \id)(v)=(\id\otimes\Delta)(v)
$$
and let $\sigma$ be the coaction of $H$ on $A$ defined by $\sigma=\Ad(v)\circ\rho$.
Then there is a unitary element $x\in A$ such that
$$
\sigma=\Ad(x\otimes 1)\circ\rho\circ\Ad(x^* ) .
$$
Furthermore ,we shall give the following 2-cohomology vanishing theorem:
Let $(\rho, u)$ be a twisted coaction of $H^0$ on $A$ with the
Rohlin property. Then there is a unitary element $x\in A\otimes H$ such that
$$
(x\otimes 1)(\rho\otimes\id)(x)u(id\otimes\Delta)(x)^* =1\otimes 1\otimes 1.
$$
Finally,  we shall introduce the notion of the approximately unitary
equivalence of coactions of
$H$ and show that if $\rho$ and $\sigma$, coactions
of $H$ on a separable unital $C^*$-algebra $A$, which have the Rohlin property, are
approximately unitarily equivalent, then there is an approximately inner automorphism
$\alpha$ on $A$ such that
$$
\sigma=(\alpha\otimes\id)\circ\rho\circ\alpha^{-1} .
$$
For an algebra $X$, we denote by $1_X $ and $\id_X $ the unit element in $X$ and the
identity map on $X$, respectively. If no confusion arises, we denote them by $1$ and $\id$,
respectively.
\par
For projections $p, q$ in a $C^*$-algebra $C$, we write $p\sim q$ in $C$ if $p$ is
Murray-von Neumann equivalent to $q$ in $C$.
\par
For each $n\in\BN$, we denote by $M_n (\BC)$ the $n\times n$-matrix algebra over $\BC$ and $I_n$
denotes the unit element in $M_n (\BC)$.

\section{Preliminaries}\label{sec:pre}
Let $H$ be a finite dimensional $C^*$-Hopf algebra. We denote its comultiplication,
counit and antipode by $\Delta$, $\epsilon$ and $S$. We shall use Sweedler's notation
$\Delta(h)=h_{(1)}\otimes h_{(2)}$ for any $h\in H$ which suppresses a possible summation
when we write the comultiplications. We denote by $N$ the dimension of $H$.
Let $H^0 $ be the dual $C^*$-Hopf algebra of $H$. We denote its comultiplication, counit
and antipode by $\Delta^0 $, $\epsilon^0$ and $S^0 $. There is the distinguished projection $e$
in $H$. We note that $e$ is the Haar trace on $H^0$.
Also, there is the distinguished projection $\tau$ in $H^0$ which is the
Haar trace on $H$.
\par
Throughout this paper, $H$ denotes a finite dimensional $C^*$-Hopf algebra
and $H^0$ its dual $C^*$-Hopf algebra. Since $H$ is finite dimensional,
$H\cong\oplus_{k=1}^K M_{d_k }(\BC)$ as $C^*$-algebras.
Let $\{v_{ij}^k \, | \, k=1,2,\dots, K, \, i,j=1,2,\dots, d_k \}$ be a system of matrix units of $H$.
Let $\{w_{ij}^k \, | \, k=1,2,\dots, K, \, i,j=1,2,\dots, d_k \}$ be a basis of $H$ satisfying
Szyma\'nski and Peligrad \cite [Theorem 2.2,2]{SP:saturated}. We call it a system of
\it
comatrix units
of $H$.
\rm
Also, let $\{\phi_{ij}^k \, | \, k=1,2,\dots, K, \, i,j=1,2,\dots, d_k \}$ and
$\{\omega_{ij}^k \, | \, k=1,2,\dots, K, \, i,j=1,2,\dots, d_k \}$ be systems of matrix units and comatrix units
of $H^0$, respectively. Furthermore, let $\rho_H^A$ be the trivial coaction of $H$
on $A$ defined by $\rho_H^A (a)=a\otimes 1$ for any $a\in A$.
\par
Following Masuda and Tomatsu \cite{MT:minimal},
we shall define a twisted coaction of $H$ on $A$ and its exterior equivalence.

\begin{Def}\label{Def:twisted}Let $\rho$ be a weak coaction of $H$ on $A$ and $u$
a unitary element in $A\otimes H\otimes H$. The pair $(\rho, u)$ is a
\sl
twisted
\rm
coaction of $H$ on $A$ if the following conditions hold:
\newline
(1) $(\rho\otimes\id)\circ\rho=\Ad(u)\circ(\id\otimes\Delta)\circ\rho$,
\newline
(2) $(u\otimes 1)(\id\otimes\Delta\otimes\id)(u)=(\rho\otimes\id\otimes\id)(u)
(\id\otimes\id\otimes\Delta)(u)$,
\newline
(3) $(\id\otimes\phi\otimes\epsilon)(u)=(\id\otimes\epsilon\otimes\phi)(u)=
\phi(1)1$ for any $\phi\in H^0$.
\end{Def}
\begin{Def}\label{Def:equivalence}For $i=1, 2$, let $(\rho_i , u_i )$ be a twisted coaction of $H$
on $A$. We say that $(\rho_1 , u_1 )$ is
\sl
exterior
\rm
equivalent to $(\rho_2 , u_2 )$ if there is a unitary element $v$ in $A\otimes H$ satisfying following
conditions:
\newline
(1) $\rho_2 =\Ad(v)\circ\rho_1 $,
\newline
(2) $u_2 =(v\otimes 1)(\rho_1 \otimes\id)(v)u_1 (\id\otimes\Delta)(v^* )$.
\end{Def}

By routine computations, $(\id\otimes \epsilon)(v)=1$ and the above equivalence is an equivalence one.
We write  $(\rho_1 , u_1 )\sim(\rho_2 , u_2 )$ if $(\rho_1 , u_1 )$ is exterior equivalent to
$(\rho_2 , u_2 )$.
\begin{remark}\label{rem:coboundary}Let $(\rho, u)$ be a twisted coaction of $H$
on $A$ and $v$ any unitary element in $A\otimes H$ with $(\id\otimes\epsilon)(v)=1$.
Let
$$
\rho_1 =\Ad(v)\circ\rho, \quad
u_1 =(v\otimes 1)(\rho\otimes\id)(v)u(\id\otimes\Delta)(v^* ) .
$$
Then $(\rho_1 , u_1 )$ is a twisted coaction of $H$ on $A$ by easy computations.
\end{remark}
Let $\Hom(H^0 , A)$ be the linear space of all linear maps from $H^0$ to $A$.
By Sweedler \cite[pp. 69--70]{Sweedler:hopf}
it becomes a unital $*$-algebra which is also defined in \cite[Sections 2 and 3]
{KT1:inclusion}. In the same way as \cite[Sections 2 and 3]{KT1:inclusion}, we define a unital
$*$-algebra $\Hom(H^0 \otimes H^0 , A)$.
As mentioned in Blattner, Cohen and Montgomery \cite[pp. 163]{BCM:crossed},
there are an isomorphism $\imath$ of $A\otimes H$ onto $\Hom(H^0 , A)$ and an isomorphism
$\jmath $ of $A\otimes H\otimes H$ onto $\Hom(H^0 \otimes H^0 , A)$ defined by
\begin{align*}
\imath(a\otimes h)(\phi) & =\phi(h)a, \\
 \jmath(a\otimes h\otimes l)(\phi , \psi) & =\phi(h)\psi(l)a
\end{align*}
for any $a\in A$, $h, l\in H$ and $\phi, \psi\in H^0$.
For any $x\in A\otimes H$, $y\in A\otimes H\otimes H$,
we denote $\imath(x)$, $\jmath(y)$, by $\widehat{x}$, $\widehat{y}$,
respectively.
\par
For any weak coaction $\rho$ of $H$ on $A$, we can construct the weak
action $\cdot_{\rho}$ of $H^0$ on $A$
as follows: For any $a\in A$ and $\phi\in H^0$
$$
\phi\cdot_{\rho} a=\imath(\rho(a))=\rho(a)^{\widehat{}}(\phi)=(\id\otimes\phi)(\rho(a)) .
$$
If no confusion arises, we denote $\phi\cdot_{\rho}a$ by $\phi\cdot a$ for any $a\in A$ and $\phi\in H^0$.
Furthermore, if $(\rho, u)$ is a twisted coaction of $H$ on $A$, $\widehat{u}$ is a unitary cocycle
for the above weak action induced by $\rho$. We call the pair of the weak action and the unitary
cocycle $\widehat{u}$ the
\sl
twisted
\rm
action of $H^0$ on $A$ induced by $(\rho, u)$. By \cite[Section 3]{KT1:inclusion}, we can construct
the twisted crossed product of $A$ by $H^0$ which is denoted by $A\rtimes_{\rho, u}H^0$.
Let $\widehat{\rho}$ be the dual coaction of $\rho$, which is defined by for any $a\in A$, $\phi\in H^0$,
$$
\widehat{\rho}(a\rtimes_{\rho, u}\phi)=(a\rtimes_{\rho, u}\phi_{(1)})\otimes\phi_{(2)} ,
$$
where $a\rtimes_{\rho, u}\phi$ denotes the element in $A\rtimes_{\rho, u}H^0$ induced by
$a \in A$ and $\phi\in H^0$. If no confusion arises, we denote it by $a\rtimes\phi$.
\par
Let $(\rho, u)$ be a twisted coaction of $H$ on $A$ and $A\rtimes_{\rho, u}H^0$
the twisted crossed product induced by $(\rho, u)$. Let $E_1^{\rho}$ be the canonical conditional
expectation from $A\rtimes_{\rho, u}H^0$ onto $A$ defined by $E_1^{\rho} (a\rtimes\phi)=\phi(e)a$
for any $a\in  A$, $\phi\in H^0$.
We note that $E_1^{\rho} $ is faithful by \cite [Lemma 3.14]{KT1:inclusion}. Also, let $\widehat{V}$ be
an element in $\Hom(H^0 , A\rtimes_{\rho, u}H^0 )$
defined by $\widehat{V}(\phi)=1\rtimes\phi$ for any $\phi\in H^0 $. Let $V$ be an element in
$(A\rtimes_{\rho, u}H^0 )\otimes H$ induced by $\widehat{V}$.
By \cite [Lemma 3.12]{KT1:inclusion}, we can see that $V$ and $\widehat{V}$ are unitary
elements in $(A\rtimes_{\rho, u}H^0 )\otimes H$ and $\Hom(H^0 , A\rtimes_{\rho, u}H^0 )$, respectively and that
$$
u=(V\otimes 1)(\rho_H^{A\rtimes_{\rho, u}H^0 }\otimes\id)(V)(\id\otimes\Delta)(V^* ) .
$$
Thus, for any $\phi, \psi\in H^0$
\newline
(1) $\widehat{u}(\phi, \psi)=\widehat{V}(\phi_{(1)})\widehat{V}(\psi_{(1)})\widehat{V^* }(\phi_{(2)}\psi_{(2)})$,
\newline
(2) $\widehat{u^* }(\phi, \psi)=\widehat{V}(\phi_{(1)}\psi_{(1)})\widehat{V^* }(\psi_{(2)})\widehat{V^* }(\phi_{(2)})$.

\begin{lemma}\label{lem:iso}For $i=1,2$ let $(\rho_i , u_i )$ be a twisted coaction of $H$ on $A$ with
$(\rho_1 , u_1 )\sim(\rho_2 , u_2 )$. Let $E_1^{\rho_i }$ be the canonical conditional
expectation from $A\rtimes_{\rho_i , u_i }H^0$ onto $A$ for $i=1,2$. Then there is an isomorphism
$\Phi$ of $A\rtimes_{\rho_1 , u_1 }H^0$ onto $A\rtimes_{\rho_2 , u_2 }H^0$ satisfying that
$\Phi(a)=a$ for any $a\in A$ and $E_1^{\rho_1 }=E_1^{\rho_2 } \circ\Phi$, where
$A$ is identified with $A\rtimes_{\rho_i , u_i }1^0$ for $i=1,2$.
\end{lemma}
\begin{proof}Since $(\rho_1 , u_1 )\sim(\rho_2 , u_2 )$, there is a unitary element in $v$ in $A\otimes H$
satisfying that
$$
\rho_2 =\Ad(v)\circ\rho_1 , \quad
u_2 =(v\otimes 1)(\rho_1 \otimes\id)(v)u_1 (\id\otimes\Delta)(v^* ) .
$$
Let $\Phi$ be a map from $A\rtimes_{\rho_1 , u_1 }H^0$ to $A\rtimes_{\rho_2 , u_2 }H^0$ defined
by $\Phi(a\rtimes_{\rho_1 , u_1 }\phi)=a\widehat{v}^* (\phi_{(1)})\rtimes_{\rho_2 , u_2} \phi_{(2)}$
for any $a\in A$, $\phi\in H^0$. Then by routine computations, $\Phi$ is a homomorphism
of $A\rtimes_{\rho_1 ,u_1}H^0$ to $A\rtimes_{\rho_2 ,u_2}H^0$. Also, let $\Psi$ be a map
from $A\rtimes_{\rho_2 ,u_2}H^0$ to $A\rtimes_{\rho_1 ,u_1}H^0$ defined by for any
$a\in A$, $\phi\in H^0$, $\Psi(a\rtimes_{\rho_2 , u_2 }\phi)=a\widehat{v} (\phi_{(1)})\rtimes_{\rho_1 , u_1} \phi_{(2)}$.
By routine computations, $\Psi$ is also a homomorphism of
$A\rtimes_{\rho_2 ,u_2}H^0$ to $A\rtimes_{\rho_1 ,u_1}H^0$
and $\Phi\circ\Psi=\id$ and $\Psi\circ\Phi=\id$. Therefore, we obtain the conclusion.
\end{proof}

Let $\rho$ be a coaction of $H$ on $A$ and $A^{\rho}$ the fixed point $C^*$-subalgebra of $A$
for $\rho$, that is,
$$
A^{\rho}=\{a\in A \, | \, \rho(a)=a\otimes 1 \} .
$$
Let $E^{\rho}$ be the canonical conditional expectation from $A$ onto
$A^{\rho}$ defined by $E^{\rho}(a)=\tau\cdot_{\rho}a=(\id\otimes\tau)(\rho(a))$ for
any $a\in A$. We note that $E^{\rho}$ is faithful by \cite [Proposition 2.12]{SP:saturated}.
\begin{Def}\label{Def:saturated}We say that $\rho$ is
\sl
saturated
\rm
if the action of $H^0$ on $A$ induced by $\rho$ is saturated in the sense of \cite {SP:saturated}.
\end{Def}

In Sections 4, 5 and 6 of \cite{KT1:inclusion}, we suppose that the action of
$H$ on $A$ is saturated. But, without saturation, all the statements in Sections 4 and 5
and Theorem 6.4 of \cite{KT1:inclusion} hold. Hence we obtain the following
proposition.

\begin{prop}\label{prop:key}Let $\rho$ be a coaction of $H$ on $A$ such that
$\widehat{\rho}(1\rtimes\tau)\sim(1\rtimes\tau)\otimes 1^0 $ in $(A\rtimes_{\rho} H^0 )\otimes H^0$.
Then there are a twisted coaction $(\alpha, u)$ of $H^0$ on $A^{\rho}$ and an
isomorphism $\pi$ of $A^{\rho}\rtimes_{\alpha, u}H$ onto $A$ such that $E_1^{\alpha}=E^{\rho}\circ\pi$
and $\rho\circ\pi=(\pi\otimes\id)\circ\widehat{\alpha}$.
\end{prop}

\begin{cor}\label{cor:saturated}Let $\rho$ be a coaction of $H$ on $A$ such that
$\widehat{\rho}(1\rtimes\tau)\sim(1\rtimes\tau)\otimes1^0$ in $(A\rtimes_{\rho} H^0 )\otimes H^0$.
Then $\rho$ is saturated.
\end{cor}
\begin{proof}
Since the dual coaction of a twisted coaction is saturated, this is immediate by Proposition \ref{prop:key}.
\end{proof}

\section{Duality}\label{sec:duality}In this section we shall show the duality theorem for a twisted
coaction of $H^0$ on $A$. It has already proved, but we shall present it in a useful form in this paper.
\par
Let $(\rho, u)$ be a twisted coaction of $H^0$ on $A$. Let $\Lambda$ be the set of all triplets $(i,j,k)$,
where $i,j=1,2,\dots,d_k$ and $k=1,2,\dots, K$ amd $\sum_{k=1}^K d_k^2 =N$.
For any $I=(i,j,k)\in \Lambda$, let $W_I$ and $V_I$ be elements in
$A\rtimes_{\rho, u}H\rtimes_{\widehat{\rho}}H^0 $ defined by
\begin{align*}
W_I & =\sqrt{d_k}\rtimes_{\rho,u}w_{ij}^k , \\
V_I & =(1\rtimes_{\widehat{\rho}}\tau)(W_I\rtimes_{\widehat{\rho}}1^0 ) .
\end{align*}
\begin{lemma}\label{lem:orthogonal}With the above notations,
$$
V_I V_J^* =
\begin{cases}1\rtimes_{\widehat{\rho}}\tau & \text{if $I=J$} \\
0 & \text{if $I\ne J$.}
\end{cases}
$$
\end{lemma}
\begin{proof}Let $I=(i,j,k)$ and $J=(s,t,r)$ be any elements in $\Lambda$.
Then
\begin{align*}
V_I V_J^* & =(1\rtimes_{\widehat{\rho}}\tau)(W_I \rtimes_{\widehat{\rho}}1^0 )(W_J^* \rtimes_{\widehat{\rho}}1^0 )
(1\rtimes_{\widehat{\rho}}\tau) \\
& =(1\rtimes_{\widehat{\rho}}\tau)(W_I W_J^* \rtimes_{\widehat{\rho}}1^0 )(1\rtimes_{\widehat{\rho}}\tau)
=[\tau_{(1)}\cdot_{\widehat{\rho}}W_I W_J^* ]\rtimes_{\widehat{\rho}}\tau_{(2)}\tau' \\
& =[\tau\cdot_{\widehat{\rho}}W_I W_J^* ]\rtimes_{\widehat{\rho}}\tau
=E_1^{\rho}(W_I W_J^* )\rtimes_{\widehat{\rho}}\tau ,
\end{align*}
where $\tau' =\tau$. Here, by \cite [Lemma 3.3 (1)]{KT1:inclusion} and \cite [Theorem 2.2]{SP:saturated}
\begin{align*}
W_I W_J^* & =(\sqrt{d_k }\rtimes_{\rho, u}w_{ij}^k )(\sqrt{d_r}\rtimes_{\rho, u}w_{st}^r )^*  \\
& =\sum_{t_1, t_2 }\sqrt{d_k d_r }(1\rtimes_{\rho, u}w_{ij}^k )(\widehat{u}(S(w_{t_2 t_1 }^r ), w_{st_2 }^r )^* 
\rtimes_{\rho, u}w_{t_1 t}^{r*} ) \\
& =\sum_{t_1, t_2 , j_1 , j_2 , m}\sqrt{d_k d_r}[w_{ij_2 }^k \cdot_{\rho, u}\widehat{u}(S(w_{t_2  t_1 }^r ), w_{st_2 }^r )^* ]
\widehat{u}(w_{j_2  j_1 }^k , w_{t_1 m}^{r*})\rtimes_{\rho, u}w_{j_1 j}^k w_{mt}^{r*} \\
& =\sum_{t_1, t_2 , j_1 , j_2 , m}\sqrt{d_k d_r}[w_{j_2 i}^k \cdot_{\rho, u}\widehat{u}(S(w_{t_2  t_1 }^r ), w_{st_2 }^r ) ]^*
\widehat{u}(w_{j_2  j_1 }^k , w_{t_1 m}^{r*})\rtimes_{\rho, u}w_{j_1 j}^k w_{mt}^{r*} \\
& =\sum_{t_1, t_2 , t_3 , j_1 , j_2 , j_3 , m}\sqrt{d_k d_r}[\widehat{u}(w_{j_2  j_3 }^k , S(w_{t_3 t_1 }^r ))
\widehat{u}(w_{j_3 i}^k
S(w_{t_2 t_3 }^r ), w_{st_2 }^r )]^* \\
& \times \widehat{u}(w_{j_2  j_1 }^k , w_{t_1 m}^{r*})
\rtimes_{\rho, u}w_{j_1 j}^k w_{mt}^{r*} \\
& =\sum_{t_1, t_2 , t_3 , j_1 , j_2 , j_3 , m}\sqrt{d_k d_r}\widehat{u}(w_{j_3 i}^k
S(w_{t_2 t_3 }^r ), w_{st_2 }^r )^* \widehat{u^*}(w_{j_3  j_2 }^k , w_{t_3 t_1 }^{r*} )
\widehat{u}(w_{j_2  j_1 }^k , w_{t_1 m}^{r*}) \\
& \rtimes_{\rho, u}w_{j_1 j}^k w_{mt}^{r*} \\
& =\sum_{t_2 , t_3 , j_3}\sqrt{d_k d_r}\widehat{u}(w_{j_3 i}^k S(w_{t_2 t_3 }^r ), w_{st_2 }^r )^* \rtimes_{\rho, u}
w_{j_3 j}^k w_{t_3 t}^{r*} .
\end{align*}
Thus, by \cite [Theorem 2.2]{SP:saturated}
$$
V_I V_J^* =\sum_{t_2 , t_3 , j_3}\sqrt{d_k d_r }\tau(w_{j_3 j}^k w_{t_3 t}^{r*})
\widehat{u^* }(w_{ij_3 }^k w_{t_2 t_3}^{r*}, w_{t_2 s}^r )\rtimes_{\widehat{\rho}}\tau .
$$
If $k\ne r$ or $j\ne t$, then $V_I V_J^* =0$. We suppose that $k=r$ and $j=t$. Then
$$
V_I V_J^* =\sum_{t_2 , t_3}
\widehat{u^* }(w_{it_3 }^k S(w_{t_3 t_2}^k), w_{t_2 s}^k )\rtimes_{\widehat{\rho}}\tau
=\epsilon(w_{is}^k )\rtimes_{\widehat{\rho}}\tau=\delta_{is}\rtimes_{\widehat{\rho}}\tau ,
$$
where $\delta_{is}$ is the Kronecker delta.
Therefore by \cite [Theorem 2.2]{SP:saturated}, we obtain the conclusion.
\end{proof}
Let $\Psi$ be a map from $M_N (A)$ to $A\rtimes_{\rho, u}H\rtimes_{\widehat{\rho}}H^0$ defined by
$$
\Psi([a_{IJ}])=\sum_{I,J}V_I^* (a_{IJ}\rtimes_{\rho, u}1\rtimes_{\widehat{\rho}}1^0 )V_J
$$
for any $[a_{IJ}]\in M_N (A)$. Clearly $\Psi$ is a linear map.

\begin{prop}\label{prop:map}The map $\Psi$ is an isomorphism of $M_N (A)$ onto
$A\rtimes_{\rho, u}H\rtimes_{\widehat{\rho}}H^0$.
\end{prop}
\begin{proof}For any $[a_{IJ}], [b_{IJ}]\in M_N (A)$,
\begin{align*}
\Psi([a_{IJ}]\Psi([b_{IJ}]) & =\sum_{I,J,L}V_I^* (a_{IJ}\rtimes_{\rho, u}1
\rtimes_{\widehat{\rho}}1^0 )(1\rtimes_{\widehat{\rho}}\tau)(b_{JL}\rtimes_{\rho, u}1\rtimes_{\widehat{\rho}}1^0 )V_L \\
& =\sum_{I,J,L}V_I^* (1\rtimes_{\widehat{\rho}}\tau)(a_{IJ}b_{JL}\rtimes_{\rho, u}1\rtimes_{\widehat{\rho}}1^0 )
V_L \\
& =\sum_{I,J,L}V_I^* (a_{IJ}b_{JL}\rtimes_{\rho, u}1\rtimes_{\widehat{\rho}}1^0 )V_L =\Psi([a_{IJ}][b_{IJ}])
\end{align*}
by Lemma \ref{lem:orthogonal}. For any $[a_{IJ}]\in M_N (A)$,
$$
\Psi([a_{IJ}])^* =\sum_{I,J}V_J^* (a_{IJ}^* \rtimes_{\rho, u}1\rtimes_{\widehat{\rho}}1^0 )V_I =\Psi([a_{JI}^* ]) .
$$
Hence $\Psi$ is a homomorphism of $M_N (A)$ to $A\rtimes_{\rho, u}H\rtimes_{\widehat{\rho}}H^0$.
Since $\widehat{\rho}$ is saturated, for any $z\in A\rtimes_{\rho, u}H\rtimes_{\widehat{\rho}}H^0$,
we can write that
$$
z=\sum_{i=1}^n (x_i \rtimes_{\widehat{\rho}}1^0 )(1\rtimes_{\widehat{\rho}}\tau)(y_i \rtimes_{\widehat{\rho}}1^0 )
$$
by \cite [Proposition 4.5]{SP:saturated}, where $x_i , y_i \in A\rtimes_{\rho, u}H$
for $i=1,2,\dots, n$. Thus, in order to prove that $\Psi$ is surjective, it suffices to show that
for any $x,y\in A\rtimes_{\rho, u}H$, there is an element $[a_{IJ}]\in M_N (A)$
such that $\Psi([a_{IJ}])=(x\rtimes_{\widehat{\rho}}1^0 )(1\rtimes_{\widehat{\rho}}\tau)(y\rtimes_{\widehat{\rho}}1^0 )$.
Since $\{(W_I^* , W_I \})$ is a quasi-basis for $E_{\rho}$ by \cite [Proposition 3.18]{KT1:inclusion},
\begin{align*}
x & =\sum_I W_I^* E_1^{\rho}(W_I x)=\sum_I W_I^* (E_1^{\rho}(W_I x)\rtimes_{\rho, u}1), \\
y & =\sum_I E_1^{\rho}(yW_I^* )W_I =\sum_I (E_1^{\rho}(yW_I^* )\rtimes_{\rho, u}1)W_I .
\end{align*}
Hence
\begin{align*}
& (x\rtimes_{\widehat{\rho}}1^0 )(1\rtimes_{\widehat{\rho}}\tau)(y\rtimes_{\widehat{\rho}}1^0 )
=\sum_{I,J}(W_I^* E_1^{\rho}(W_I x)\rtimes_{\widehat{\rho}}1^0 )(1\rtimes_{\widehat{\rho}}\tau)(E_1^{\rho}(yW_J^* )W_J
\rtimes_{\widehat{\rho}}1^0 ) \\
& =\sum_{I,J}(W_I^* \rtimes_{\widehat{\rho}}1^0 )(1\rtimes_{\widehat{\rho}}\tau)(E_1^{\rho}(W_I x)
E_1^{\rho}(yW_J^* )\rtimes_{\rho, u}1\rtimes_{\widehat{\rho}}1^0 )
(1\rtimes_{\widehat{\rho}}\tau )(W_J \rtimes_{\widehat{\rho}}1^0 ) \\
& =\sum_{I,J}V_I^* (E_1^{\rho}(W_I x)E_1^{\rho}(yW_J^* )\rtimes_{\rho, u}1\rtimes_{\widehat{\rho}}1^0 )V_J \\
& =\Psi([E_1^{\rho}(W_I x)E_1^{\rho}(yW_J^* )]_{I,J}) .
\end{align*}
Next, we shall show that $\Psi$ is injective. We suppose that for an element $[a_{IJ}]\in M_N (A)$,
$\Psi([a_{IJ}])=0$. Then $\sum_{I,J}V_I^* (a_{IJ}\rtimes_{\rho, u}1\rtimes_{\widehat{\rho}}1^0 )V_J =0$.
Thus for any $M, L \in \Lambda$,
$$
0=V_M \sum_{I,J}V_I^* (a_{IJ}\rtimes_{\rho, u}1\rtimes_{\widehat{\rho}}1^0 )V_J V_L^* 
=a_{ML}\rtimes_{\rho, u}1\rtimes_{\widehat{\rho}}1^0
$$
by Lemma \ref{lem:orthogonal}. Hence $a_{ML}=0$ for any $M,L\in\Lambda$. Therefore, $\Psi$ is injective.
\end{proof}

Since $V_I V_I^*=1\rtimes_{\widehat{\rho}}\tau$ for any $I\in\Lambda$ by Lemma \ref{lem:orthogonal},
the set $\{V_I^* V_I \}_{I\in\Lambda}$ is a family of orthogonal projections in
$A\rtimes_{\rho, u}H\rtimes_{\widehat{\rho}}H^0$.
Let $P_I =V_I^* V_I$ for any $I\in\Lambda$. By Lemma \ref{lem:orthogonal}
and Proposition \ref{prop:map}, 
$$
1=\Psi(1\otimes I_N )=\sum_{I\in\Lambda}V_I^* V_I=\sum_{I\in\Lambda}P_I ,
$$
where $I_N$ is the unit element in $M_N (\BC)$.
\par
We recall that $\widehat{V}$ is a unitary element in $\Hom (H, A\rtimes_{\rho, u}H)$
defined by $\widehat{V}(h)=1\rtimes_{\rho, u}h$ for any $h\in H$.
Let $V$ be the unitary element in $(A\rtimes_{\rho, u}H)\otimes H^0$ induced
by $\widehat{V}$. We regard $A\rtimes_{\rho, u}H$ as a $C^*$-subalgebra
$A\rtimes_{\rho, u}H\rtimes_{\widehat{\rho}}1^0 $ of $A\rtimes_{\rho, u}H\rtimes_{\widehat{\rho}}H^0$.
Thus we regard $V$ as a unitary element in $(A\rtimes_{\rho, u}H\rtimes_{\widehat{\rho}}H^0 )\otimes H^0$.
For any $I\in\Lambda$, let
$$
U_I =(V_I^* \otimes 1^0 )V\widehat{\widehat{\rho}}(V_I )\in (A\rtimes_{\rho, u}H\rtimes_{\widehat{\rho}}H^0 )\otimes H^0 .
$$
Then for any $I\in\Lambda$, $U_I U_I^* =P_I \otimes 1^0 $ and $U_I^* U_I =\widehat{\widehat{\rho}}(P_I )$
since
$$
\widehat{\widehat{\rho}}(1\rtimes_{\widehat{\rho}}\tau)=V^* [(1\rtimes_{\widehat{\rho}}\tau)\otimes 1^0 ]V
$$
by \cite [Proposition 3.19]{KT1:inclusion}. Let $U=\sum_{I\in\Lambda}U_I $. Then $U$ is a unitary element in
$(A\rtimes_{\rho, u}H\rtimes_{\widehat{\rho}}H^0 )\otimes H^0 $. Since $(\rho, u)$ is a twisted coaction of $H^0$ on $A$,
$(\rho\otimes\id_{M_N (\BC)}, u\otimes I_N )$ is also a twisted coaction of $H^0$ on $M_N (A)$.
Then by easy computations,
$$
((\Psi\otimes\id_{H^0})\circ(\rho\otimes\id_{M_N (\BC)})\circ\Psi^{-1}, \, (\Psi\otimes\id_{H^0 }\otimes\id_{H^0 })
(u\otimes I_N))
$$
is a twisted coaction of $H^0$ on $A\rtimes_{\rho, u}H\rtimes_{\widehat{\rho}}H^0$, where we identify
$A\otimes M_N (\BC)\otimes H^0 \otimes H^0$ with $A\otimes H^0 \otimes H^0 \otimes M_N (\BC)$.

\begin{thm}\label{thm:duality}Let $A$ be a unital $C^*$-algebra and $H$ a finite dimensional
$C^*$-Hopf algebra with its dual $C^*$-Hopf algebra $H^0$. Let $(\rho, u)$ be a twisted coaction
of $H^0$ on $A$. Then there is an isomorphism $\Psi$ of $M_N (A)$ onto
$A\rtimes_{\rho, u}H\rtimes_{\widehat{\rho}}H^0$ and a unitary element
$U\in (A\rtimes_{\rho, u}H\rtimes_{\widehat{\rho}}H^0 )\otimes H^0 $ such that
\begin{align*}
& \Ad(U)\circ\widehat{\widehat{\rho}}=(\Psi\otimes\id_{H^0})\circ(\rho\otimes\id_{M_N (\BC)})\circ\Psi^{-1} , \\
& (\Psi\otimes\id_{H^0}\otimes\id_{H^0})(u\otimes I_N )=(U\otimes 1^0 )(\widehat{\widehat{\rho}}\otimes\id_{H^0})(U)
(\id\otimes\Delta^0 )(U^* ) .
\end{align*}
That is, $\widehat{\widehat{\rho}}$ is exterior equivalent to the twisted coaction
$$
((\Psi\otimes\id_{H^0})\circ(\rho\otimes\id_{M_N (\BC)})\circ\Psi^{-1}, \, (\Psi\otimes\id_{H^0 }\otimes\id_{H^0 })
(u\otimes I_N)) .
$$
\end{thm}
\begin{proof}Let $\Psi$ be the isomorphism of $M_N (\BC)$ onto $A\rtimes_{\rho, u}H\rtimes_{\widehat{\rho}}H^0$
defined in Proposition \ref{prop:map} and let $U$ be a unitary element in
$(A\rtimes_{\rho, u}H\rtimes_{\widehat{\rho}}H^0 )\otimes H^0 $ defined above.
Let $[a_{IJ}]_{I,J\in\Lambda}$ be any element in $M_N (A)$. Then
\begin{align*}
& (\Ad(U)\circ\widehat{\widehat{\rho}})(\Psi([a_{IJ}]))=U\widehat{\widehat{\rho}}(\sum_{I,J}V_I^* (a_{IJ}\rtimes_{\rho, u}1
\rtimes_{\widehat{\rho}}1^0 )V_J )U^* \\
& =\sum_{I,J}(V_I^* \otimes 1^0 )V\widehat{\widehat{\rho}}(1\rtimes_{\widehat{\rho}}\tau)\widehat{\widehat{\rho}}
(a_{IJ}\rtimes_{\rho, u}1\rtimes_{\widehat{\rho}}1^0 )
\widehat{\widehat{\rho}}(1\rtimes_{\widehat{\rho}}\tau)V^* (V_J\otimes 1^0 ) \\
& =\sum_{I,J}(V_I^* \otimes 1^0 )[(1\rtimes_{\widehat{\rho}}\tau)\otimes 1^0 ]
V((a_{IJ}\rtimes_{\rho, u}1\rtimes_{\widehat{\rho}}1^0 )\otimes 1^0 )V^*
[(1\rtimes_{\widehat{\rho}}\tau)\otimes 1^0 ](V_J\otimes 1^0 )
\end{align*}
since $\widehat{\widehat{\rho}}(1\rtimes_{\widehat{\rho}}\tau)=V^* [(1\rtimes_{\widehat{\rho}}\tau)\otimes 1^0 ]V$
by \cite [Proposition 3.19]{KT1:inclusion}, where we identify $A$ with
$A\rtimes_{\rho, u}1$ and $A\rtimes_{\rho, u}1\rtimes_{\widehat{\rho}}1^0$.
Hence
$$
(\Ad(U)\circ\widehat{\widehat{\rho}})(\Psi([a_{IJ}]))
=\sum_{I,J}(V_I^* \otimes 1^0 )\rho(a_{IJ}\rtimes_{\rho, u}1\rtimes_{\widehat{\rho}}1^0 )(V_J \otimes 1^0 )
$$
since $\rho(a)=V(a\otimes 1^0 )V^* $ for any $a\in A$ by \cite [Lemma 3.12(1)]{KT1:inclusion}.
On the other hand,
$$
((\Psi\otimes\id_{H^0})\circ(\rho\otimes\id))([a_{IJ}])
=(\Psi\otimes\id_{H^0})([\rho(a_{IJ}\rtimes_{\rho ,u}1\rtimes_{\widehat{\rho}}1^0 ]) .
$$
Since $\rho(a_{IJ}\rtimes_{\rho ,u}1\rtimes_{\widehat{\rho}}1^0 )\in A\otimes H^0$. we can write that
$$
\rho(a_{IJ}\rtimes_{\rho ,u}1\rtimes_{\widehat{\rho}}1^0 )
=\sum_i (b_{IJi}\rtimes_{\rho, u}1\rtimes_{\widehat{\rho}}1^0 )\otimes\phi_{IJi} ,
$$
where $b_{IJi}\in A$ and $\phi_{IJi}\in H^0 $ for any $I,J,i$.
Hence
\begin{align*}
(\Psi\otimes\id_{H^0}) & ([\rho(a_{IJ}\rtimes_{\rho, u}1\rtimes_{\widehat{\rho}}1^0 )])
=(\Psi\otimes\id_{H^0})(\sum_i (b_{IJi}\rtimes_{\rho, u}1\rtimes_{\widehat{\rho}}1^0 )\otimes\phi_{IJi}) \\
& =\sum_{i}V_I^* (b_{IJi}\rtimes_{\rho, u}1\rtimes_{\widehat{\rho}}1^0 )V_J \otimes\phi_{IJi} \\
& =\sum_{i}(V_I^* \otimes 1^0 )[(b_{IJi}\rtimes_{\rho, u}1\rtimes_{\widehat{\rho}}1^0 )
\otimes\phi_{IJi}](V_J \otimes1^0 ) \\
& =\sum_{i}(V_I^* \otimes 1^0 )\rho(a_{IJ}\rtimes_{\rho, u}1\rtimes_{\widehat{\rho}}1^0 )
(V_J \otimes1^0 ) .
\end{align*}
Thus we obtain that
$$
\Ad(U)\circ\widehat{\widehat{\rho}}\circ\Psi=(\Psi\otimes\id_{H^0})\circ(\rho\otimes\id_{M_N (\BC)}) .
$$
Next, we shall show that
$$
(\Psi\otimes\id_{H^0}\otimes\id_{H^0})(u\otimes I_N )=(U\otimes 1^0)(\widehat{\widehat{\rho}}\otimes\id_{H^0})(U)
(\id\otimes\Delta^0 )(U^* ) .
$$
Since $u\in A\otimes H^0 \otimes H^0 $, we can write that
$u=\sum_{i,j}a_{ij}\otimes\phi_i \otimes\psi_j $,
where $a_{ij}\in A$ and $\phi_i $, $\psi_j \in H^0$ for any $i,j$.
Thus for any $h,l\in H$
\begin{align*}
(\Psi\otimes\id_{H^0}\otimes\id_{H^0})(u\otimes I_N )^{\widehat{}}(h, l)
& =\sum_{I,i,j}V_I^* (a_{ij}\rtimes_{\rho, u}1\rtimes_{\widehat{\rho}}1^0 )V_I \phi_i(h)\psi_j (l) \\
& =\sum_I V_I^* (\widehat{u}(h, l)\rtimes_{\rho, u}1\rtimes_{\widehat{\rho}}1^0 )V_I .
\end{align*}
On the other hand, by Lemma \ref{lem:orthogonal} and \cite [Proposition 3.19]{KT1:inclusion}
\begin{align*}
& (U\otimes 1^0 )(\widehat{\widehat{\rho}}\otimes\id_{H^0})(U)(\id\otimes\Delta^0 )(U^* ) \\
& =[\sum_I (V_I^* \otimes1^0 \otimes1^0 )(V\otimes 1^0 )(\widehat{\widehat{\rho}}(V_I )\otimes 1^0 )]
[\sum_J (\widehat{\widehat{\rho}}(V_J^* )\otimes 1^0 )(\widehat{\widehat{\rho}}\otimes\id_{H^0 })(V) \\
& \times ((\widehat{\widehat{\rho}}\otimes\id)\circ\widehat{\widehat{\rho}})(V_J )]
[\sum_L (\id\otimes\Delta^0 )(\widehat{\widehat{\rho}}(V_L^* ))
(\id\otimes\Delta^0 )(V^* )(V_L \otimes 1^0 \otimes 1^0 )] \\
& =\sum_I (V_I^* \otimes 1^0 \otimes 1^0 )(V\otimes 1^0 )(\widehat{\widehat{\rho}}(1\rtimes_{\widehat{\rho}}\tau)
\otimes 1^0 )(\widehat{\widehat{\rho}}\otimes\id_{H^0 })(V) \\
& \times((\id\otimes\Delta^0 )\circ\widehat{\widehat{\rho}})(1\rtimes_{\widehat{\rho}}\tau)(\id\otimes\Delta^0 )(V^* )(V_I \otimes 1^0 \otimes 1^0 ) \\
& =\sum_I (V_I^* \otimes 1^0 \otimes 1^0 )((1\rtimes_{\widehat{\rho}}\tau)\otimes 1^0 \otimes 1^0 )(V\otimes 1^0 )
(\widehat{\widehat{\rho}}\otimes\id_{H^0})(V) \\
& \times (\id\otimes\Delta^0 )(V^* )((1\rtimes_{\widehat{\rho}}\tau)\otimes 1^0 \otimes 1^0 )
(V_I \otimes 1^0 \otimes 1^0 ) \\
& =\sum_I (V_I^* \otimes 1^0 \otimes 1^0 )(V\otimes 1^0 )(\widehat{\widehat{\rho}}\otimes\id_{H^0 })(V)
(\id\otimes\Delta^0 )(V^* )(V_I \otimes 1^0 \otimes 1^0 ) .
\end{align*}
Thus for any $h,l\in H$,
\begin{align*}
& [(U\otimes 1^0 )(\widehat{\rho}\otimes\id_{H^0})(U)(\id\otimes\Delta^0 )(U^* )]^{\widehat{}}(h, l) \\
& =\sum_I V_I^* [(V\otimes 1^0 )(\widehat{\widehat{\rho}}\otimes\id_{H^0})(V)(\id\otimes\Delta^0 )(V^* )]^{\widehat{}}
(h, l)V_I  .
\end{align*}
Here for any $h,l\in H$
\begin{align*}
& [(V\otimes 1^0 )(\widehat{\widehat{\rho}}\otimes\id_{H^0})(V)(\id\otimes\Delta^0 )(V^* )]^{\widehat{}}(h, l)
=\widehat{V}(h_{(1)})[h_{(2)}\cdot_{\widehat{\widehat{\rho}}}\widehat{V}(l_{(1)})]\widehat{V^* }(h_{(3)}l_{(2)}) \\
& =\widehat{V}(h_{(1)})[h_{(2)}\cdot_{\widehat{\widehat{\rho}}}(1\rtimes_{\rho ,u}l_{(1)}\rtimes_{\widehat{\rho}}1^0 )]
\widehat{V^* }(h_{(3)}l_{(2)}) \\
& =\widehat{V}(h_{1)})(1\rtimes_{\rho, u}l_{(1)}\rtimes_{\widehat{\rho}}1^0 )\widehat{V^* }(h_{(2)}l_{(2)}) \\
& =\widehat{V}(h_{(1)})\widehat{V}(l_{(1)})\widehat{V^* }(h_{(2)}l_{(2)})=\widehat{u}(h, l)
\end{align*}
by \cite [Lemma 3.12]{KT1:inclusion}. Thus
$$
(V\otimes 1^0 )(\widehat{\widehat{\rho}}\otimes\id_{H^0 })(V)(\id\otimes\Delta^0 )(V^* )=u .
$$
Therefore
$$
(\Psi\otimes\id_{H^0}\otimes\id_{H^0})(u\otimes I_N )
=(U\otimes 1^0 )(\widehat{\widehat{\rho}}\otimes\id_{H^0})(U)(\id\otimes\Delta^0 )(U^* ) .
$$
\end{proof}

\section{Approximately representable coactions}\label{sec:approximate}
For a unital $C^*$-algebra $A$, we set
\begin{align*}
c_0 (A) & =\{(a_n )\in l^{\infty}(\BN, A) \, | \, \lim_{n\to\infty}||a_n ||=0 \} , \\
A^{\infty} & =l^{\infty}(\BN , A)/ c_0 (A) .
\end{align*}
We denote an element in $A^{\infty}$ by the same symbol $(a_n )$ in $l^{\infty}(\BN, A)$.
We identify $A$ with the $C^*$-subalgebra of $A^{\infty}$ consisting of the equivalence classes
of constant sequences and set
$$
A_{\infty}=A^{\infty}\cap A' .
$$
For a weak coaction of $H^0$ on $A$, let $\rho^{\infty}$ be the weak coaction of $H^0$ on $A^{\infty}$ defined by
$\rho^{\infty}((a_n ))=(\rho(a_n ))$ for any $(a_n )\in A^{\infty}$. Hence for a twisted coaction $(\rho, u)$
of $H^0$ on $A$, we can define the twisted coaction $(\rho^{\infty}, u)$ of $H^0$ on $A^{\infty}$.
We have the following easy lemmas.

\begin{lemma}\label{lem:product}Let $(\rho, u)$ be a twisted coaction of $H^0$ on $A$ and $(\rho^{\infty}, u)$
the twisted coaction of $H^0$ on $A^{\infty}$ induced by $(\rho, u)$. Then
$$
A^{\infty}\rtimes_{\rho^{\infty}, u}H\cong (A\rtimes_{\rho, u}H)^{\infty}
$$
as $C^*$-algebras.
\end{lemma}
\begin{proof}Let $\Phi$ be a map from $A^{\infty}\rtimes_{\rho^{\infty} , u}H$ to $(A\rtimes_{\rho, u}H)^{\infty}$
defined by $\Phi((a_n )\rtimes h)=(a_n \rtimes h)$ for any $(a_n )\in A^{\infty}$ and
$h\in H$. For any $(a_n ), (b_n )\in A^{\infty}$ with $(a_n )=(b_n )$ in $A^{\infty}$,
$$
||a_n \rtimes h-b_n \rtimes h ||\leq ||a_n - b_n ||||h||\to 0 \quad (n\to \infty) .
$$
Hence $\Phi$ is well-defined. Also, clearly $\Phi$ is linear. 
For $x\in A^{\infty}\rtimes_{\rho^{\infty}, u}H$, we suppose that
$\Phi(x)=0$. Then we can write that $x=\sum_i (x_{n i})\rtimes h_i $,
where $x_{ni}\in A$ and $\{h_i \}$ is a basis of $H$ such that $\tau(h_i h_j )=0$
and $\delta_{ij}$ is the Kronecker delta. Since $\Phi(x)=0$, $||\sum_i a_{ni}\rtimes_{\rho, u}h_i ||\to 0$
as $n\to\infty$. Hence
$$
||(\sum_i a_{ni}\rtimes_{\rho, u}h_i )(\sum_j a_{nj}\rtimes_{\rho, u}h_j )^* ||\to 0 \quad (n\to\infty) .
$$
Also, by the proof of \cite [Lemma 3.14]{KT1:inclusion}
$$
E_1^{\rho}((\sum_i x_{ni}\rtimes_{\rho, u}h_i )(\sum_j x_{nj}\rtimes_{\rho, u}h_j )^* )
=\sum_i x_{ni}x_{ni}^* .
$$
Thus $||\sum_i x_{ni}x_{ni}^* ||\to 0$ as $n\to\infty$. Hence for any $i$, $x_{ni}\to 0$ as
$n\to 0$. That is, $x=0$.
Thus $\Phi$ is injective. For any $x\in (A\rtimes_{\rho, u}H)^{\infty}$, we write
$x=(x_n )$, $x_n=\sum_{i,j,k}x_{n i j}^k \rtimes w_{ij}^k$, where $x_{nij}^k\in A$. Then
$y=\sum_{i,j,k}(x_{nij}^k )\rtimes w_{ij}^k$ is an element in $A^{\infty}\rtimes_{\rho^{\infty}, u}H$ and
$\Phi(y)=x$. Hence $\Phi$ is surjective. Furthermore, by routine computations, we see that $\Phi$ is a
homomorphism of $A^{\infty}\rtimes_{\rho^{\infty} , u}H$ to $(A\rtimes_{\rho, u}H)^{\infty}$.
Therefore, we obtain the conclusion.
\end{proof}

By the isomorphism defined in the above lemma, we identify $A^{\infty}\rtimes_{\rho^{\infty}, u}H$  with
$(A\rtimes_{\rho, u}H)^{\infty}$. Thus $\widehat{(\rho^{\infty})}=(\widehat{\rho})^{\infty}$.
We denote them by $\widehat{\rho}^{\infty}$.

\begin{lemma}\label{lem:subalgebra}Let $\rho$ be a coaction of $H^0$ on $A$ and $\rho^{\infty}$
the coaction of $H^0$ on $A^{\infty}$ induced by $\rho$. Then $(A^{\infty})^{\rho^{\infty}}=(A^{\rho})^{\infty}$.
\end{lemma}
\begin{proof}It is clear that $(A^{\rho})^{\infty}\subset (A^{\infty})^{\rho^{\infty}}$. We shall show that $(A^{\rho})^{\infty}
\supset (A^{\infty})^{\rho^{\infty}}$. Let $E^{\rho}$ and $(E)^{\rho^{\infty}}$ be the canonical conditional expectations from
$A$ and $A^{\infty}$ onto $A^{\rho}$ and $(A^{\infty})^{\rho^{\infty}}$, respectively.
Then $(A^{\infty})^{\rho^{\infty}}=(E)^{\rho^{\infty}}(A^{\infty})$ and $A^{\rho}=E^{\rho}(A)$.
Let $(a_n )_n \in (A^{\infty})^{\rho^{\infty}}$. We note that
\begin{align*}
(a_n )_n = & (E)^{\rho^{\infty}}((a_n )_n )=e\cdot_{\rho^{\infty}} (a_n )_n =(\id\otimes e)(\rho^{\infty}((a_n )_n ))
= (\id\otimes e)((\rho(a_n ))_n ) \\
& =((\id\otimes e)(\rho(a_n )))_n =(E^{\rho}(a_n ))_n .
\end{align*}
Hence $||E^{\rho}(a_n )-a_n ||\to 0$ $(n\to\infty)$.
Let $b_n =E^{\rho}(a_n )$  for any $n\in \BN$. Since $b_n \in A^{\rho}$, $(b_n )\in (A^{\rho})^{\infty}$.
Then $||b_n -a_n ||=||E^{\rho}(a_n )-a_n ||\to 0 $ $(n\to \infty)$. Thus $(b_n )=(a_n )$
in $(A^{\rho})^{\infty}$. Therefore, $(a_n )\in (A^{\rho})^{\infty}$.
\end{proof}

Since $(A^{\infty})^{\rho^{\infty}}=(A^{\rho})^{\infty}$ by the above lemma,
we can identify $(E)^{\rho^{\infty}}$ with $(E^{\rho})^{\infty}$ the conditional expectation from
$A^{\infty}$ onto $(A^{\rho})^{\infty}$. We denote them by $E^{\rho^{\infty}}$.

\begin{Def}\label{Def:representable}Let $(\rho, u)$ be a twisted coaction of $H$ on $A$.  We say that
$(\rho, u)$ is
\sl
approximately representable
\rm
if there is a unitary element $w\in A^{\infty}\otimes H$ satisfying the following conditions:
\newline
(1) $\rho(a)=(\Ad(w)\circ\rho^A_{H} )(a)$ for any $a\in A$, 
\newline
(2) $u= (w\otimes 1)(\rho^{A^{\infty}}_{H} \otimes \id)(w)(\id\otimes\Delta)(w^* )$,
\newline
(3) $u =(\rho^{\infty}\otimes\id)(w)(w\otimes 1)(\id\otimes\Delta)(w^* )$.
\end{Def}

\begin{lemma}\label{lem:exterior}For $i=1,2$, let $(\rho_i, u_i )$ be a twisted coaction of $H$ on $A$.
We suppose that $(\rho_1 , u_1 )$ is exterior equivalent to $(\rho_2 , u_2 )$. Then $(\rho_1 , u_1 )$
is approximately representable if and only if $(\rho_2 , u_2 )$ is approximately representable.
\end{lemma}

\begin{proof}Since $(\rho_1 , u_1 )$ and $(\rho_2 , u_2 )$ are exterior equivalent,
there is a unitary element $v\in A\otimes H$ satisfying Conditions (1), (2) in Definition \ref{Def:equivalence}.
We suppose that $(\rho_1 , u_1 )$ is approximately representable. Then there is a unitary element
$w_1 \in A^{\infty}\otimes H$ satisfying Conditions (1)-(3) in
Definition \ref{Def:representable} for $(\rho_1 , u_1 )$.
Let $w_2 =vw_1 $. Then by routine computations, we can see that $w_2 $ is a unitary
element in $A^{\infty}\otimes H$ satisfying Conditions (1)-(3) in Definition \ref{Def:representable}
for $(\rho_2 , u_2 )$.
Therefore, we obtain the conclusion.
\end{proof}

\begin{lemma}\label{lem:amplification}Let $(\rho, u)$ be a twisted coaction of $H$ on $A$ and let
$(\rho\otimes\id, u\otimes I_n )$ be the twisted coaction of $H$ on $A\otimes M_n(\BC)$
induced by $(\rho, u)$, where we identify $A\otimes M_n (\BC)\otimes H$ with
$A\otimes H\otimes M_n (\BC )$. Then $(\rho, u)$ is approximately representable
if and only if $(\rho\otimes\id, u\otimes I_n )$ is approximately representable.
\end{lemma}

\begin{proof}We suppose that $(\rho, u)$ is approximately representable.
Then there is a unitary element $w\in A^{\infty}\otimes H$ satisfying
Conditions (1)-(3) in Definition \ref{Def:representable} for $(\rho, u)$. Let $W=w\otimes I_n $.
By routine computations, we can see that $W$ satisfies Conditions (1)-(3) in Definition
\ref{Def:representable} for $(\rho\otimes\id, u\otimes I_n )$. Next, we suppose that
$(\rho\otimes\id, u\otimes I_n )$ is approximately representable. Then there is a unitary
element $W\in A\otimes M_n (\BC)\otimes H$ satisfying Conditions (1)-(3) in
Definition \ref{Def:representable} for $(\rho\otimes\id, u\otimes I_n )$.
Let $f$ be a minimal projection in $M_n (\BC)$ and let $p_0 =1_A \otimes f \otimes 1_H $.
Let $w=p_0 Wp_0 $. Since $\rho\otimes\id_{M_n (\BC)}=\Ad(W)\circ\rho_H^{A\otimes M_n (\BC)}$
on $A\otimes M_n (\BC)$, $Wp_0 =p_0 W$. By routine computations and identifying
$A\otimes M_n (\BC)\otimes H$ with $A\otimes H\otimes M_n (\BC)$, we can
see that the element $w$ satisfies Conditions (1)-(3) in Definition \ref{Def:representable} for
$(\rho, u)$. Therefore, we obtain the conclusion.
\end{proof}

\begin{prop}\label{prop:property}Let $(\rho, u)$ be a twisted coaction of $H$ on $A$. Then
$(\rho, u)$ is approximately representable if and only if so is $\widehat{\widehat{\rho}}$.
\end{prop}
\begin{proof}This is immediate by  Theorem \ref{thm:duality} and Lemmas \ref{lem:exterior}, \ref{lem:amplification}.
\end{proof}

In the rest of this section, we shall show that the approximate representability of coactions of
finite dimensional $C^*$-Hopf algebras is
an extension of the approximate representability of actions of
finite groups in the sense of Izumi  \cite [Remark 3.7]{Izumi:group}.
\par
Let $G$ be a finite group with the order $n$ and $\alpha$ an action of $G$ on $A$.
We consider the coaction of $C(G)$ on $A$ induced by the action $\alpha$ of $G$ on $A$.
We denote it by the same symbol $\alpha$. That is,
$$
\alpha: A\longrightarrow A\otimes C(G), \quad a\longmapsto \sum_{t\in G}\alpha_t (a)\otimes\delta_t
$$
for any $a\in A$, where for any $t\in G$, $\delta_t $ is a projection in $C(G)$ defined by
$$
\delta_t (s)=\begin{cases} 0 & \text{if $s\ne t$} \\
1 & \text{if $s=t$} \end{cases}.
$$
 
 \begin{prop}\label{prop:extension1}With the above notations, the following conditions are
 equivalent:
 \newline
 $(1)$ the action $\alpha$ of $G$ on $A$ is approximately representable,
 \newline
 $(2)$ the coaction $\alpha$ of $C(G)$ on $A$ is approximately representable.
 \end{prop}
\begin{proof}We suppose Condition (1). Then there is a unitary representation $u$ of $G$
in $A^{\infty}$ such that
\begin{align*}
\alpha_t (a) & =u(t)au(t)^* \quad a\in A, \, t\in G, \\
\alpha_t^{\infty}(u(s)) & =u(tst^{-1}) \quad s,t\in G ,
\end{align*}
where $\alpha^{\infty}$ is the automorphism of $A^{\infty}$ induced by
$\alpha$.
Let $w$ be a unitary element in $A^{\infty}\otimes C(G)$ defined by
$w=\sum_{t\in G}u(t)\otimes\delta_t $.
Since $u$ is a unitary representation of $G$ in $A^{\infty}$, we obtain
Condition (2) in Definition \ref{Def:representable} for the coaction $\alpha$. Also, by the above two conditions,
we obtain Conditions (1) and (3) in Definition \ref{Def:representable} for the coaction $\alpha$. Next we suppose
Condition (2). Then there is a unitary element $w\in A^{\infty}\otimes C(G)$ satisfying
Condition (1)-(3) in Definition \ref{Def:representable} for the coaction $\alpha$.
We can regard $A^{\infty}\otimes C(G)$ as the $C^*$-algebra of all $A^{\infty}$-valued
functions on $G$. Hence there is the function from $G$ to $A^{\infty}$ corresponding to
$w$. We denote it by $u$. Since $w$ is a unitary element in $A^{\infty}\otimes C(G)$,
$u(t)$ is a unitary element in $A^{\infty}$ for any $t\in G$. By easy computations, Condition (2) in
Definition \ref{Def:representable} for the coaction $\alpha$ implies that $u$ is a unitary representation of $G$ in $A^{\infty}$. Also,
Conditions (1) and (3) in Definition \ref{Def:representable} for the coaction $\alpha$ imply that
\begin{align*}
\alpha_t (a) & =u(t)au(t)^* \quad a\in A, \, t\in G, \\
\alpha_t^{\infty}(u(s)) & =u(tst^{-1}) \quad s,t\in G .
\end{align*}
Therefore, we obtain the conclusion.
\end{proof}

\section{Coactions with the Rohlin property}\label{sec:Rohlin}
In this section, we shall introduce the Rohlin property for coactions of a finite dimensional
$C^*$-Hopf algebra on a unital $C^*$-algebra.

\begin{Def}\label{Def:Rohlin}Let $(\rho, u)$ be a twisted coaction of $H^0 $ on $A$.
We say that $(\rho, u)$ has the
\sl
Rohlin
\rm
property if the dual coaction $\widehat{\rho}$ of $H$ on $A\rtimes_{\rho, u} H $
is approximately representable.
\end{Def}

First, we shall begin with the following easy propositions.

\begin{prop}\label{prop:saturated}Let $\rho$ be a coaction of $H^0 $ on $A$ with
the Rohlin property. Then $\rho$ is saturated.
\end{prop}

\begin{proof} This is immediate by Corollary \ref{cor:saturated}.
\end{proof}

\begin{prop}\label{prop:property2}Let $(\rho, u)$ be a twisted coaction of $H^0 $ on $A$.
Then $(\rho, u)$ has the Rohlin property if and only if so does $\widehat{\widehat{\rho}}$.
\end{prop}
\begin{proof} This is immediate by Proposition \ref{prop:property}.
\end{proof}

Let $(\rho, u)$ be a twisted coaction of $H^0$ on $A$ with the Rohlin property.
Then there is a unitary element $w\in (A^{\infty}\rtimes_{\rho^{\infty}, u}H)\otimes H$
satisfying that
\newline
(5, 1) $\widehat{\rho}(x) =(\Ad(w)\circ\rho_H^{A\rtimes_{\rho, u}H})(x)$ for any $x\in A\rtimes_{\rho, u}H$,
\newline
(5, 2) $(w\otimes 1) (\rho_H^{A^{\infty}\rtimes_{\rho^{\infty}, u}H}\otimes\id_H )(w)
=(\id_{A^{\infty}\rtimes_{\rho^{\infty}, u}H}\otimes\Delta)(w)$,
\newline
(5, 3) $(\widehat{\rho}^{\infty}\otimes\id_H )(w)(w\otimes 1)=
(\id_{A^{\infty}\rtimes_{\rho^{\infty}, u}H}\otimes\Delta)(w)$.

\vskip 0.3cm
Let $\widehat{w}$ be the element in $\Hom (H^0 , A^{\infty}\rtimes_{\rho^{\infty}, u}H)$
induced by $w$.

\begin{lemma}\label{lem:homo}With the above notations, $\widehat{w}$ is a homomorphism
of $H^0 $ to $(A^{\infty}\rtimes_{\rho^{\infty}, u}H)\cap A'$ satisfying the following conditions:
\newline
$(1)$ $\widehat{w}(1^0 )=1_{A^{\infty}}$,
\newline
$(2)$ the element $\widehat{w}(\tau)$ is a projection in $A_{\infty}$,
\newline
$(3)$ $\widehat{w}(\tau)x\widehat{w}(\tau)=E_1^{\rho}(x)\widehat{w}(\tau)$ for any $x\in A\rtimes_{\rho, u}H$.
\end{lemma}

\begin{proof} By Equation (5, 2), $\widehat{w}\in \Alg(H^0 , A^{\infty}\rtimes_{\rho^{\infty}}H)$.
Furthermore, by \cite [Lemma 1.16]{BCM:crossed}, $\widehat{w}^* =\widehat{w}\circ S^0 $.
Thus for any $\phi\in H^0 $, 
$\widehat{w}(\phi)^* =\widehat{w}^* (S^0 (\phi^* )) =\widehat{w}(\phi^* )$.
Hence $\widehat{w}$ is a homomorphism of $H^0 $ to $A^{\infty}\rtimes_{\rho^{\infty}}H$.
Next we shall show that $\widehat{w}(\phi)(a\rtimes 1)=(a\rtimes 1)\widehat{w}(\phi)$
for any $a\in A$. By Equation (5, 1), for any $a\in A$, 
$$
(a\rtimes 1)\otimes 1=w[(a\rtimes 1)\otimes 1]w^* .
$$
Thus $[(a\rtimes 1)\otimes 1]w=w[(a\rtimes 1)\otimes 1]$. Hence for any $\phi\in H^0 $
$$
(a\rtimes 1)\widehat{w}(\phi)=\widehat{w}(\phi)(a\rtimes 1) .
$$
Hence $\widehat{w}$ is a homomorphism
of $H^0 $ to $(A^{\infty}\rtimes_{\rho^{\infty}, u}H)\cap A'$. Also, by Equation (5, 2),
\begin{align*}
(\id_{A^{\infty}\rtimes_{\rho^{\infty}, u}H}\otimes\epsilon\otimes\id_H ) &
((w\otimes 1)(\rho_H^{A^{\infty}\rtimes_{\rho^{\infty}, u}H}\otimes\id_H )(w)) \\
& =(\id_{A^{\infty}\rtimes_{\rho^{\infty}, u}H}\otimes\epsilon\otimes\id_H )
((\id_{A^{\infty}\rtimes_{\rho^{\infty}, u}H}\otimes\Delta)(w)) .
\end{align*}
Thus $[(\id_{A^{\infty}\rtimes_{\rho^{\infty}, u}H}\otimes\epsilon)(w)\otimes 1]w=w$. Since $w$ is a unitary
element in $(A^{\infty}\rtimes_{\rho^{\infty}, u}H)\otimes H$,
$(\id_{A^{\infty}\rtimes_{\rho^{\infty}, u}H}\otimes\epsilon)(w)=1 $, that is, $\widehat{w}(1^0 )=1$.
Furthermore, since $\tau$ is a projection in $H^0 $ and $\widehat{w}$ is a homomorphism of
$H^0 $ to $A^{\infty}\rtimes_{\rho^{\infty}}H$, $\widehat{w}(\tau)$ is a projection. Also,
by Equation (5, 3), for any $\phi\in H^0 $
\begin{align*}
\phi\cdot_{\widehat{\rho}^{\infty}}\widehat{w}(\tau) & =\widehat{w}(\phi_{(1)}\tau)\widehat{w}^* (\phi_{(2)})
=\widehat{w}(\tau)\widehat{w}^* (\phi)=\widehat{w}(\tau)\widehat{w}(S^0(\phi)^* )^*
=\widehat{w}(\tau)\widehat{w}(S^0 (\phi)) \\
& =\epsilon^0 (\phi)\widehat{w}(\tau) .
\end{align*}
Hence by \cite [Lemma 3.17]{KT1:inclusion}, $\widehat{w}(\tau)\in A^{\infty}\cap A' =A_{\infty}$.
Finally, we shall show that $\widehat{w}(\tau)x\widehat{w}(\tau)
=E_1^{\rho}(x)\widehat{w}(\tau)$ for any $x\in A\rtimes_{\rho, u}H$.
For any $a\in A$, $h\in H$, $\widehat{\rho}(a\rtimes h)=w[(a\rtimes h)\otimes 1]w^* $.
Thus
$$
(a\rtimes h_{(1)})\tau(h_{(2)} )=\widehat{w}(\tau_{(1)})(a\rtimes h)\widehat{w}^* (\tau_{(2)}) .
$$
That is, $\tau(h)(a\rtimes 1)=\widehat{w}(\tau_{(1)})(a\rtimes h)\widehat{w}^* (\tau_{(2)})$.
Since $E_1^{\rho}(a\rtimes h)=\tau(h)(a\rtimes 1)$ and $\widehat{w^* }=\widehat{w}\circ S^0 $,
\begin{align*}
E_1^{\rho}(a\rtimes h)\widehat{w}(\tau) & =\tau(h)(a\rtimes 1)\widehat{w}(\tau)
=\widehat{w}(\tau_{(1)})(a\rtimes h)\widehat{w}^* (\tau_{(2)})\widehat{w} (\tau) \\
& =\widehat{w}(\tau_{(1)})(a\rtimes h)\widehat{w} (\tau)\epsilon^0 (\tau_{(2)})
=\widehat{w}(\tau)(a\rtimes h)\widehat{w} (\tau) .
\end{align*}
Thus we obtain the last condition.
\end{proof}

\begin{prop}\label{prop:exterior2}For $i=1,2$, let $(\rho_i , u_i )$ be a twisted coaction of
$H^0 $ on $A$ with $(\rho_1 , u_1 )\sim(\rho_2 , u_2 )$. Then $(\rho_1 , u_1 )$ has the Rohlin
property if and only if so does $(\rho_2 , u_2 )$.
\end{prop}
\begin{proof}Since $(\rho_1 , u_1 )\sim (\rho_2 , u_2 )$, there is a unitary element
$v\in A\otimes H^0 $ satisfying that
$$
\rho_2 =\Ad (v)\circ\rho_1 \quad u_2 =(v\otimes 1^0 )(\rho_1 \otimes\id)(v)u_1 (\id\otimes \Delta^0 )(v^* ) .
$$
Then there is an isomorphism $\Phi$ of $A\rtimes_{\rho_1 , u_1}H$ onto $A\rtimes_{\rho_2 , u_2 }H$
defined by
$$
\Phi(a\rtimes_{\rho_1 , u_1 }h )=a\widehat{v}^* (h_{(1)})\rtimes_{\rho_2 , u_2 }h_{(2)} .
$$
By easy computations, we can see that the following conditions hold:
\newline
(1) $\widehat{\rho_2 }\circ\Phi=(\Phi\otimes\id_H )\circ\widehat{\rho_1 } $,
\newline
(2) $\rho_H^{A\rtimes_{\rho_2 , u_2 }H}\circ\Phi=(\Phi\otimes\id_H )\circ\rho_H^{A\rtimes_{\rho_1 , u_1 }H}$,
\newline
(3) $(\id_{A\rtimes_{\rho_2 ,u_2}H}\otimes\Delta)\circ(\Phi\otimes\id_H )
=(\Phi\otimes\id_H  \otimes\id_H )\circ(\id_{A\rtimes_{\rho_1 , u_1 }H}\otimes\Delta)$
\vskip 0.3cm
Let $\Phi^{\infty}$ be the isomorphism of $A^{\infty}\rtimes_{\rho_1 , u_1 }H$ onto
$A^{\infty}\rtimes_{\rho_2 , u_2 }H$ induced by $\Phi$. We suppose that $(\rho_1 ,u_1 )$ has
the Rohlin property and let $w_1 $ be a unitary element in $(A\rtimes_{\rho_1 ,u_1 }H)\otimes H$ satisfying
Equations (5, 1)-(5, 3) for the coaction $\widehat{\rho_1}$.
Let $w_2 =(\Phi^{\infty}\otimes\id_H )(w_1 )$. By Conditions (1)-(3),
we can see that $w_2 $ satisfies Equations (5, 1)-(5, 3) for the coaction $\widehat{\rho_2 }$.
Therefore we obtain the conclusion.
\end{proof}

\begin{lemma}\label{lem:projection2}For $i=1,2$, let $(\rho_i , u_i )$ be a twisted coaction of
$H^0 $ on $A$ with $(\rho_1 , u_1 )\sim(\rho_2 , u_2 )$. We suppose that $(\rho_i , u_i )$ has
the Rohlin property for $i=1,2$. Let $w_i $ be as in the above proof for $i=1,2$.
Then $\widehat{w_1 }(\tau)=\widehat{w_2}(\tau)$.
\end{lemma}
\begin{proof}Let $w_1 =\sum_{i, j}(a_{ij}\rtimes_{\rho_1, u_1 }h_i )\otimes l_j$,
where $a_{ij}\in A^{\infty}$.
Then
$$
w_2 =\sum_{i,j}(a_{ij}\widehat{v}^* (h_{i(1)})\rtimes_{\rho_2 , u_2 }h_{i(2)})\otimes l_j ,
$$
where $v$ is a unitary element in $A\otimes H^0 $ defined in the above proof.
Thus
\begin{align*}
\widehat{w_2 }(\tau) & =\sum_{i,j}(a_{ij}\widehat{v}^* (h_{i(1)})\rtimes_{\rho_2 , u_2 }h_{i(2)}) \tau(l_j ) \\
& =\Phi(\sum_{i, j}(a_{ij}\rtimes_{\rho_1, u_1 }h_i )\tau(l_j ))=\Phi(\widehat{w_1 }(\tau)) ,
\end{align*}
where $\Phi$ is the isomorphism of $A\rtimes_{\rho_1 , u_1 }H$ onto $A\rtimes_{\rho_2 , u_2 }H$
defined in the above proof. On the other hand, since $\widehat{w_1 }(\tau) \in A_{\infty}\subset A^{\infty}$
by Lemma \ref{lem:homo}, $\widehat{w_2 }(\tau)=\Phi(\widehat{w_1 }(\tau))=\widehat{w_1 }(\tau)$.
\end{proof}

Let $(\rho, u)$ be a twisted coaction of $H$ on $A$ with
the Rohlin property. Let $w$ be a unitary
element in $(A^{\infty}\rtimes_{\rho^{\infty}, u}H)\otimes H$ satisfying Equations (5, 1)-(5, 3) for $\widehat{\rho}$.

\begin{lemma}\label{lem:projection}With the above notations,
$e\cdot\widehat{w}(\tau)=\frac{1}{N}$.
\end{lemma}

\begin{proof}We note that $\widehat{\rho}(1\rtimes h)=w[(1\rtimes h)\otimes 1]w^* $
for any $h\in H$. Since $\widehat{w^* } =\widehat{w}\circ S^0 $, we see that for any $h\in H$,
$(1\rtimes h)\widehat{w}(S^0 (\tau))=\widehat{w}(S^0 (\tau_{(1)}))(1\rtimes h_{(1)})\tau_{(2)}(h_{(2)})$.
Hence for any $h\in H$,
$\widehat{V}(h)\widehat{w}(\tau)=\widehat{w}(S^0 (\tau_{(1)}))\widehat{V}(h_{(1)})\tau_{(2)}(h_{(2)})$.
Then
\begin{align*}
& e\cdot\widehat{w}(\tau) =\sum_{i,k}\frac{d_k}{N}[w_{ii}^k \cdot\widehat{w}(\tau)]
=\sum_{i,j,k}\frac{d_k }{N}\widehat{V}(w_{ij}^k )(\widehat{w}(\tau)\rtimes 1)\widehat{V^* }(w_{ji}^k ) \\
& =\sum_{i,j,k,j_1 }\frac{d_k }{N}\widehat{w}(S^0 (\tau_{(1)}))\widehat{V}(w_{ij_1 }^k )
\tau_{(2)}(w_{j_1 j}^k )\widehat{V^* }(w_{ji}^k ) \\
& =\sum_{i,j,k,j_1 }\frac{d_k }{N}\widehat{w}(S^0 (\tau_{(1)}))(1\rtimes w_{ij_1 }^k )(1\rtimes w_{ij}^k )^* 
\tau_{(2)}(w_{j_1 j}^k ) \\
& =\sum_{i,j,k,j_1 , j_2, j_3 }\frac{d_k }{N}\widehat{w}(S^0 (\tau_{(1)}))
\tau_{(2)}(w_{j_1 j}^k )(1\rtimes w_{ij_1 }^k )
(\widehat{u}(S(w_{j_2 j_3}^k ), w_{ij_2}^k )^* \rtimes w_{j_3 j}^{k*}) \\
& =\sum_{i,j,k,j_1 , j_2, j_3 , j_4 , j_5 , s}\frac{d_k }{N}\widehat{w}(S^0 (\tau_{(1)}))\tau_{(2)}(w_{j_1 j}^k )
[w_{ij_4 }^k \cdot\widehat{u}(S(w_{j_2 j_3}^k ), w_{ij_2}^k )^* ]\widehat{u}(w_{j_4 j_5}^k , w_{j_3 s}^{k*}) \\
& \rtimes w_{j_5 j_1}^k w_{sj}^{k*}
\end{align*}
since $w_{ij}^{k*} =S(w_{ji}^k )$ for any $i,j,k$ by \cite [Theorem 2.2 2]{SP:saturated}.
Since $e\cdot \widehat{w}(\tau)\in A^{\infty}$,
$E_1^{\rho^{\infty}}(e\cdot \widehat{w}(\tau))=e\cdot \widehat{w}(\tau)$.
Thus since $\tau(w_{ij}^k w_{st}^{r*})=\frac{1}{d_k }\delta_{kr}\delta_{is}\delta_{jt}$ 
by  \cite [Theorem 2.2, 2]{SP:saturated},
\begin{align*}
& e\cdot \widehat{w}(\tau) =\sum_{i,j,k, j_2, j_3 , j_4 , s}\frac{1 }{N}
\widehat{w}(S^0 (\tau_{(1)}))\tau_{(2)}(w_{jj}^k )
[w_{ij_4 }\cdot\widehat{u}(S(w_{j_2 j_3}^k ), w_{ij_2}^k )^* ] \\
& \times\widehat{u}(w_{j_4 s}^k , w_{j_3 s}^{k*}) \\
& =\sum_{i,j,k, j_2, j_3 , j_4 , s}\frac{1 }{N}\widehat{w}(S^0 (\tau_{(1)}))\tau_{(2)}(w_{jj}^k )
[w_{j_4 i}\cdot\widehat{u}(S(w_{j_2 j_3}^k ), w_{ij_2}^k )]^* \widehat{u}(w_{j_4 s}^k , w_{j_3 s}^{k*}) .
\end{align*}
Furthermore by \cite [Lemma 3.3(1)]{KT1:inclusion},
\begin{align*}
& e\cdot \widehat{w}(\tau)
=\sum_{i,j,k, j_2, j_3 , j_4 , s,t,r}\frac{1}{N}\widehat{w}(S^0 (\tau_{(1)}))\tau_{(2)}(w_{jj}^k )
(\widehat{u}(w_{j_4 t}^k , S(w_{rj_3 }^k )) \\
& \times\widehat{u}(w_{ti}^kS(w_{j_2 r}^k ), w_{ij_2 }^k ))^* \widehat{u}(w_{j_4 s}^k , w_{j_3 s}^{k*}) \\
& =\sum_{i,j,k, j_2, j_3 , j_4 , s, t, r}\frac{1}{N}\widehat{w}(S^0 (\tau_{(1)}))\tau_{(2)}(w_{jj}^k )
\widehat{u^* }(S(w_{ti}^{k*})w_{j_2 r}^{k*}, S(w_{ij_2 }^{k*})) \\
& \times\widehat{u^* }(S(w_{j_4 t}^{k *}), w_{rj_3 }^{k*} ) \widehat{u}(w_{j_4 s}^k , w_{j_3 s}^{k*}) \\
& =\sum_{i,j,k, j_2, j_3 , j_4 , s, t, r}\frac{1}{N}\widehat{w}(S^0 (\tau_{(1)}))\tau_{(2)}(w_{jj}^k )
\widehat{u^* }(w_{it}^kS(w_{r j_2 }^k ), w_{j_2 i}^k ) \\
& \times\widehat{u^* }(w_{tj_4 }^k , w_{rj_3 }^{k*} ) \widehat{u}(w_{j_4 s}^k , w_{j_3 s}^{k*}) \\
& =\sum_{i,j,k, j_2, s}\frac{1}{N}\widehat{w}(S^0 (\tau_{(1)}))\tau_{(2)}(w_{jj}^k )
\widehat{u^* }(w_{is}^kS(w_{s j_2 }^k ), w_{j_2 i}^k ) \\
& =\sum_{i,j,k}\frac{1}{N}\widehat{w}(S^0 (\tau_{(1)}))\tau_{(2)}(w_{jj}^k )\epsilon(w_{ii}^k ) .
\end{align*}
Since $e=\sum_{ik}\frac{d_k }{N}w_{ii}^k $ and $\epsilon(w_{ij}^k )=\delta_{ij}$
for any $k$ by \cite [Theorem 2.2, 2]{SP:saturated},
$$
e\cdot \widehat{w}(\tau) = \sum_{j,k}\frac{d_k }{N}\widehat{w}(S^0 (\tau_{(1)}))\tau_{(2)}(w_{jj}^k )
=\widehat{w}(S^0 (\tau_{(1)}))\tau_{(2)}(e)=\frac{1}{N}
$$
by Lemma \ref{lem:homo}(1). Therefore, we obtain the conclusion.
\end{proof}

By Lemmas \ref{lem:homo}(2) and  \ref{lem:projection}, we can can see that if $\rho$ is
a coaction of $H^0 $ on $A$ with the Rohlin property, then there is a projection $p\in A_{\infty}$
such that $e\cdot p=\frac{1}{N}$. We shall show the inverse direction with the assumption that
$\rho$ is saturated. Let $\rho$ be
a saturated coaction of $H^0$ on $A$.
We suppose that there is a projection $p\in A_{\infty}$ such that $e\cdot p=\frac{1}{N}$.

\begin{lemma}\label{lem:covariance}With the above notations and assumptions,
for any $x\in A\rtimes H$, $(p\rtimes 1)x(p\rtimes 1)=E_1^{\rho}(x)(p\rtimes 1)$.
\end{lemma}

\begin{proof}Let $q=N(p\rtimes 1)(1\rtimes e)(p\rtimes 1)$. Then $q$ is a projection
in $A^{\infty}\rtimes_{\rho^{\infty}}H$. Indeed, $q^* =q$. Also,
$q^2 = N^2 (p\rtimes 1)([e\cdot p]\rtimes e)(p\rtimes 1)=q$
by the assumption. Furthermore, $E_1^{\rho^{\infty}}(q)=p=E_1^{\rho^{\infty}}(p\rtimes 1)$.
Since $q\le p$ and $E_1^{\rho^{\infty}}$ is faithful, we obtain that $p=q$. That is,
$p=N(p\rtimes 1)(1\rtimes e)(p\rtimes 1)$. For any $a, b\in A$,
\begin{align*}
(p\rtimes 1)(a\rtimes 1)(1\rtimes e)(b\rtimes 1)(p\rtimes 1)
& =(a\rtimes 1)(p\rtimes 1)(1\rtimes e)(p\rtimes 1)(b\rtimes 1) \\
& =\frac{1}{N}(ab\rtimes 1)(p\rtimes 1) .
\end{align*}
Since $\rho$ is saturated, $A(1\rtimes e)A=A\rtimes_{\rho}H$.
Hence we obtain the conclusion.
\end{proof}

By Watatani \cite [Proposition 2.2.7 and Lemma 2.2.9]{Watatani:index}
and Lemma \ref{lem:covariance}, we can see that
there is a homomorphism $\pi$ of $A\rtimes_{\rho}H\rtimes_{\widehat{\rho}}H^0 $
to $A^{\infty}\rtimes_{\rho^{\infty}}H$ such that
$$
\pi((x\rtimes 1^0 )(1_A \rtimes 1_H \rtimes\tau)(y\rtimes 1^0 ))=x(p\rtimes 1)y
$$
for any $x,y\in A\rtimes_{\rho}H$. The restriction of $\pi$ to $1_{A\rtimes_{\rho}H}\rtimes H^0 $ is a
homomorphism of $H^0 $ to $A^{\infty}\rtimes_{\rho^{\infty}}H$. Thus
there is an element $w\in (A^{\infty}\rtimes_{\rho^{\infty}}H)\otimes H$ such that
$\widehat{w}$ is the above restriction of $\pi$ to $H^0 $.
Let $\{(u_i , u_i^* )\}$ be a quasi-basis of $E_1^{\rho}$.

\begin{lemma}\label{lem:unitary0}With the above notations and assumptions, for any $\phi\in H^0 $,
$\widehat{w}(\phi)=\sum_j [\phi\cdot_{\widehat{\rho}} u_j ](p\rtimes 1)u_j^* $.
\end{lemma}

\begin{proof}We note that $\tau\cdot x=E_1^{\rho}(x)$ for any $x\in A\rtimes_{\rho}H$.
Since $\sum_i (u_i \rtimes 1^0 )(1\rtimes \tau)(u_i^* \rtimes 1^0 )=1$, 
\begin{align*}
1\rtimes\phi & =\sum_i (1\rtimes\phi)(u_i \rtimes 1^0 )(1\rtimes\tau)(u_i^* \rtimes 1^0 )
=\sum_i ([\phi_{(1)}\cdot u_i ]\rtimes\phi_{(2)}\tau)(u_i^* \rtimes1^0 ) \\
& =\sum_i ([\phi\cdot u_i ]\rtimes 1^0 )(1\rtimes\tau)(u_i^* \rtimes1^0 ).
\end{align*}
Hence we obtain the conclusion by the definition of $\widehat{w}$.
\end{proof}

\begin{lemma}\label{lem:unit}With the above notations, $\widehat{w}(1^0 )=1_A$.
\end{lemma}

\begin{proof}By \cite [Proposition 3.18]{KT1:inclusion},
$\{((\sqrt{d_k }\rtimes w_{ij}^k)^* , \, \sqrt{d_k }\rtimes w_{ij}^k ) \}$ is a quasi-basis
of $E_1^{\rho}$. Hence by Lemma \ref{lem:unitary0},
\begin{align*}
\widehat{w}(1^0 ) & =\sum_{i,j,k}d_k (1\rtimes w_{ij}^{k *}) (p\rtimes w_{ij}^k )
=\sum_{i,j,k,t}d_k [w_{it}^{k *}\cdot p]\rtimes w_{tj}^{k *}w_{ij}^k \\
&=\sum_{i,j,k,t}[w_{it}^{k *}\cdot p]\rtimes S(w_{jt}^k )w_{ij}^k
=\sum_{i,k,t}[S(w_{ti}^k )\cdot p]\rtimes\epsilon(w_{it}^k ) \\
& =\sum_{i,k}[S(w_{ii}^k )\cdot p]\rtimes 1
=N[e\cdot p]=1.
\end{align*}
\end{proof}

\begin{lemma}\label{lem:inverse}With the above notations, the element $w$ is a unitary element
in $(A^{\infty}\rtimes_{\rho^{\infty}}H)\otimes H$ satisfying Equations $(5, 1)-(5, 3)$.
\end{lemma}

\begin{proof}Since $\widehat{w}$ is a homomorphism of $H^0$ to $A^{\infty}\rtimes_{\rho^{\infty}}H$,
the element $w$ satisfies Equation (5, 2). Also, for any $\phi\in H^0 $
$$
(\widehat{w}\widehat{w}^* )(\phi) =\widehat{w}(\phi_{(1)})\widehat{w}(S^0 (\phi_{(2)})^* )^*
=\widehat{w}(\phi_{(1)}S^0 (\phi_{(2)}))
=\epsilon^0 (\phi)\widehat{w}(1^0 )=\epsilon^0 (\phi)
$$
by Lemma \ref{lem:unit}. Similarly $(\widehat{w}^* \widehat{w})(\phi) =\epsilon^0 (\phi)$.
Hence $w$ is a unitary element in $(A^{\infty}\rtimes_{\rho^{\infty}}H)\otimes H$. 
Let $\{(u_i , u_i^* )\}$ be a quasi-basis of $E_1^{\rho}$. By Lemmas \ref{lem:covariance} and \ref{lem:unitary0} for any $\phi, \psi\in H^0 $,
\begin{align*}
[\phi_{(1)}\cdot_{\widehat{\rho}}\widehat{w}(\psi)]\widehat{w}(\phi_{(2)}) & =\sum_{i,j}[\phi_{(1)}\cdot([\psi\cdot u_j ]
(p\rtimes 1)u_j^* )][\phi_{(2)}\cdot u_i ](p\rtimes 1)u_i^*  \\
& =\sum_{i,j}[\phi\cdot([\psi\cdot u_j ](p\rtimes 1)u_j^* u_i )](p\rtimes 1)u_i^* \\
& =\sum_{i,j}[\phi\cdot([\psi\cdot u_j ](p\rtimes 1)u_j^* u_i (p\rtimes 1))]u_i^*  \\
& =\sum_{i,j}[\phi\cdot([\psi\cdot u_j ]E_1^{\rho}(u_j^* u_i )(p\rtimes 1))]u_i^* \\
& =\sum_{i,j}[\phi\cdot([\psi\cdot (u_j E_1^{\rho}(u_j^* u_i ))](p\rtimes 1))]u_i^* \\
& =\sum_i [\phi\cdot([\psi\cdot u_i ](p\rtimes 1))]u_i^* 
=\sum_i [(\phi\psi)\cdot u_i ](p\rtimes 1)u_i^* =\widehat{w}(\phi\psi) .
\end{align*}
Thus the element $w$ satisfies Equation (5, 3). Finally, for any $a\in A$, $h\in H$ and $\phi\in H^0 $,
\begin{align*}
\widehat{w}(\phi_{(1)})(a\rtimes h)\widehat{w}^* (\phi_{(2)}) & =\sum_{i,j}
[\phi_{(1)}\cdot u_j ](p\rtimes 1)u_j^* (a\rtimes h)[S^0 (\phi_{(2)})\cdot u_i ](p\rtimes 1)u_i^* \\
& =\sum_{i,j}[\phi_{(1)}\cdot u_j ]E_1^{\rho}(u_j^* (a\rtimes h)[S^0 (\phi_{(2)})\cdot u_i ])(p\rtimes 1)u_i^* \\
& =\sum_{i,j}[\phi_{(1)}\cdot (u_j E_1^{\rho}(u_j^* (a\rtimes h)[S^0 (\phi_{(2)})\cdot u_i ])](p\rtimes 1)u_i^*  \\
& =\sum_i [\phi_{(1)}\cdot ((a\rtimes h)[S^0 (\phi_{(2)})\cdot u_i ])](p\rtimes 1)u_i^* \\
& =\sum_i (a\rtimes h_{(1)})\phi_{(1)}(h_{(2)})[\phi_{(2)}\cdot [S^0 (\phi_{(3)})\cdot u_i ]](p\rtimes 1)u_i^* \\
& =\sum_i (a\rtimes h_{(1)})\phi(h_{(2)})u_i(p\rtimes 1)u_i^* \\
& =(a\rtimes h_{(1)})\phi(h_{(2)})\widehat{w}(1^0 )=(a\rtimes h_{(1)})\phi(h_{(2)})
\end{align*}
by Lemmas \ref{lem:unitary0} and \ref{lem:unit}. Hence $w$ satisfies Equation (5, 1).
Therefore we obtain the conclusion.
\end{proof}

\begin{thm}\label{thm:equivalent}Let $\rho$ be a coaction of a finite dimensional
 $C^*$-Hopf algebra $H$ on a unital $C^*$-algebra $A$. If $\rho$ is saturated, then the following conditions
 are equivalent:
 \newline
 $(1)$ the coaction $\rho$ has the Rohlin property,
 \newline
 $(2)$ there is a projection $p$ in $A_{\infty}$ such that $e\cdot_{\rho^{\infty}} p =\frac{1}{N}$,
 where $N=\dim H$.
 \end{thm}
 
 \begin{proof}This is immediate by Lemmas \ref{lem:projection} and
 \ref{lem:inverse}.
 \end{proof}

In the rest of this section, we shall show that the Rohlin property of coactions of a finite
dimensional $C^*$-Hopf algebra is an extension of the Rohlin property of a finite group in the
sense of \cite [Remark 3.7]{Izumi:group}.
Let $G$ and $\alpha$ be as in the end of Section \ref{sec:approximate}.

\begin{prop}\label{prop:extension2}With the above notations, the following conditions are equivalent:
\newline
$(1)$ the action $\alpha$ of $G$ on $A$ has the Rohlin property,
\newline
$(2)$ the coaction $\alpha$ of $C(G)$ on $A$ has the Rohlin property.
\end{prop}
\begin{proof}We suppose Condition (1). Then there is a partition of unity
$\{e_t \}_{t\in G}$ consisting of projections in $A_{\infty}$ satisfying that
$\alpha_t^{\infty}(e_s )=e_{ts}$ for any $t,s\in G$. By easy computations,
$e_e $ is a projection in $A_{\infty}$ such that $\tau\cdot e_e=\frac{1}{n}$,
where $\tau$ is the Haar trace on $C(G)$. Since the coaction $\alpha$ of $C(G)$
on $A$ is saturated,  by Theorem \ref{thm:equivalent}
the coaction $\alpha$ has the Rohlin property. Next, we suppose
Condition (2). Then there is a projection $p\in A_{\infty}$ such that
$\tau\cdot p=\frac{1}{n}$ by Theorem \ref{thm:equivalent}. Hence
$$
(\id\otimes\tau)(\sum_{t\in G}\alpha_t^{\infty}(p)\otimes\delta_t )=\frac{1}{n} .
$$
Thus, we see that $\sum_{t\in G}\alpha_t^{\infty}(p)=1$ by the definition of $\tau$.
Let $e_t =\alpha_t^{\infty}(p)$ for any $t\in G$. Then clearly,
$\{e_t \}_{t\in G}$ is a partition of unity consisting of projections in $A_{\infty}$ satisfying
that $\alpha_t^{\infty}(e_s)=e_{ts}$.
\end{proof}

\section{Another equivalent condition}\label{sec:another}
In this section, we shall give another condition which is equivalent to the Rohlin property.
\par
Let $(\rho, u)$ be a twisted coaction of $H^0$ on $A$. We suppose that $(\rho, u)$ has the
Rohlin property. Then there is a unitary element $w\in (A^{\infty}\rtimes_{\rho^{\infty}, u}H)\otimes H$
satisfying Equations (5, 1)-(5, 3) for $(\rho, u)$. Let $\widehat{w}$ be the
unitary element in $\Hom (H^0 , A^{\infty}\rtimes_{\rho^{\infty}, u}H)$ induced by $w\in (A\rtimes_{\rho^{\infty}, u}H)
\otimes H$. By  Lemma \ref {lem:homo}, $\widehat{w}(\tau)$ is a projection in $A_{\infty}$.
By Theorem \ref{thm:duality} there are an isomorphism $\Psi$ of $M_N (A)$ onto $A\rtimes_{\rho, u}H\rtimes_{\widehat{\rho}}H^0$
and a unitary element $U$ in $(A\rtimes_{\rho, u}H\rtimes_{\widehat{\rho}}H^0 )\otimes H^0$ such that
\begin{align*}
& \Ad(U)\circ\widehat{\widehat{\rho}}=(\Psi\otimes\id_{H^0 })\circ(\rho\otimes\id_{M_N (\BC)})\circ\Psi^{-1} , \\
& (\Psi\otimes\id_{H^0 }\otimes\id_{H^0 })(u\otimes I_N )=(U\otimes 1^0 )(\widehat{\widehat{\rho}}\otimes \id_{H^0 })(U)
(\id\otimes\Delta^0 )(U^* ) .
\end{align*}
Let  $\sigma=(\Psi\otimes \id)\circ(\rho\otimes\id_{M_N (\BC)})\circ\Psi^{-1}$ and
$W=(\Psi\otimes\id_{H^0 }\otimes\id_{H^0})(u\otimes I_N )$. Then
$(\sigma, W)$ is a twisted coaction of $H^0$ on $A\rtimes_{\rho, u}H\rtimes_{\widehat{\rho}}H^0$
which is exterior equivalent to $\widehat{\widehat{\rho}}$.
Let $\widehat{\Psi}$ be the isomorphism of $M_N (A)\rtimes_{\rho\otimes\id, u\otimes I_N }H$ onto
$A\rtimes_{\rho, u}H\rtimes_{\widehat{\rho}}H^0 \rtimes_{\sigma, W}H$ induced by
$\Psi$, which is defined by $\widehat{\Psi}(x\rtimes_{\rho\otimes\id, u\otimes I_N }h)=\Psi(x)\rtimes_{\sigma ,W}h$ for any
$x\in M_N (A)$, $h\in H$.
Let $\widehat{\Psi}^{\infty}$ be the isomorphism of $M_N (A^{\infty})\rtimes_{\rho^{\infty}\otimes\id, u\otimes I_N}H$ onto
$A^{\infty}\rtimes_{\rho^{\infty}, u}H\rtimes_{\widehat{\rho}^{\infty}}H^0 \rtimes_{\sigma^{\infty}, W}H$ induced by
$\widehat{\Psi}$. By easy computations, $(\sigma, W)$ has the Rohlin property and the unitary
element $(\widehat{\Psi}^{\infty}\otimes\id_{H})(w\otimes I_N )$ is in
$(A\rtimes_{\rho, u}H\rtimes_{\widehat{\rho}}H^0 \rtimes_{\sigma, W}H)\otimes H$ and satisfies Equations (5, 1)-(5, 3)
for the twisted coaction $(\sigma, W)$. Let
$z=(\widehat{\Psi}^{\infty}\otimes\id_{H})(w\otimes I_N )$.
Then
\begin{align*}
\widehat{z}(\tau) & =((\id\otimes\tau)\circ(\widehat{\Psi}^{\infty}\otimes\id_H ))(w\otimes I_N )
=\widehat{\Psi}^{\infty}((\id\otimes\tau)(w\otimes I_N )) \\
& =\widehat{\Psi}^{\infty}(\widehat{w}(\tau)\otimes I_N )=\Psi^{\infty}(\widehat{w}(\tau)\otimes I_N ) .
\end{align*}

\begin{lemma}\label{lem:quasi1}With the above notations and assumptions,
$$
\sum_{i,j,k}(\sqrt{d_k }\rtimes_{\rho, u}w_{ij}^k )^* \widehat{w}(\tau)(\sqrt{d_k }\rtimes_{\rho, u}w_{ij}^k )=1 .
$$
\end{lemma}
\begin{proof}
By Proposition \ref {prop:property2}, $\widehat{\widehat{\rho}}$ has the Rohlin property.
Let $z_1 $ be a unitary element in
$(A\rtimes_{\rho, u}H\rtimes_{\widehat{\rho}}H^0 \rtimes_{\widehat{\widehat{\rho}}}H)\otimes H$
satisfying Equations (5, 1)-(5, 3) for $\widehat{\widehat{\rho}}$.
Then by Lemmas \ref {lem:projection2} and \ref {lem:projection}, $e\cdot_{\widehat{\widehat{\rho}}}\widehat{z_1 }(\tau)
=e\cdot_{\widehat{\widehat{\rho}}}\widehat{z}(\tau)=\frac{1}{N}$.
Since $\widehat{z}(\tau)=\Psi^{\infty}(\widehat{w}(\tau)\otimes I_N )$,
$$
\frac{1}{N}=e\cdot_{\widehat{\widehat{\rho}}}\Psi^{\infty}(\widehat{w}(\tau)\otimes I_N )
=\sum_I [e\cdot_{\widehat{\widehat{\rho}}}V_I^* (\widehat{w}(\tau)\rtimes_{\rho, u}1\rtimes_{\widehat{\rho}}1^0 )V_I ] .
$$
Since $V_I =(1\rtimes_{\widehat{\rho}}\tau)(W_I \rtimes_{\widehat{\rho}}1^0 )$,
$$
\frac{1}{N}=\sum_I (W_I^* \rtimes_{\widehat{\rho}}1^0 )(1\rtimes_{\widehat{\rho}}\tau_{(1)})\tau_{(2)}(e_{(1)})
(\widehat{w}(\tau)\rtimes_{\rho, u}1\rtimes_{\widehat{\rho}}1^0 )(1\rtimes_{\widehat{\rho}}\tau_{(1)}')\tau_{(2)}'(e_{(2)})
(W_I \rtimes_{\widehat{\rho}}1^0 ) ,
$$
where $\tau' =\tau$. Furthermore,
\begin{align*}
\frac{1}{N} & =\sum_I (W_I^* \rtimes_{\widehat{\rho}}1^0 )(1\rtimes_{\widehat{\rho}}\tau_{(1)})
(\widehat{w}(\tau)\rtimes_{\rho, u}1\rtimes_{\widehat{\rho}}1^0 )(1\rtimes_{\widehat{\rho}}\tau_{(1)}' )
(W_I \rtimes_{\widehat{\rho}}1^0 )(\tau_{(2)}\tau_{(2)}' )(e) \\
& =\sum_I (W_I^* \rtimes_{\widehat{\rho}}1^0 )(\widehat{w}(\tau)\rtimes_{\rho, u}1\rtimes_{\widehat{\rho}}\tau_{(1)}
\tau_{(1)}')(W_I \rtimes_{\widehat{\rho}}1^0 )(\tau_{(2)}\tau_{(2)}' )(e) \\
& =\sum_I (W_I^* \rtimes_{\widehat{\rho}}1^0 )(\widehat{w}(\tau)\rtimes_{\rho, u}1\rtimes_{\widehat{\rho}}1^0 )
(W_I \rtimes_{\widehat{\rho}}1^0 )(\tau\tau' )(e) \\
& =\frac{1}{N}\sum_I (W_I^* \rtimes_{\widehat{\rho}}1^0 )(\widehat{w}(\tau)\rtimes_{\rho, u}1\rtimes_{\widehat{\rho}}
1^0 )(W_I \rtimes_{\widehat{\rho}}1^0 ) \\
& =\frac{1}{N}\sum_I W_I^* (\widehat{w}(\tau)\rtimes_{\rho, u}1 )W_I .
\end{align*}
Therefore we obtain the conclusion.
\end{proof}

Next, we shall show the inverse direction of Lemma \ref{lem:quasi1}.

\begin{lemma}\label{lem:quasi2}Let $(\rho, u)$ be a twisted coaction of $H^0$ on $A$.
We suppose that there is a projection $p\in A_{\infty}$ such that
$$
\sum_{i,j,k}(\sqrt{d_k }\rtimes_{\rho, u}w_{ij}^k )^* (p\rtimes_{\rho, u} 1)(\sqrt{d_k }\rtimes_{\rho, u}w_{ij}^k )=1 .
$$
Then $(\rho, u)$ has the Rohlin property.
\end{lemma}
\begin{proof}
Le $\Psi$ be the isomorphism of $M_N (A)$ onto $A\rtimes_{\rho, u}H\rtimes_{\widehat{\rho}}H^0 $
defined in Theorem \ref{thm:duality}. Let $q=\Psi^{\infty}(p\otimes I_N )$. Then $q$ is a projection in
$(A\rtimes_{\rho, u}H\rtimes_{\widehat{\rho}}H^0 )_{\infty}$ since $p\otimes I_N \in M_N (A)_{\infty}$.
In the same way as in the proof of Lemma \ref{lem:quasi1},
$$
e\cdot_{\widehat{\widehat{\rho}}}q=e\cdot_{\widehat{\widehat{\rho}}}\Psi^{\infty}(p\otimes I_N )
=\frac{1}{N}\sum_I W_I^* (p\rtimes_{\rho, u}1 )W_I =\frac{1}{N} .
$$
Hence by Theorem \ref {thm:equivalent}, $\widehat{\widehat{\rho}}$ has the Rohlin property
since $\widehat{\widehat{\rho}}$ is saturated by Jeong and Park \cite [Theorem 3.3]{JP:finite}
and \cite [Proposition 3.18]{KT1:inclusion}. Therefore $(\rho, u)$ has the Rohlin property
by Proposition \ref {prop:exterior2}.
\end{proof}

\begin{thm}\label{thm:more1}Let $(\rho, u)$ be a twisted coaction of a finite dimensional
$C^*$-Hopf algebra $H^0$ on a unital $C^*$-algebra $A$. Let $\{w_{ij}^k \}$ be a system
of comatrix units of $H$. Then the following conditions are equivalent:
\newline
$(1)$ the twisted coaction $(\rho, u)$ has the Rohlin property,
\newline
$(2)$ there is a projection $p\in A_{\infty}$ such that
$\sum_{i,j,k}(\sqrt{d_k}\rtimes_{\rho, u}w_{ij}^k )^* p(\sqrt{d_k }\rtimes_{\rho, u}w_{ij}^k )=1$.
\end{thm}
\begin{proof}This is immediate by Lemmas \ref{lem:quasi1} and \ref{lem:quasi2}.
\end{proof}

\begin{cor}\label{cor:more2}Let $\rho$ be a coaction of $H^0$ on $A$.Then the
following conditions are equivalent:
\newline
$(1)$ the coaction $\rho$ has the Rohlin property,
\newline
$(2)$ there is a projection $p\in A_{\infty}$ such that $e\cdot_{\rho^{\infty}}p=\frac{1}{N}$.
\end{cor}
\begin{proof}
(1) implies (2): This is immediate by Lemma \ref {lem:projection}.
\newline
(2) implies (1): By Theorem \ref{thm:more1}, it suffices to show that (2)
implies that
$$
\sum_{i,j,k}(\sqrt{d_k }\rtimes_{\rho}w_{ij}^k )^* p(\sqrt{d_k }\rtimes_{\rho}w_{ij}^k )=1 .
$$
Since $\rho$ is a coaction of $H^0$ on $A$,
\begin{align*}
& \sum_{i,j,k}(\sqrt{d_k}\rtimes_{\rho}w_{ij}^k )^* p(\sqrt{d_k}\rtimes_{\rho}w_{ij}^k )
=\sum_{i,j,k}d_k (1\rtimes_{\rho}w_{ij}^{k*})p(1\rtimes_{\rho}w_{ij}^k ) \\
& =\sum_{i,j,k}d_k (1\rtimes_{\rho}S(w_{ji}^k ))p(1\rtimes_{\rho}w_{ij}^k )
=N\sum_{i,j,k}\frac{d_k }{N}\widehat{V}(S(w_{ji}^k ))p\widehat{V^* }(S(w_{ij}^k )) \\
& =N\sum_{i,k}\frac{d_k }{N}[S(w_{ii}^k )\cdot_{\rho^{\infty}}p]
=N[S(e)\cdot_{\rho^{\infty}}p]=N[e\cdot_{\rho^{\infty}}p]=1.
\end{align*}
Therefore we obtain the conclusion.
\end{proof}

\section{Example}\label{sec:example}
In this section, we shall give an example of an approximately representable
coaction of a finite dimensional $C^*$-Hopf algebra on a UHF-algebra which has also the Rohlin property.
\par
We note that the comultiplication $\Delta^0$ of $H^0$ can be regarded as a coaction
of $H^0$ on a $C^*$-algebra $H^0$. Hence we can consider the crossed product $H^0 \rtimes_{\Delta^0 }H$,
which is isomorphic to $M_N (\BC)$. Let $A=H^0 \rtimes_{\Delta^0}H$.
Let $A_n =\otimes_1^n A$, the $n$- times tensor product of $A$, for any $n\in \BN$.
In the usual way, we regard $A_n$ as a $C^*$-subalgebra
of $A_{n+1}$, that is, for any $a\in A_n$, the map $\imath_n : a\mapsto a\otimes(1^0 \rtimes_{\Delta^0}1)$ is
regarded as the inclusion of $A_n$ into $A_{n+1}$. Let $B$ be the inductive limit $C^*$-algebra of $\{(A_n , \imath_n \}$.
Then $B$ can be regarded as
a UHF-algebra of type $N^{\infty}$. 
Let $\widehat{V}$ be a unitary element
in $\Hom(H, A)$ defined by $\widehat{V}(h)=1^0 \rtimes_{\Delta^0}h$ for any $h\in H$
and let $V$ be the unitary element in $A\otimes H^0$ induced by $\widehat{V}$.
Recall that $\{v_{ij}^k \}$ and $\{w_{ij}^k \}$ are systems of matrix units and comatrix units of
$H$, respectively. Also, let $\{\phi_{ij}^k \}$ and $\{\omega_{ij}^k \}$ be systems of matrix
units and comatrix units of $H^0$, respectively.
Let
$$
v_1 =V=\sum_{i,j,k}(1^0 \rtimes_{\Delta^0}w_{ij}^k )\otimes\phi_{ij}^k
=\sum_{i,j,k}(1^0 \rtimes_{\Delta^0}v_{ij}^k )\otimes\omega_{ij}^k .
$$
For any $n\in \BN$ with $n\ge 2$, let 
\begin{align*}
v_n & =(1^0 \rtimes_{\Delta^0}1)\otimes \cdots\otimes(1^0 \rtimes_{\Delta^0}1)\otimes V \\
& =\sum_{i,j,k}(1^0 \rtimes_{\Delta^0}1)\otimes \cdots\otimes(1^0 \rtimes_{\Delta^0}1)
\otimes\widehat{V}(w_{ij}^k )\otimes\phi_{ij}^k \\
& =\sum_{i,j,k}(1^0 \rtimes_{\Delta^0}1)\otimes \cdots\otimes(1^0 \rtimes_{\Delta^0}1)
\otimes\widehat{V}(v_{ij}^k )\otimes\omega_{ij}^k .
\end{align*}
Let $u_n =v_1 v_2 \cdots v_n \in A_n \otimes H^0$ for any $n\in \BN$.
Then $u_n $ is a unitary element in $A_n \otimes H^0$ for any $n\in \BN$.

\begin{lemma}\label{lem:form}With the above notations,
\begin{align*}
u_n & =\sum\widehat{V}(w_{ij_1}^k )\otimes\widehat{V}(w_{j_1 j_2}^k)\otimes\cdots
\otimes\widehat{V}(w_{j_{n-1}j})\otimes\phi_{ij}^k \\
& =\sum\widehat{V}(v_{i_1 j_1}^{k_1} )\otimes\widehat{V}(v_{i_2  j_2}^{k_2})\otimes\cdots
\otimes\widehat{V}(v_{i_n j_n }^{k_n})\otimes\omega_{i_1 j_1}^{k_1}\cdots\omega_{i_n j_n}^{k_n} ,
\end{align*}
where the above summations are taken under all indices.
\end{lemma}
\begin{proof}
It is clear that the second equation holds.
We show the first equation by the induction.
We assume that
$$
u_n =\sum\widehat{V}(w_{ij_1}^k )\otimes\widehat{V}(w_{j_1 j_2}^k)\otimes\cdots
\otimes\widehat{V}(w_{j_{n-1}j})\otimes\phi_{ij}^k ,
$$
where the summation is taken under all indices.
Then
\begin{align*}
u_{n+1}=u_n v_{n+1} & =\sum\widehat{V}(w_{ij_1}^k )\otimes\widehat{V}(w_{j_1 j_2}^k )\otimes
\cdots\otimes\widehat{V}(w_{j_{n-1}j}^k )\otimes\widehat{V}(w_{st}^r )\otimes\phi_{ij}^k \phi_{st}^r \\
& =\sum\widehat{V}(w_{ij_1}^k )\otimes\widehat{V}(w_{j_1 j_2}^k )\otimes
\cdots\otimes\widehat{V}(w_{j_{n-1}j}^k )\otimes\widehat{V}(w_{jt}^k )\otimes\phi_{it}^k ,
\end{align*}
where the summations are taken under all indices. Therefore, we obtain the conclusion.
\end{proof}

For any $n\in\BN$, let $\rho_n =\Ad(u_n )\circ\rho_{H^0}^{A_n}$, that is,
for any $a\in A_n$,
$$
\rho_n (a)=u_n (a\otimes 1^0 )u_n^* .
$$

\begin{lemma}\label{lem:strong}With the above notations, $\rho_n$ is a coaction of
$H^0$ on $A_n$.
\end{lemma}
\begin{proof}
We have only to show that
$$
(u_n\otimes 1^0 )(\rho_{H^0}^{A_n}\otimes\id)(u_n )=(\id\otimes\Delta^0 )(u_n ) .
$$
By Lemma \ref{lem:form}, we can write that
$$
u_n =\sum\widehat{V}(v_{i_1 j_1}^{k_1} )\otimes\widehat{V}(v_{i_2  j_2}^{k_2})\otimes\cdots
\otimes\widehat{V}(v_{i_n j_n }^{k_n})\otimes\omega_{i_1 j_1}^{k_1}\cdots\omega_{i_n j_n}^{k_n},
$$
where the summation is taken under all indices.
Hence
\begin{align*}
u_n \otimes 1^0 & =\sum\widehat{V}(v_{i_1 j_1}^{k_1} )\otimes\cdots
\otimes\widehat{V}(v_{i_n j_n }^{k_n})\otimes\omega_{i_1 j_1}^{k_1}
\cdots\omega_{i_n j_n}^{k_n}\otimes 1^0 , \\
(\rho_{H^0}^{A_n}\otimes\id)(u_n )& =\sum\widehat{V}(v_{i_1 j_1}^{k_1} )
\otimes\cdots
\otimes\widehat{V}(v_{i_n j_n }^{k_n})\otimes1^0 \otimes\omega_{i_1 j_1}^{k_1}
\cdots\omega_{i_n j_n}^{k_n} ,
\end{align*}
where the summations are taken under all indices.
Thus since $\widehat{V}$ is a $C^*$-homomorphism of $H$ to $A$,
\begin{align*}
& (u_n \otimes 1^0 )(\rho_{H^0}^{A_n}\otimes\id)(u_n ) \\
& =\sum\widehat{V}(v_{i_1 j_1}^{k_1} )\widehat{V}(v_{s_1 t_1}^{r_1})\otimes\cdots
\otimes\widehat{V}(v_{i_n j_n }^{k_n})\widehat{V}(v_{s_n t_n}^{r_n})\otimes\omega_{i_1 j_1}^{k_1}
\cdots\omega_{i_n j_n}^{k_n} \otimes\omega_{s_1 t_1}^{r_1}\cdots\omega_{s_n t_n}^{r_n} \\
& =\sum\widehat{V}(v_{i_1 t_1}^{k_1} )\otimes\cdots
\otimes\widehat{V}(v_{i_n t_n }^{k_n})\otimes\omega_{i_1 j_1}^{k_1}\cdots\omega_{i_n j_n}^{k_n}
\otimes\omega_{j_1 t_1}^{k_1}\cdots\omega_{j_n t_n}^{k_n} \\
& =(\id\otimes\Delta^0 )(u_n ) ,
\end{align*}
where the summations are taken under all indices.
Therefore we obtain the conclusion.
\end{proof}

\begin{lemma}\label{lem:commutative}With the above notations, $(\imath_n \otimes\id)\circ\rho_n =\rho_{n+1}\circ\imath_n$
for any $n\in\BN$.
\end{lemma}
\begin{proof}In this proof, the summations are taken under all indices. Let $a$ be any element
in $A_n$. Then by Lemma \ref{lem:form}
\begin{align*}
& \rho_n (a)=u_n (a\otimes 1^0 )u_n^* \\
& =\sum(\widehat{V}(w_{ij_1}^k )\otimes\cdots\otimes\widehat{V}(w_{j_{n-1}j}))a(\widehat{V}(w_{st_1}^r )^* \otimes
\cdots\otimes\widehat{V}(w_{t_{n-1}t}^r )^* )\otimes\phi_{ij}^k \phi_{ts}^r \\
& =\sum(\widehat{V}(w_{ij_1}^k )\otimes\cdots\otimes\widehat{V}(w_{j_{n-1}j}))a(\widehat{V}(w_{st_1}^k )^* \otimes
\cdots\otimes\widehat{V}(w_{t_{n-1}j}^k )^* )\otimes\phi_{is}^k .
\end{align*}
Hence
\begin{align*}
& ((\imath_n \otimes\id)\circ\rho_n )(a) \\
& =\sum(\widehat{V}(w_{ij_1}^k )\otimes\cdots\otimes\widehat{V}(w_{j_{n-1}j}))a(\widehat{V}(w_{st_1}^k )^* \otimes
\cdots\otimes\widehat{V}(w_{t_{n-1}j}^k )^* )\otimes(1^0 \rtimes_{\Delta^0 }1) \\
& \otimes\phi_{is}^k .
\end{align*}
On the other hand,
\begin{align*}
& (\rho_{n+1}\circ\imath_n )(a)=\rho_{n+1}(a\otimes(1^0 \rtimes_{\Delta^0}1)) \\
& =\sum(\widehat{V}(w_{ij_1}^k )\otimes\cdots\otimes\widehat{V}(w_{j_{n-1}j_n}))a(\widehat{V}(w_{st_1}^k )^* \otimes
\cdots\otimes\widehat{V}(w_{t_{n-1}t_n }^k )^* )\otimes\widehat{V}(w_{j_n j}^k w_{t_n j}^{k*}) \\
& \otimes\phi_{is}^k \\
& =\sum(\widehat{V}(w_{ij_1}^k )\otimes\cdots\otimes\widehat{V}(w_{j_{n-1}j_n}))a(\widehat{V}(w_{st_1}^k )^* \otimes
\cdots\otimes\widehat{V}(w_{t_{n-1}t_n }^k )^* ) \\
& \otimes(1\rtimes_{\Delta^0}1)\epsilon(w_{j_n t_n }^k )\otimes\phi_{is}^k  \\
& =((\imath_n \otimes\id)\circ\rho_n )(a)
\end{align*}
since $w_{t_n j}^{k*}=S(w_{jt_n }^k )$. Therefore, we obtain the conclusion.
\end{proof}

By Lemma \ref{lem:commutative}, the inductive limit of $\{(\rho_n , \imath_n )\}$ is a
homomorphism of $B$ to $B\otimes H^0$. Furthermore, by Lemma \ref{lem:strong},
it is a coaction of $H^0$ on $B$. We denote it by $\rho$.

\begin{prop}\label{prop:appro}With the above notations, $\rho$ is approximately representable.
\end{prop}
\begin{proof}
Let $u$ be a unitary element in $B^{\infty}\otimes H^0$ defined by
$u=(u_n )$, where $A_n$ is regarded as $C^*$-subalgebra of $B$ for any $n\in\BN$.
We can easily show that $\rho$ and $u$ hold the following conditions:
\newline
(1) $\rho(x)=(\Ad(u)\circ\rho_{H^0}^B )(x)$ for any $x\in B$,
\newline
(2) $(u\otimes 1^0 )(\rho_{H^0 }^{B^{\infty}}\otimes\id)(u)=(\id\otimes\Delta^0 )(u)$,
\newline
(3) $(\rho^{\infty}\otimes\id)(u)(u\otimes 1^0 )=(\id\otimes\Delta^0 )(u)$.
\newline
Therefore, we obtain the conclusion.
\end{proof}

\begin{prop}\label{prop:example}With the above notations, $\rho$ has the Rohlin property.
\end{prop}
\begin{proof}
By Corollary \ref{cor:more2}, it suffices to show that there is a projection $p\in B_{\infty}$
such that $e\cdot_{\rho^{\infty}}p=\frac{1}{N}$. For any $n\in\BN$, let
$$
p_n =(1^0 \rtimes_{\Delta^0}1)\otimes\cdots(1^0 \rtimes_{\Delta^0}1)
\otimes(\tau\rtimes_{\Delta^0} 1)\in A_{n-1}' \cap A_n .
$$
Also, let $p=(p_n )$. Then clearly $p$ is a projection in $B_{\infty}$.
In order to show that $e\cdot_{\rho^{\infty}}p=\frac{1}{N}$, we have only to show that
$e\cdot_{\rho_n }p_n =\frac{1}{N}$ for any $n\in\BN$. We note that
\begin{align*}
& u_n (p_n \otimes 1^0 )u_n^* \\
& =\sum\widehat{V}(w_{ij_1}^k S(w_{t_1 s}^k ))\otimes
\widehat{V}(w_{j_1 j_2}^k S(w_{t_2 t_1}^k ))\otimes\cdots\otimes\widehat{V}(w_{j_{n-2}j_{n-1}}^k S(w_{t_{n-1}t_{n-2}}^k )) \\
& \otimes\widehat{V}(w_{j_{n-1}j}^k)(\tau\rtimes_{\Delta^0}1)\widehat{V}(S(w_{jt_{n-1}}^k ))\otimes\phi_{is}^k,
\end{align*}
where the summation is taken under all indices.
Hence since $e=\sum_{f,q}\frac{d_f}{N}w_{qq}^f$,
\begin{align*}
& e\cdot_{\rho_n}p_n \\
& =\sum\frac{d_f}{N}\widehat{V}(w_{ij_1}^k S(w_{t_1 s}^k ))\otimes
\widehat{V}(w_{j_1 j_2}^k S(w_{t_2 t_1}^k ))\otimes\cdots\otimes\widehat{V}(w_{j_{n-2}j_{n-1}}^k S(w_{t_{n-1}t_{n-2}}^k )) \\
& \otimes\widehat{V}(w_{j_{n-1}j}^k )(\tau\rtimes_{\Delta^0}1)\widehat{V}(S(w_{jt_{n-1}}^k ))\phi_{is}^k (w_{qq}^f ) \\
& =\sum\frac{d_k}{N}\widehat{V}(w_{ij_1}^k S(w_{t_1 i}^k ))\otimes
\widehat{V}(w_{j_1 j_2}^k S(w_{t_2 t_1}^k ))\otimes\cdots\otimes\widehat{V}(w_{j_{n-2}j_{n-1}}^k S(w_{t_{n-1}t_{n-2}}^k )) \\
& \otimes\widehat{V}(w_{j_{n-1}j}^k )(\tau\rtimes_{\Delta^0}1)\widehat{V}(S(w_{jt_{n-1}}^k )) \\
& =\sum\frac{d_k}{N}\widehat{V}(\epsilon(w_{t_1 j_1}^k ))\otimes
\widehat{V}(w_{j_1 j_2}^k S(w_{t_2 t_1}^k ))\otimes\cdots\otimes\widehat{V}(w_{j_{n-2}j_{n-1}}^k S(w_{t_{n-1}t_{n-2}}^k )) \\
& \otimes\widehat{V}(w_{j_{n-1}j}^k )(\tau\rtimes_{\Delta^0}1)\widehat{V}(S(w_{jt_{n-1}}^k )) \\
& =\sum\frac{d_k}{N}(1^0 \rtimes_{\Delta^0}1)\otimes\widehat{V}(w_{j_1 j_2}^k S(w_{t_2 j_1}^k ))\otimes
\widehat{V}(w_{j_1 j_2}^k S(w_{t_2 t_1}^k ))\otimes\cdots \\
& \otimes\widehat{V}(w_{j_{n-2}j_{n-1}}^k S(w_{t_{n-1}t_{n-2}}^k ))
\otimes\widehat{V}(w_{j_{n-1}j}^k )(\tau\rtimes_{\Delta^0}1)\widehat{V}(S(w_{jt_{n-1}}^k )) ,
\end{align*}
where the summations are taken under all indices. Doing  this in the same way as in the above
for $n-1$ times, we can obtain that
\begin{align*}
e\cdot_{\rho_n}p_n & =\sum\frac{d_k}{N}(1^0 \rtimes_{\Delta^0}1)\otimes\cdots\otimes(1^0 \rtimes_{\Delta^0}1)
\otimes\widehat{V}(w_{j_{n-1}j}^k )(\tau\rtimes_{\Delta^0 }1)\widehat{V}(S(w_{jj_{n-1}}^k )) \\
& =\sum\frac{d_k}{N}(1^0 \rtimes_{\Delta^0}1)\otimes\cdots\otimes(1^0 \rtimes_{\Delta^0}1)
\otimes([w_{j_{n-1}j_{n-1}}\cdot_{\Delta^0}\tau]\rtimes_{\Delta^0} 1) \\
& =(1^0 \rtimes_{\Delta^0}1)\otimes\cdots\otimes(1^0 \rtimes_{\Delta^0}1)
\otimes([e\cdot_{\Delta^0}\tau]\rtimes_{\Delta^0} 1) \\
& =(1^0 \rtimes_{\Delta^0}1)\otimes\cdots\otimes(1^0 \rtimes_{\Delta^0}1)
\otimes\tau_{(1)}\tau_{(2)}(e) (1^0 \rtimes_{\Delta^0}1) \\
& =\frac{1}{N}(1^0 \rtimes_{\Delta^0}1)\otimes\cdots\otimes(1^0 \rtimes_{\Delta^0}1) ,
\end{align*}
where the summations are taken under all indices. Therefore, we obtain the conclusion.
\end{proof}

\section{1-cohomology vanishing theorem}\label{sec:vanish1}
Let $\rho$ be a coaction of $H^0 $ on $A$ with the Rohlin property.
In this section, we shall show that for any coaction $\sigma$ of $H^0$ on $A$ which is
exterior equivalent to $\rho$, there is a unitary element $x\in A\otimes H^0$
such that $\sigma=\Ad(x\otimes 1^0 )\circ\rho\circ\Ad(x^* )$.
\par
Let $\rho$ and $\sigma$ be as above. Since $\rho$ and $\sigma$ are
exterior equivalent, there is a unitary
element $v\in A\otimes H^0 $ satisfying the following conditions.
\newline
(8, 1) $\sigma=\Ad(v)\circ\rho$,
\newline
(8, 2) $(v\otimes 1^0)(\rho\otimes\id_{H^0 })(v) =(\id\otimes\Delta^0 )(v)$.
\newline
Since $\rho$ has the Rohlin property, there is a unitary element $w$
in $(A\rtimes_{\sigma}H)\otimes H$ satisfying Equations (5, 1)-(5, 3) for $\widehat{\rho}$.
By Proposition \ref{prop:exterior2}, $\sigma$ has also the Rohlin property.
Hence there is a unitary element $w_1 \in (A\rtimes_{\sigma}H)\otimes H$
satisfying Equations (5, 1)-(5, 3) for $\widehat{\sigma}$. By Lemma \ref{lem:projection2},
$\widehat{w_1 }(\tau)=\widehat{w}(\tau)$.

Let $x=N(\id\otimes e)(v\rho^{\infty}(\widehat{w}(\tau)))
=N\widehat{v}(e_{(1)})[e_{(2)}\cdot_{\rho^{\infty}}\widehat{w}(\tau)]$.

\begin{lemma}\label{lem:unitary2}With the above notations, the element $x$ is
a unitary element in $A^{\infty}$ such that
$\rho^{\infty}(x)=v^* (x\otimes 1^0 )$.
\end{lemma}

\begin{proof}
Let $f=e$. Then by Lemmas \ref{lem:homo} and \ref{lem:projection2}
\begin{align*}
xx^* & =N^2 \widehat{v}( e_{(1)})[e_{(2)}\cdot_{\rho^{\infty}}\widehat{w}(\tau)]
[f_{(2)}\cdot_{\rho^{\infty}}\widehat{w}(\tau)]^* \widehat{v}(f_{(1)})^*  \\
& =N^2 \widehat{v}( e_{(1)})[e_{(2)}\cdot_{\rho^{\infty}}\widehat{w}(\tau)]
[S(f_{(2)})^* \cdot_{\rho^{\infty}}\widehat{w}(\tau)] \widehat{v}(f_{(1)})^*  \\
& =N^2 \widehat{v}( e_{(1)})(1\rtimes_{\rho}e_{(2)})\widehat{w}(\tau)(1\rtimes_{\rho}S(e_{(3)}))
(1\rtimes_{\rho}S(f_{(3)}^* ))\widehat{w}(\tau)(1\rtimes_{\rho}f_{(2)}^* )\widehat{v}(f_{(1)})^* \\ 
& =N^2 \widehat{v}( e_{(1)})(1\rtimes_{\rho}e_{(2)})\widehat{w}(\tau)\tau(e_{(3)}f_{(3)}^* )
(1\rtimes_{\rho}f_{(2)}^* )\widehat{v}(f_{(1)})^* \\
& =N^2 \widehat{v}(e_{(1)})(1\rtimes_{\rho}e_{(2)})\widehat{w}(\tau)(1\rtimes_{\rho}S(e_{(3)})e_{(4)}f_{(2)}^* )
\tau(e_{(5)}f_{(3)}^* )\widehat{v}(f_{(1)})^* \\
& =N^2 \widehat{v}(e_{(1)})(1\rtimes_{\rho}e_{(2)})\widehat{w}(\tau)
(1\rtimes_{\rho}S(e_{(3)}))\tau(e_{(4)}f_{(2)}^* )\widehat{v}(f_{(1)})^* \\
& =N^2 \widehat{v}(e_{(1)})(1\rtimes_{\rho}e_{(2)})\widehat{w}(\tau)(1\rtimes_{\rho}S(e_{(3)}))\widehat{v}^* (S(f_{(1)})^* )
\tau(e_{(4)}f_{(2)}^* ) \\
& =N^2 \widehat{v}(e_{(1)})(1\rtimes_{\rho}e_{(2)})\widehat{w}(\tau)(1\rtimes_{\rho}S(e_{(3)}))
\widehat{v}^* (S(S(e_{(4)})e_{(5)}f_{(1)}^* ))\tau(e_{(6)}f_{(2)}^* ) \\
& =N^2 \widehat{v}(e_{(1)})(1\rtimes_{\rho}e_{(2)})\widehat{w}(\tau)(1\rtimes_{\rho}S(e_{(3)}))
\widehat{v}^* (e_{(4)})\tau(e_{(5)}f) \\
& =N \widehat{v}(e_{(1)})(1\rtimes_{\rho}e_{(2)})\widehat{w}(\tau)(1\rtimes_{\rho}S(e_{(3)}))
\widehat{v}^* (e_{(4)}) \\
& =N(\id\otimes e)(v\rho^{\infty}(\widehat{w}(\tau))v^* ) \\
& =N[e\cdot_{\sigma^{\infty}}\widehat{w_1 }(\tau)] =1.
\end{align*}
Let $y=N(\id\otimes e)(v^* \sigma^{\infty}(\widehat{w_1 }(\tau)))
=N\widehat{v}^* (e_{(1)})[e_{(2)}\cdot_{\sigma^{\infty}}\widehat{w_1 }(\tau)]$.
Then by the above discussions, $yy^* =1$. On the other hand, by Lemmas \ref{lem:projection2}
and \ref{lem:projection}
\begin{align*}
y^* & =N[S(e_{(2)}^* )\cdot_{\sigma^{\infty}}\widehat{w}(\tau)]\widehat{v}(S(e_{(1)}^* ))
=N\widehat{v}(S(e_{(4)}^* ))[S(e_{(3)})^* \cdot_{\rho^{\infty}}\widehat{w}(\tau)]
\widehat{v}^* (S(e_{(2)}^* ))\widehat{v}(S(e_{(1)}^* )) \\
& =N\widehat{v}(S(e_{(2)}^* ))[S(e_{(1)})^* \cdot_{\rho^{\infty}}\widehat{w}(\tau)]
=N(\id\otimes S(e)^* )(v\rho^{\infty}(\widehat{w}(\tau)))=x .
\end{align*}
Thus $x^* x=1$. Hence $x$ is a unitary element in $A^{\infty}$. Finally, we shall show
that $\rho^{\infty}(x)=v^* (x\otimes 1^0 )$. Noting that $(v\otimes 1^0 )(\rho\otimes\id)(v)=(\id\otimes\Delta^0 )(v)$,
\begin{align*}
\rho^{\infty}(x) & =N\rho^{\infty}((\id\otimes e)(v\rho^{\infty}(\widehat{w}(\tau)))) \\
& =N((\id\otimes\id_{H^0 } \otimes e)\circ(\rho^{\infty}\otimes\id_{H^0 }))(v\rho^{\infty}(\widehat{w}(\tau))) \\
& =N(\id\otimes\id_{H^0 }\otimes e)((\rho^{\infty}\otimes\id_{H^0 })(v)
((\rho^{\infty}\otimes\id_{H^0 })\circ\rho^{\infty})(\widehat{w}(\tau))) \\
& =N(\id\otimes\id_{H^0 }\otimes e)((v^* \otimes 1^0 )
(\id\otimes\Delta^0 )(v)((\id\otimes\Delta^0 )\circ\rho^{\infty})(\widehat{w}(\tau)))) \\
& =Nv^* (\id\otimes\id_{H^0 }\otimes e)((\id\otimes\Delta^0 )(v\rho^{\infty}(\widehat{w}(\tau)))) \\
& =Nv^* (\id\otimes e)(v\rho^{\infty}(\widehat{w}(\tau)))\otimes 1^0 \\
& =v^* (x\otimes 1^0 ) .
\end{align*}
\end{proof}

\begin{lemma}\label{lem:appro}With the above notations, for any $\epsilon>0$ there is a unitary element
$x_0 $ in $A$ such that
$$
||v-(x_0 \otimes 1)\rho(x_0^* )||<\epsilon.
$$
\end{lemma}

\begin{proof}By Lemma \ref{lem:unitary2}, there is a unitary element $x\in A^{\infty}$
such that $v=(x\otimes 1^0 )\rho^{\infty}(x^* )$. Since $x$ is a unitary element in $A^{\infty}$,
for any $\epsilon>0$, there is a unitary element $x_0 \in A$ such that
$||v-(x_0 \otimes 1)\rho(x_0^* )||<\epsilon$.
\end{proof}

\begin{thm}\label{thm:vanish1}Let $\rho$ and $\sigma$ be coactions of $H^0 $ on $A$ which are
exterior equivalent.
We suppose that $\rho$ has the Rohlin property.
Then there is a unitary element $x\in A$ such that
$$
\sigma=\Ad(x\otimes 1^0 )\circ\rho\circ\Ad(x^* ) .
$$
\end{thm}

\begin{proof}Let $v$ be a unitary element in $A\otimes H^0$ satisfying 
Equations (8, 1) and (8, 2). By Lemma \ref{lem:appro}, there is a unitary element $x_0 \in A$ such that
$$
||v-(x_0 \otimes 1)\rho(x_0^* )||<1 .
$$
Let
$$
\rho_1 =\Ad(x_0 \otimes 1)\circ\rho\circ\Ad(x_0^* )=\Ad((x_0 \otimes 1^0 )\rho(x_0^* ))\circ\rho .
$$
Let $v_1 =(x_0 \otimes 1^0 )\rho(x_0^* )$. Then $\rho_1 $ is a coaction of $H^0 $ on $A$.
Also, $\sigma=\Ad(vv_1^* )\circ\rho_1$. Let $v_2 =vv_1^* $. Then $v_2$ is a unitary
element in $A\otimes H^0 $ with
$$
||v_2 -1||=||v-v_1 ||=||v-(x_0 \otimes 1^0 )\rho(x_0^* )||<1 .
$$
Furthermore, since $v_1 =(x_0 \otimes 1^0 )\rho(x_0^* )$,
\begin{align*}
(v_2 \otimes 1^0 )(\rho_1 \otimes\id)(v_2 ) & =(v_2 \otimes 1^0 )(v_1 \otimes 1^0 )(\rho\otimes\id)(v_2 )
(v_1^* \otimes 1^0 ) \\
& =(v\otimes 1^0 )(\rho\otimes\id)(v_2 )(v_1^* \otimes 1^0 ) \\
& =(v\otimes 1^0 )(\rho\otimes\id)(v)(\rho\otimes\id)(v_1^* )(v_1^* \otimes 1^0  ) \\
& =(\id\otimes\Delta^0 )(v)(\rho\otimes\id)(v_1^* )(v_1^* \otimes 1^0 ) \\
& =(\id\otimes \Delta^0 )(v_2 )(\id\otimes\Delta^0 )(v_1 )(\rho\otimes\id)(v_1^* )(v_1^* \otimes 1^0 ) \\
& =(\id\otimes \Delta^0 )(v_2 )(x_0 \otimes 1^0 \otimes 1^0 )((\id\otimes\Delta^0 )\circ\rho)(x_0^* ) \\
& \times ((\rho\otimes\id)\circ\rho)(x_0 )(\rho(x_0^* )\otimes 1^0 )(\rho(x_0 )\otimes 1^0 )
(x_0^* \otimes 1^0 \otimes 1^0 ) \\
& =(\id\otimes \Delta^0 )(v_2 )(x_0 \otimes 1^0 \otimes 1^0 )
((\rho\otimes\id)\circ\rho)(x_0^* x_0 )(x_0^* \otimes 1^0 \otimes 1^0 ) \\
& =(\id\otimes \Delta^0 )(v_2 ) .
\end{align*}
Thus
$(v_2 \otimes 1^0 )(\rho_1 \otimes\id)(v_2 )=(\id_A \otimes\Delta^0 )(v_2 )$.
Let $y=(\id_A \otimes e)(v_2 )$. Then
\begin{align*}
\rho_1 (y) & =(\id_A \otimes\id_{H^0 } \otimes e)((\rho_1 \otimes\id_{H^0 })(v_2 )) \\
& =(\id_A \otimes\id_{H^0 } \otimes e)((v_2^* \otimes 1^0 )(\id_A \otimes\Delta^0 )(v_2 )) \\
& =v_2^* [(\id_A \otimes e)(v_2 )\otimes 1^0 ]=v_2^* (y\otimes 1^0 ) .
\end{align*}
Since $||1-y||=||(\id_A \otimes e)(1-v_2 )||\leq ||1-v_2 ||<1$, $y$ is invertible.
Let $y=x|y|$ be the polar decomposition of $y$. Then $x$ is a unitary element in $A$ and
$$
\rho_1 (y)=v_2^* (y\otimes 1^0 )=v_2^* (x\otimes 1^0 )(|y|\otimes 1^0 ) .
$$
Hence
$$
\rho_1 (x)\rho_1 (|y|)=v_2^* (x\otimes 1^0 )(|y|\otimes 1^0 ) .
$$
Also,
$$
\rho_1 (y^* y)=(y^* \otimes 1^0 )v_2 v_2^* (y\otimes 1)=y^* y\otimes 1 .
$$
Thus $\rho_1 (|y|)=|y|\otimes 1^0 $. Hence $\rho_1 (x)=v_2^* (x\otimes 1^0 )$.
It follows that
\begin{align*}
\Ad(x\otimes 1^0 )\circ\rho_1 \circ\Ad(x^* ) & =\Ad((x\otimes 1^0 )\rho_1 (x^*))\circ\rho_1 \\
& =\Ad(v_2 )\circ\rho_1 \\
& =\Ad(v)\circ\rho =\sigma .
\end{align*}
Since $\rho_1 =\Ad(x_0 \otimes 1^0 )\circ\rho \circ\Ad(x_0^* )$, we obtain the conclusion.
\end{proof}

\section{2-cohomology vanishing theorem}\label{sec:vanish2}
Let $(\rho, u)$ be a twisted coaction of $H^0$ on $A$ with the Rohlin
property. Let $w$ be a unitary element in $(A^{\infty}\rtimes_{\rho^{\infty}, u}H)\otimes H$
satisfying Equation (5, 1)-(5, 3) and let $\widehat{w}$ be the unitary element in
$\Hom (H^0, A^{\infty}\rtimes_{\rho^{\infty}, u}H)$ induced by $w$. In this section, we shall show that
there is a unitary element $x\in A\otimes H^0$ such that
$$
(x\otimes 1^0 )(\rho\otimes\id)(x)u(\id\otimes\Delta^0 )(x^* )=1\otimes1^0 \otimes 1^0 .
$$
We recall that in Section \ref{sec:duality}, we construct a system of matrix units of $M_N (\BC)$,
$$
\{(W_I^* \rtimes_{\widehat{\rho}}1^0 )(1\rtimes_{\rho, u}1\rtimes_{\widehat{\rho}}\tau)
(W_J \rtimes_{\widehat{\rho}}1^0 )\}_{I, J\in\Lambda}
$$
which is in contained in
$A^{\infty}\rtimes_{\rho^{\infty}, u}H$, where $W_I =\sqrt{d_k}\rtimes_{\rho, u}w_{ij}^k$
for any $I=(i, j, k)\in\Lambda$. By Lemmas \ref{lem:homo} and \ref{lem:quasi1}, we obtain the following lemma.

\begin{lemma}\label{lem:matrix}With the above notations and assumptions,
$\{W_I^* \widehat{w}(\tau)W_J \}_{I,J\in\Lambda}$ is a system of matrix units of
$M_N (\BC)$, which is contained in $A^{\infty}\rtimes_{\rho^{\infty}, u}H$.
\end{lemma}
\begin{proof}By the proof of Lemma \ref{lem:orthogonal}, for any $I=(i, j, k), J=(s, t, r)\in\Lambda$,
$$
W_I W_J^* =\sum_{t_2 , t_3 , j_3}\sqrt{d_k , d_r}\widehat{u}(w_{j_3 i}^k
S(w_{t_2 t_3}^r ), w_{s t_2}^r )^*\rtimes_{\rho, u}w_{j_3 j}^k w_{t_3 t}^{r*}.
$$
Hence by Lemma \ref{lem:homo} and \cite [Theorem 2.2]{SP:saturated},
$$
\widehat{w}(\tau)W_I W_J^* \widehat{w}(\tau)=\sum_{t_2 , t_3 , j_3 }\sqrt{d_k d_r}\tau(w_{j_3 j}^k w_{t_3 t}^{r*} )
\widehat{u}^* (w_{ij_3 }^k w_{t_2 t_3}^{r*}, w_{t_2 s}^r )\widehat{w}(\tau).
$$
If $k\ne r$ or $j\ne t$, then $\widehat{w}(\tau)W_I W_J^* \widehat{w}(\tau)=0$. We suppose that
$k=r$ and $j=t$.
$$
\widehat{w}(\tau)W_I W_J^* \widehat{w}(\tau) =\sum_{t_2 , t_3}\widehat{u}^*
(w_{it_3 }^k S(w_{t_3 t_2 }^k), w_{t_2 s}^k )\widehat{w}(\tau)
=\epsilon(w_{is}^k )\widehat{w}(\tau)=\delta_{is}\widehat{w}(\tau),
$$
where $\delta_{is}$ is the Kronecker delta.
Thus for any $K, L, I, J\in\Lambda$,
$$
W_K^* \widehat{w}(\tau)W_I W_J^* \widehat{w}(\tau)W_L =0
$$
if $I\ne J$. We suppose that $I=J$. Then since $\widehat{w}(\tau)W_I W_I^* \widehat{w}(\tau)=\widehat{w}(\tau)$,
$$
W_K^* \widehat{w}(\tau)W_I  W_I^* \widehat{w}(\tau)W_L =W_K^* \widehat{w}(\tau)W_L .
$$
Furthermore,
$$
\sum_{I\in \Lambda}W_I^* \widehat{w}(\tau)W_I W_I^* \widehat{w}(\tau)W_I
=\sum_{I\in \Lambda}W_I^* \widehat{w}(\tau)W_I =1
$$
by Lemma \ref{lem:quasi1}. Therefore we obtain the conclusion.
\end{proof}

We suppose that the $C^*$-Hopf algebra $H^0$ acts on a unital $C^*$-algebra
$\BC$ trivially. Then by the discussions before Lemma \ref{lem:matrix},
the set $\{(W_{0 I}^* \rtimes_{\Delta}1^0 )(1\rtimes 1\rtimes_{\Delta}\tau)
(W_{0 J} \rtimes_{\Delta}1^0 )\}_{I, J\in\Lambda}$ is a system of matrix units of
$\BC\rtimes H\rtimes_{\Delta}H^0$ which is isomorphic to $M_N (\BC)$, where
$W_{0 I}=\sqrt{d_k}\rtimes w_{ij}^k \in \BC\rtimes H$ for any $I=(i,j,k)\in\Lambda$.
Thus we obtain the following homomorphism $\theta$ of $\BC\rtimes H\rtimes_{\Delta}H^0$
into $A^{\infty}\rtimes_{\rho^{\infty}, u}H$. For any $I, J\in\Lambda$,
$$
\theta((W_{0 I}^* \rtimes_{\Delta}1^0 )(1\rtimes 1\rtimes_{\Delta}\tau)
(W_{0 J} \rtimes_{\Delta}1^0 ))=W_I^* \widehat{w}(\tau)W_J .
$$

\begin{lemma}\label{lem:homo11}With the above notations, for any $h\in H$,
$$
\theta(1\rtimes h)=\sum_{i,j,k}d_k (1\rtimes_{\rho, u} w_{ij}^k )^* \widehat{w}(\tau)(1\rtimes_{\rho, u} w_{ij}^k h) .
$$
\end{lemma}
\begin{proof}Let $h$ be any element in $H$. Then by Lemma \ref{lem:quasi1},
\begin{align*}
1\rtimes h & =\sum_{I\in\Lambda}(W_I\rtimes_{\Delta}1^0 )^*
(1\rtimes 1\rtimes_{\Delta}\tau)(W_I \rtimes_{\Delta}1^0 )(1\rtimes h\rtimes_{\Delta}1^0 ) \\
& =\sum_{i, j, k}d_k (1\rtimes w_{ij}^k \rtimes_{\Delta}1^0 )^* (1\rtimes 1\rtimes_{\Delta}\tau)
(1\rtimes w_{ij}^k h \rtimes_{\Delta}1^0 ) .
\end{align*}
Since $\{w_{ij}^k \}$ is a system of comatrix units of $H$, for any $i, j,k$ there are elements
$(c_{ij}^k )_{st}^r \in \BC$ such that $w_{ij}^k h=\sum_{st}^r (c_{ij}^k )_{st}^r w_{st}^r$.
Hence
$$
1\rtimes h=\sum_{i,j,k,s,t,r}d_k (c_{ij}^k )_{st}^r (1\rtimes w_{ij}^k \rtimes_{\Delta}1^0 )^*
(1\rtimes 1\rtimes_{\Delta}\tau)(1\rtimes w_{st}^r \rtimes_{\Delta}1^0 ).
$$
Thus by the definition of $\theta$,
\begin{align*}
\theta(1\rtimes h) & =\sum_{i,j,k,r,s,t}d_k (c_{ij}^k )_{st}^r (1\rtimes_{\rho, u}w_{ij}^k )^* \widehat{w}(\tau)
(1\rtimes_{\rho, u}w_{st}^r ) \\
& =\sum_{i,j,k}d_k (1\rtimes_{\rho, u}w_{ij}^k )^* \widehat{w}(\tau)(1\rtimes_{\rho, u}\sum_{s, t, r}
(c_{ij}^k )_{st}^r w_{st}^r ) \\
& =\sum_{i,j,k}d_k (1\rtimes_{\rho, u}w_{ij}^k )^* \widehat{w}(\tau)(1\rtimes_{\rho, u}w_{ij}^k h) .
\end{align*}
\end{proof}

The restriction of $\theta$ to $1\rtimes H$, the $C^*$-subalgebra of $\BC\rtimes H\rtimes_{\Delta}H^0$
is a homomorphism of $H$ to $A^{\infty}\rtimes_{\rho^{\infty}, u}H$. Hence there is a unitary
element $v\in (A^{\infty}\rtimes_{\rho^{\infty}, u}H)\otimes H^0$ such that $\theta|_{1\rtimes H}=\widehat{v}$.
We recall the definitions $V$ and $\widehat{V}$. Let $\widehat{V}$ be a linear map from $H$ to
$A\rtimes_{\rho, u}H$ defined by $\widehat{V}(h)=1\rtimes_{\rho, u}h$ for any $h\in H$ and
let $V$ be the element in $(A\rtimes_{\rho, u}H)\otimes H^0$ induced by $\widehat{V}$. Then
$V$ and $\widehat{V}$ are unitary elements in $(A\rtimes_{\rho, u}H)\otimes H^0$ and
$\Hom (H, A\rtimes_{\rho, u}H)$, respectively. Let $x$ be a unitary element in
$(A^{\infty}\rtimes_{\rho^{\infty}, u}H)\otimes H^0$ defined by $x=vV^*$.

\begin{lemma}\label{lem:unitary11}With the above notations, $\widehat{x}(h)\in A^{\infty}$ for any
$h\in H$.
\end{lemma}
\begin{proof}By Lemma \ref{lem:homo11} and \cite [Theorem 2.2]{SP:saturated}, for any $h\in H$,
\begin{align*}
\widehat{x}(h) & =\widehat{v}(h_{(1)})\widehat{V}(S(h_{(2)}^* ))^* \\
& =\sum_{i,j,k}d_k (1\rtimes_{\rho, u}w_{ij}^k )^* \widehat{w}(\tau)(1\rtimes_{\rho, u}w_{ij}^k h_{(1)})
(1\rtimes_{\rho, u}S(h_{(2)})^* )^* \\
& =\sum_{i,j,k, j_1 , j_2}d_k (\widehat{u}(S(w_{j_1 j_2}^k ), w_{ij_1}^k )^* \rtimes_{\rho, u}w_{j_2 j}^{k*})
(\widehat{w}(\tau)\rtimes_{\rho, u}w_{ij}^k h_{(1)}) \\
& \times (\widehat{u}(h_{(3)}^* , S(h_{(4)})^* )^* \rtimes_{\rho, u}S(h_{(2)})) \\
& =\sum_{i,j,k, j_1 , j_2 , j_3 , j_4 , i_1}d_k (\widehat{u}(S(w_{j_1 j_2}^k ), w_{ij_1}^k )^* 
[w_{j_2 j_3}^{k*}\cdot_{\rho, u}\widehat{w}(\tau)]\widehat{u}(w_{j_3 j_4}^{k*}, w_{ii_1}^k h_{(1)}) \\
& \rtimes_{\rho, u} w_{j_4 j}^{k*}w_{i_1 j}^k h_{(2)})(\widehat{u}(h_{(4)}^* , S(h_{(5)})^* )^* \rtimes_{\rho, u}S(h_{(3)})) \\
& =\sum_{i,j,k, j_1 ,j_2 , j_4 , i_1}d_k (\widehat{u}^* (w_{j_1 j_2}^{k*}, w_{j_1 i}^k )[w_{j_2 j_3}^{k*}\cdot_{\rho, u}
\widehat{w}(\tau)]\widehat{u}(w_{j_3 j_4}^{k*}, w_{ii_1}^k h_{(1)}) \\
& \rtimes_{\rho, u}S(w_{jj_4}^k )w_{i_1 j}^k h_{(2)})(\widehat{u}^* (S(h_{(4)}), h_{(5)})\rtimes_{\rho, u}S(h_{(3)})) \\
& =\sum_{i,k,j_1 , j_2 , j_3 , j_4}d_k \widehat{u}^* (w_{j_1 j_2}^{k*}, w_{j_1 i}^k )
[w_{j_2 j_3 }^{k*}\cdot_{\rho, u}\widehat{w}(\tau)]\widehat{u}(w_{j_3 j_4}^{k*}, w_{ij_4}^k h_{(1)})\rtimes_{\rho, u}h_{(2)}) \\
& \times (\widehat{u}^* (S(h_{(4)}), h_{(5)})\rtimes_{\rho, u}S(h_{(3)})) \\
& =\sum_{i,k,j_1 , j_2 , j_3 , j_4}d_k \widehat{u}^* (w_{j_1 j_2}^{k*}, w_{j_1 i}^k )[w_{j_2 j_3}^{k*}\cdot_{\rho, u}
\widehat{w}(\tau)]\widehat{u}(w_{j_3 j_4}^{k*}, w_{ij_4}^k h_{(1)}) \\
& \times [h_{(2)}\cdot_{\rho, u}\widehat{u}^* (S(h_{(7)}), h_{(8)})]\widehat{u}(h_{(3)}, S(h_{(6)}))\rtimes_{\rho, u}h_{(4)}
S(h_{(5)}) \\
& =\sum_{i,k,j_1 , j_2 , j_3 , j_4}d_k \widehat{u}^* (w_{j_1 j_2}^{k*}, w_{j_1 i}^k )[w_{j_2 j_3}^{k*}\cdot_{\rho, u}
\widehat{w}(\tau)]\widehat{u}(w_{j_3 j_4}^{k*}, w_{ij_4}^k h_{(1)}) \\
& \times [h_{(2)}\cdot_{\rho, u}\widehat{u}^* (S(h_{(5)}), h_{(6)})]\widehat{u}(h_{(3)}, S(h_{(4)})).
\end{align*}
Furthermore, using the Equations (1) and (2) in Section \ref{sec:pre}, we can see that for any $h\in H$,
$$
\widehat{x}(h)=\sum_{i,k,j_1 , j_2 , j_3 , j_4}d_k \widehat{u}^* (w_{j_1 j_2}^{k*}, w_{j_1 i}^k )[w_{j_2 j_3}^{k*}\cdot_{\rho, u}
\widehat{w}(\tau)]\widehat{u}(w_{j_3 j_4}^{k*}, w_{ij_4}^k h).
$$
Since $w_{j_2 j_3 }^{k*}\cdot_{\rho, u}\widehat{w}(\tau)\in A^{\infty}$ for any $j_2 , j_3 , k$, we obtain the conclusion.
\end{proof}

By the above lemma, we can see that $x$ is a unitary element in $A^{\infty}\otimes H^0$.
We recall that $\rho_{H^0}^{A\rtimes_{\rho, u}H}$ is the trivial coaction of $H^0$ on
$A\rtimes_{\rho, u}H$ defined by $\rho_{H^0}^{A\rtimes_{\rho, u}H}(a)=a\otimes 1^0$
for any $a\in A\rtimes_{\rho, u}H$. Also, we note that $\rho=\Ad(V)\circ\rho_{H^0}^{A\rtimes_{\rho, u}H}$
by \cite [Lemma 3.12]{KT1:inclusion}, where we regard $A$ as a $C^*$-subalgebra of $A\rtimes_{\rho, u}H$.
Furthermore, since $\widehat{v}$ is a homomorphism of $H$ to $A^{\infty}\rtimes_{\rho^{\infty}, u}H$,
$$
(v\otimes 1^0 )(\rho_{H^0}^{A^{\infty}\rtimes_{\rho^{\infty}, u}H}\otimes\id)(v)=(\id\otimes\Delta^0 )(v).
$$
\begin{prop}\label{prop:coboundary}With the above notations,
$$
(x\otimes 1^0 )(\rho^{\infty}\otimes\id)(x)u(\id\otimes\Delta^0 )(x^* )=1\otimes 1^0 \otimes 1^0.
$$
\end{prop}
\begin{proof}
Since $x=vV^*$ and $\rho=\Ad(V)\circ\rho_{H^0}^{A\rtimes_{\rho, u}H}$,
\begin{align*}
& (\rho^{\infty}\otimes\id)(x^* )(x^* \otimes 1^0 )(\id\otimes\Delta^0 )(x) \\
& =(\rho^{\infty}\otimes\id)(Vv^* )(Vv^* \otimes 1^0 )(\id\otimes\Delta^0 )(vV^* ) \\
& =(V\otimes 1^0 )(\rho_{H^0}^{A^{\infty}\rtimes_{\rho^{\infty}, u}H}
\otimes\id)(Vv^* )(v^* \otimes1^0 )(\id\otimes\Delta^0 )(vV^* ).
\end{align*}
Since $(v\otimes 1^0 )(\rho_{H^0}^{A^{\infty}\rtimes_{\rho^{\infty}, u}H}\otimes\id)(v)=(\id\otimes\Delta^0 )(v)$,
$$
(\rho^{\infty}\otimes\id)(x^* )(x^* \otimes 1^0 )(\id\otimes\Delta^0 )(x)=(V\otimes 1^0 )(\rho_{H^0}^{A\rtimes_{\rho, u}H}
\otimes\id)(V)(\id\otimes\Delta^0 )(V^* )
=u
$$
by \cite [Lemma 3.12]{KT1:inclusion}.
\end{proof}

We recall that $\{\phi_{ij}^k \}$ is a system of matrix units of $H^0$.

\begin{lemma}\label{lem:coboundary2}Let $(\rho, u)$ be a twisted coaction of $H^0$ on $A$ with the Rohlin property.
Then for any $\epsilon>0$, there is a unitary element $x\in A\otimes H^0$ satisfying that
\begin{align*}
|| & (x\otimes 1^0)(\rho\otimes\id)(x)u(\id\otimes\Delta^0 )(x^* )-1\otimes 1^0 \otimes 1^0 ||<\epsilon \\
|| & x-1\otimes 1^0 ||<\epsilon +L||u-1\otimes 1^0 \otimes 1^0 ||,
\end{align*}
where
$$
L=\max \{ \sum_{i,j,k,t,r,t_1 , t_2 , j_1}d_k ||\widehat{V}(w_{ij}^k )^* ||\,
||\widehat{V}(w_{j_1 j}^k w_{t_2 t_1} ^r )\widehat{V}^* (w_{t_1 t}^r )|| \, , \, 1\}
$$
\end{lemma}
\begin{proof}Modifying the proof of Izumi \cite [Lemma 3.12]{Izumi:group},
we shall prove this lemma. By Proposition \ref {prop:coboundary}, there is a unitary
element $x_0 \in A^{\infty}\otimes H^0$ satisfying that
$$
(x_0 \otimes 1^0 )(\rho^{\infty}\otimes\id)(x_0 )u(\id\otimes\Delta^0 )(x_0^* )=1\otimes1^0 \otimes 1^0 .
$$
By the proof of Lemma \ref{lem:unitary11}, for any $h\in H$,
$$
\widehat{x_0}(h)=\sum_{i,j,k}d_k \widehat{V}(w_{ij}^k )^* \widehat{w}(\tau)\widehat{V}(w_{ij}^k h_{(1)})
\widehat{V}^* (h_{(2)}).
$$
Thus
$$
x_0 =\sum_{i,j,k,s,t,r,t_1}d_k \widehat{V}(w_{ij}^k )^* \widehat{w}(\tau)\widehat{V}(w_{ij}^k w_{st_1}^r )
\widehat{V}^* (w_{t_1 t}^r )\otimes\phi_{st}^r .
$$
Since $\sum_{i,j,k}d_k \widehat{V}(w_{ij}^k )^* \widehat{w}(\tau)\widehat{V}(w_{ij}^k )=1$ by Lemma \ref{lem:quasi1},
$$
1\otimes 1^0 =\sum_{i,j,k,s,t,r,t_1}d_k \widehat{V}(w_{ij}^k )^* \widehat{w}(\tau)\widehat{V}(w_{ij}^k )
\widehat{V}(w_{st_1}^r )\widehat{V}^* (w_{t_1 t}^r )\otimes\phi_{st}^r .
$$
Thus
$$
x_0 -1\otimes 1^0 =\sum_{i,j,k,s,t,r,t_1}d_k \widehat{V}(w_{ij}^k )^* \widehat{w}(\tau)
[\widehat{V}(w_{ij}^k w_{st_1}^r )-\widehat{V}(w_{ij}^k )
\widehat{V}(w_{st_1}^r )]\widehat{V}^* (w_{t_1 t}^r )\otimes\phi_{st}^r .
$$
Since $u=(V\otimes 1^0 )(\rho_{H^0}^{A\rtimes_{\rho, u}H}\otimes\id)(V)(\id\otimes\Delta^0 )(V^* )$
by \cite [Lemma 3.12]{KT1:inclusion},
$$
\widehat{V}(w_{ij}^k )\widehat{V}(w_{st_1}^r )
=\sum_{j_1 , t_2}\widehat{u}(w_{ij_1}^k , w_{st_2}^r )\widehat{V}(w_{j_1 j}^k w_{t_2 t_1}^r )
$$
for any $i,j,k,s,t_1 , r$. Hence
\begin{align*}
x_0 -1\otimes 1^0 
& =\sum_{i,j,k,s,t,r,t_1 , t_2 , j_1}d_k \widehat{V}(w_{ij}^k )^* \widehat{w}(\tau)
[\epsilon(w_{ij_1}^k )\epsilon(w_{st_2}^r )
-\widehat{u}(w_{ij_1}^k , w_{st_2}^r )] \\
& \times \widehat{V}(w_{j_1 j}^k w_{t_2 t_1}^r )
\widehat{V}^* (w_{t_1 t}^r )\otimes\phi_{st}^r .
\end{align*}
Since
$||\epsilon(w_{ij_1}^k )\epsilon(w_{st_2}^r )-\widehat{u}(w_{ij_1}^k , w_{st_2}^r )||
\le ||1\otimes 1^0 \otimes1^0 -u ||$ for any $i,j_1 , k, s,t_2 , r$,
\begin{align*}
||x_0 -1\otimes1\otimes1^0 || & \le\sum_{i,j,k,s,t,r,t_1 , t_2 , j_1}
d_k ||\widehat{V} (w_{ij}^k ) ||\, ||\widehat{V}(w_{j_1 j}^k w_{t_2 t_1}^r )\widehat{V}^* (w_{t_1 t}^r )|| \\
& \times ||1\otimes 1^0 \otimes1^0 -u||.
\end{align*}
Since $x_0$ is a unitary element in $A^{\infty}\otimes H^0$, we can choose a desired unitary element $x$
in $A\otimes H^0$.
\end{proof}

\begin{thm}\label{thm:vanish2}Let $(\rho, u)$ be a twisted coaction of a finite dimensional $C^*$-Hopf
algebra $H^0 $ on a unital $C^*$-algebra $A$ with the Rohlin property. Then there is a unitary element
$x\in A\otimes H^0 $such that
$$
(x\otimes 1^0 )(\rho\otimes\id)(x)u(\id\otimes\Delta^0 )(x^* )=1\otimes 1^0 \otimes 1^0 .
$$
\end{thm}
\begin{proof}We shall prove this lemma modifying the proof of \cite [Lemma 3.12]{Izumi:group}.
Let $u_0 =u$ and $\rho_0 =\rho$. By Lemma \ref{lem:coboundary2}, for $\frac{1}{2L}$, there is a unitary element
$y_0 \in A\otimes H^0 $ such that
$$
||1\otimes 1^0 \otimes 1^0 -(y_0 \otimes 1^0 )(\rho_0 \otimes\id)(y_0 )u_0 (\id\otimes\Delta^0 )(y_0^* )||
<\frac{1}{2L}<\frac{1}{2} .
$$
Let 
$$
\rho_1 =\Ad(y_0 )\circ\rho_0 , \quad u_1=(y_0 \otimes 1^0 )(\rho_0 \otimes \id)(y_0 )u_0 (\id\otimes\Delta^0 )(y_0^* ) .
$$
Then since $(\rho_1 , u_1 )$ is a twisted coaction of $H^0 $ on $A$ which is exterior equivalent to
$(\rho_ 0 , u_0 )$, by Proposition \ref{prop:exterior2} $(\rho_1 , u_1 )$ has the Rohlin property. Thus
by Lemma \ref{lem:coboundary2}, for $\frac{1}{(2L)^2 }$, there is a unitary element $y_1 \in A\otimes H^0 $
such that
\begin{align*}
& ||1\otimes 1^0 \otimes 1^0 -(y_1 \otimes 1^0 )(\rho_1 \otimes\id)(y_1 )u_1 (\id\otimes\Delta^0 )(y_1^* )||
<\frac{1}{(2L)^2 }<\frac{1}{2^2 }, \\
& ||y_1 -1\otimes 1^0 ||<\frac{1}{(2L)^2 }+L||u_1 -1\otimes 1^0 \otimes 1^0 ||
<\frac{1}{(2L)^2 }+\frac{1}{2}
<\frac{1}{2^2 }+\frac{1}{2}=\frac{3}{2^2}
\end{align*}
since $u_1=(y_0 \otimes 1^0 )(\rho_0 \otimes \id)(y_0 )u_0 (\id\otimes\Delta^0 )(y_0^* )$.
Let 
$$
\rho_2 =\Ad(y_1 )\circ\rho_1 , \quad u_2=(y_1 \otimes 1^0 )(\rho_1 \otimes \id)(y_1 )u_1 (\id\otimes\Delta^0 )(y_1^* ) .
$$
Then since $(\rho_2 , u_2 )$ is a twisted coaction of $H^0 $ on $A$ which is exterior equivalent to
$(\rho_ 1 , u_1 )$, by Proposition \ref{prop:exterior2} $(\rho_2 , u_2 )$ has the Rohlin property. Thus
by Lemma \ref{lem:coboundary2}, for $\frac{1}{(2L)^3 }$, there is a unitary element $y_2 \in A\otimes H^0 $
such that
\begin{align*}
& ||1\otimes 1^0 \otimes 1^0 -(y_2 \otimes 1^0 )(\rho_2 \otimes\id(y_2 )u_2 (\id\otimes\Delta^0 )(y_2^* )||
<\frac{1}{(2L)^3 }<\frac{1}{2^3 }, \\
& ||y_2 -1\otimes 1^0 ||<\frac{1}{(2L)^3 }+L||u_2 -1\otimes 1^0 \otimes 1^0 ||<\frac{1}{(2L)^3 }+\frac{1}{2^2 }
<\frac{1}{2^3 }+\frac{1}{2^2}=\frac{3}{2^3 } .
\end{align*}
It follows that by induction that there are sequences $\{(\rho_n , u_n )\}$ of twisted coactions of $H^0 $ on $A$
and $\{y_n \}$ of unitary elements in $A\otimes H^0 $ satisfying that for any $n\in \BN$,
$$
||1\otimes 1^0 \otimes 1^0 -u_n ||<\frac{1}{(2L)^n }<\frac{1}{2^n} , \quad
||1\otimes 1^0 -y_n ||<\frac{1}{2^{n+1}}+\frac{1}{2^n }=\frac{3}{2^{n+1}} .
$$
Let $x_n =y_n y_{n-1}\cdots y_0 \in A\otimes H^0 $ for any $n\in \BN\cup\{0\}$.
Then $x_n $ is a unitary element in $A\otimes H^0 $ satisfying that
$$
u_{n+1}=(x_n \otimes 1^0 )(\rho\otimes \id)(x_n )u_0 (\id\otimes\Delta^0 )(x_n^* )
$$
for any $\in\BN\cup\{0\}$ by routine computations. Furthermore,
$$
||u_n -1\otimes 1^0 \otimes 1^0 ||<\frac{1}{2^n }\longrightarrow 0 \quad (n \longrightarrow+\infty)
$$
Also, since by easy computations, we see that $\{x_n \}$ is a Cauchy sequence, there is a unitary element
$x\in A\otimes H^0 $ such that $x_n \longrightarrow x$ $(n\longrightarrow +\infty)$.
Therefore, we obtain that
$$
1\otimes 1^0 \otimes 1^0 =(x\otimes 1^0 )(\rho\otimes\id)(x)u(\id\otimes\Delta^0 )(x^* ) .
$$
\end{proof}

\section{Approximately unitary equivalence of coactions}\label{sec:appro}
Let $\rho$ be a coaction of $H^0$ on $A$ with the Rohlin property.
Let $w$ be a unitary element in $(A\rtimes_{\rho}H)\otimes H$ satisfying Equations
(5, 1)-(5, 3) for $\widehat{\rho}$.
Let $(\rho_1 , u )$ be a twisted coaction $H^0$ on $A$ which is exterior equivalent to $\rho$.
Let $v$ be a unitary element in $A\otimes H^0$ satisfying Conditions (1), (2) in
Definition \ref{Def:equivalence}, that is,
\newline
(1) $\rho_1 =\Ad(v)\circ\rho $,
\newline
(2) $u=(v\otimes 1)(\rho \otimes\id)(v)(\id\otimes\Delta)(v^* )$.
\newline
By Proposition \ref{prop:exterior2},
$(\rho_1 , u)$ has the Rohlin property. Let $w_1$ be a unitary element in
$(A\rtimes_{\rho_1 , u}H)\otimes H$ satisfying Equations (5, 1)-(5, 3)
for $\widehat{\rho_1}$. By Lemma \ref{lem:projection2},
$\widehat{w}(\tau)=\widehat{w_1}(\tau)$. Let
$$
x=N(\id\otimes e)(v\rho^{\infty}(\widehat{w}(\tau)))=N\widehat{v}(e_{(1)})[e_{(2)}\cdot_{\rho^{\infty}}\widehat{w}(\tau)] .
$$
We have the following lemma which is similar to Lemma \ref{lem:unitary2}.

\begin{lemma}\label{lem:unitary}With the above notations and assumptions, $x$ is a unitary element in $A^{\infty}$.
\end{lemma}
\begin{proof}In the same way as in the proof of Lemma \ref{lem:unitary2},
we can see that $xx^* =1$. Next we shall show that $x^* x=1$.
Let $f=e$.
\begin{align*}
& x^* x=N^2 [e_{(2)}\cdot_{\rho^{\infty}}\widehat{w}(\tau)]^* \widehat{v}(e_{(1)})^*
\widehat{v}(f_{(1)})[f_{(2)}\cdot_{\rho^{\infty}}\widehat{w}(\tau)] \\
& =N^2 [S(e_{(2)}^* )\cdot_{\rho^{\infty}}\widehat{w}(\tau)]\widehat{v^* }(S(e_{(1)})^* )
\widehat{v}(f_{(1)})[f_{(2)}\cdot_{\rho^{\infty}}\widehat{w}(\tau)] \\
& =N^2 \widehat{V}(S(e_{(3)}^* ))\widehat{w}(\tau)\widehat{V^* }(S(e_{(2)}^* ))\widehat{v^* }(S(e_{(1)}^* ))
\widehat{v}(f_{(1)})\widehat{V}(f_{(2)})\widehat{w}(\tau)\widehat{V^* }(f_{(3)}) \\
& =N^2 \widehat{V}(S(e_{(3)}^* ))\widehat{w}(\tau)(1\rtimes_{\rho}e_{(2)}^* )\widehat{v^* }(S(e_{(1)}^* ))
\widehat{v}(f_{(1)})(1\rtimes_{\rho}f_{(2)})\widehat{w}(\tau)\widehat{V^* }(f_{(3)}) \\
& =N^2 \widehat{V}(S(e_{(4)}^* ))\widehat{w}(\tau)([e_{(2)}^* \cdot_{\rho}\widehat{v^* }(S(e_{(1)}^* ))
\widehat{v}(f_{(1)})]\rtimes_{\rho}e_{(3)}^* f_{(2)})\widehat{w}(\tau)\widehat{V^* }(f_{(3)}) \\
& =N^2 \widehat{V}(S(e_{(4)}^* ))[e_{(2)}^* \cdot_{\rho}\widehat{v^* }(S(e_{(1)}^* ))
\widehat{v}(f_{(1)})]\tau(e_{(3)}^* f_{(2)})\widehat{w}(\tau)\widehat{V^* }(f_{(3)}) \\
& =N^2 \widehat{V}(S(e_{(6)}^* ))[e_{(2)}^* \cdot_{\rho}\widehat{v^* }(S(e_{(1)}^* ))
\widehat{v}(S(e_{(3)}^* )e_{(4)}^* f_{(1)})]\tau(e_{(5)}^* f_{(2)})\widehat{w}(\tau)\widehat{V^* }(f_{(3)}) \\
& =N^2 \widehat{V}(S(e_{(5)}^* ))[e_{(2)}^* \cdot_{\rho}\widehat{v^* }(S(e_{(1)}^* ))
\widehat{v}(S(e_{(3)}^* ))]\tau(e_{(4)}^* f_{(1)})\widehat{w}(\tau)\widehat{V^* }(f_{(2)}) \\
& =N^2 \widehat{V}(S(e_{(7)}^* ))[e_{(2)}^* \cdot_{\rho}\widehat{v^* }(S(e_{(1)}^* ))
\widehat{v}(S(e_{(3)}^* ))]\tau(e_{(4)}^* f_{(1)})\widehat{w}(\tau)\widehat{V^* }(S(e_{(6)}^* )e_{(5)}^* f_{(2)}) \\
& =N^2 \widehat{V}(S(e_{(6)}^* ))[e_{(2)}^* \cdot_{\rho}\widehat{v^* }(S(e_{(1)}^* ))
\widehat{v}(S(e_{(3)}^* ))]\tau(e_{(4)}^* f)\widehat{w}(\tau)\widehat{V^* }(S(e_{(5)}^* )) \\
& =N \widehat{V}(S(e_{(5)}^* ))[e_{(2)}^* \cdot_{\rho}\widehat{v^* }(S(e_{(1)}^* ))
\widehat{v}(S(e_{(3)}^* ))]\widehat{w}(\tau)\widehat{V^* }(S(e_{(4)}^* )) \\
& =N [S(e_{(4)}^* )\cdot_{\rho^{\infty}}[e_{(2)}^* \cdot_{\rho}\widehat{v^* }(S(e_{(1)}^* ))
\widehat{v}(S(e_{(3)}^* ))]\widehat{w}(\tau)] .
\end{align*}
Let $E^{\rho^{\infty}}$ be the conditional expectation from $A^{\infty}$ onto $(A^{\rho})^{\infty}$.
Then since $e=\sum_{i,k}\frac{d_k}{N}w_{ii}^k$,
\begin{align*}
E^{\rho^{\infty}}(x^* x) & =f\cdot_{\rho^{\infty}}x^* x 
=N [fS(e_{(4)}^* )\cdot_{\rho^{\infty}}[e_{(2)}^* \cdot_{\rho}\widehat{v^* }(S(e_{(1)}^* ))
\widehat{v}(S(e_{(3)}^* ))]\widehat{w}(\tau)] \\
& =N [f\cdot_{\rho^{\infty}}[e_{(2)}^* \cdot_{\rho}\widehat{v^* }(S(e_{(1)}^* ))
\widehat{v}(S(e_{(3)}^* ))]\widehat{w}(\tau)] \\
& =\sum_{i,j,j_1 ,k}d_k [f\cdot_{\rho^{\infty}}[w_{jj_1}^{k*}\cdot_{\rho}\widehat{v^*}(S(w_{ij}^{k*}))
\widehat{v}(S(w_{j_1 i}^{k*}))]\widehat{w}(\tau)] \\
& =\sum_{i,j,j_1 ,k}d_k [f\cdot_{\rho^{\infty}}[w_{jj_1}^{k*}\cdot_{\rho}\widehat{v^*}(w_{ji}^k)
\widehat{v}(w_{ij_1}^k)]\widehat{w}(\tau)] \\
& =\sum_{j,j_1 ,k}d_k [f\cdot_{\rho^{\infty}}[w_{jj_1}^{k*}\cdot_{\rho}\epsilon(w_{jj_1}^k )]
\widehat{w}(\tau)]
=\sum_{j,k}d_k [f\cdot_{\rho^{\infty}}[w_{jj}^{k*}\cdot_{\rho}1]\widehat{w}(\tau)] \\
& =N[f\cdot_{\rho^{\infty}}\epsilon(e)\widehat{w}(\tau)]=N[f\cdot_{\rho^{\infty}}\widehat{w}(\tau)]=1
\end{align*}
by Lemma \ref{lem:projection}.
Since $E^{\rho^{\infty}}$ is faithful, we obtain the conclusion.
\end{proof}

\begin{Def}\label{Def:approximate}
Coactions $\rho$ and $\sigma$ of $H^0$ on $A$ are
\sl
approximately unitarily equivalent
\rm
if there is a unitary element $v\in A^{\infty}\otimes H^0$ such that
$$
\sigma(a)=v\rho(a)v^*
$$
for any $a\in A$.
\end{Def}

Let $\rho$ and $\sigma$ be coactions of $H^0$ on $A$ which are approximately unitarily
equivalent. Then there is a unitary element $v$ in $A^{\infty}\otimes H^0$ such that
$\sigma(a)=v\rho(a)v^*$ for any $a\in A$.
We write $v=(v_n )$, where $v_n$ is a unitary element in $A$.
Then since $a(\id\otimes\epsilon^0 )(v)=(\id\otimes\epsilon^0 )(v)a$ for
any $a\in A$, $(\id\otimes\epsilon^0 )(v)$ is a unitary element in $A_{\infty}$.
Let $z=(\id\otimes\epsilon^0 )(v)$ and $w=v(z^* \otimes 1^0)$.
Then $w$ is a unitary element in $A^{\infty}\otimes H^0 $ and
$$
w\rho(a)w^* =v(z^* \otimes 1^0 )\rho(a)(z\otimes 1^0 )v^* =v\rho(a)v^* =\sigma(a)
$$
for any $a\in A$. Furthermore, $(\id\otimes\epsilon^0 )(w)=zz^*=1$.
Hence if we write $w=(w_n )$, where $w_n $ is a unitary element in $A\otimes H^0 $,
then $w_n =v_n ((\id\otimes \epsilon^0 )(v_n^*)\otimes1^0 )$. Thus
$(\id\otimes\epsilon)(w_n )=1$. Therefore, we may assume that $(\id\otimes\epsilon^0 )(v_n )=1$
for any $n\in\BN$.
We shall show the following lemma.

\begin{lemma}\label{lem:key0}Let $\sigma$ and $\rho$ be coactions of $H^0$ on $A$.
We suppose that $\rho$ has the Rohlin property and that $\sigma$ is approximately unitary
equivalent to $\rho$. Then for each finite subset $F$ of $A$ and any positive number $\epsilon>0$,
there is a unitary element $x\in A$ such that
\begin{align*}
& ||\sigma(a)-(\Ad(x\otimes 1^0 )\circ\rho\circ\Ad(x^* ))(a)||<\epsilon, \\
& ||xa-ax||<\epsilon+L\max_{i,j,k}||\sigma([S(w_{ij}^k )\cdot_{\rho}a])-\rho([S(w_{ij}^k )\cdot_{\rho}a])||
\end{align*}
for any $a\in F$, where $L=\sum_{i,j,k}d_k ||\id\otimes w_{ij}^k ||$.
\end{lemma}

We shall prove this lemma by showing a series of several lemmas.
Since $\rho$ and $\sigma$ are approximately unitarily equivalent,
there is a unitary element $v_0 \in A^{\infty}\otimes H^0$ such that
$\sigma(a)=v_0 \rho(a)v_0^* $ for any $a\in A$. Let $F$ be any finite subset
of $A$ and $\epsilon$ any positive number. Then there is a unitary element
$v\in A\otimes H^0$ with $(\id\otimes\epsilon^0 )(v)=1$ such that
\begin{align*}
& ||\sigma(a)-v\rho(a)v^* ||<\epsilon, \\
& ||\sigma([S(w_{ij}^k )\cdot_{\sigma} a])-v\rho([S(w_{ij}^k )\cdot_{\sigma}a])v^* ||<\epsilon , \\
& ||\sigma([S(w_{ij}^k )\cdot_{\rho}a])-v\rho([S(w_{ij}^k )\cdot_{\rho}a])v^* ||<\epsilon
\end{align*}
for any $a\in F$ and $i,j=1,2,\dots,d_k$, $k=1,2,\dots,K$. Let $x=N(\id\otimes e)(v\rho^{\infty}(\widehat{w}(\tau)))$.
Let $\rho_1 =\Ad(v)\circ\rho$ and $u=(v\otimes 1^0 )(\rho\otimes\id)(v)(\id\otimes\Delta^0 )(v^* )$.
Then $(\rho_1 , u)$ is a twisted coaction of $H^0$ on $A$ which is exterior
equivalent to $\rho$. Hence by Lemma \ref{lem:unitary}, $x$ is a unitary element in $A^{\infty}$.

\begin{lemma}\label{lem:key1}With the above notations and assumptions, for any $a\in F$,
$$
||\rho(x)(x^* \otimes 1^0 )v\rho(a)-N(\id\otimes e)((\rho\otimes\id)(v)(\id\otimes\Delta^0 )
(\rho^{\infty}(\widehat{w}(\tau))v^* ))\sigma(a)v||<N\epsilon.
$$
\end{lemma}
\begin{proof}
We note that
\begin{align*}
x & =N(\id\otimes e)(v\rho^{\infty}(\widehat{w}(\tau))) = \sum_{i,k}d_k (\id\otimes w_{ii}^k )(v\rho^{\infty}(\widehat{w}(\tau))) \\
& =\sum_{i,j,k}d_k \widehat{v}(w_{ij}^k )[w_{ji}^k \cdot_{\rho^{\infty}}\widehat{w}(\tau)].
\end{align*}
Also, $x^* =N[e_{(1)}\cdot_{\rho^{\infty}}\widehat{w}(\tau)]\widehat{v^* }(e_{(2)})$
since $x=N(\id\otimes S(e^* ))(v\rho^{\infty}(\widehat{w}(\tau)))$.
Then by Lemma \ref{lem:homo} for any $h\in H$,
\begin{align*}
& (\rho(x)(x^* \otimes 1^0 )v\rho(a))^{\widehat{}}(h)=[h_{(1)}\cdot_{\rho^{\infty}}x]x^* \widehat{v}(h_{(2)})[h_{(3)}\cdot_{\rho}a] \\
& =N\sum_{i,j,k}d_k [h_{(1)}\cdot_{\rho}\widehat{v}(w_{ij}^k )][h_{(2)}w_{ji}^k \cdot_{\rho^{\infty}}\widehat{w}(\tau)]
[e_{(1)}\cdot_{\rho^{\infty}}\widehat{w}(\tau)]\widehat{v^* }(e_{(2)})\widehat{v}(h_{(3)})[h_{(4)}\cdot_{\rho}a] \\
& =N\sum_{i,j,k,t}d_k [h_{(1)}\cdot_{\rho}\widehat{v}(w_{ij}^k )]\widehat{V}(h_{(2)}w_{jt}^k )\widehat{w}(\tau)\widehat{V^* }
(h_{(3)}w_{ti}^k )\widehat{V}(e_{(1)})\widehat{w}(\tau)\widehat{V^* }(e_{(2)})\widehat{v^* }(e_{(3)}) \\
& \times\widehat{v}(h_{(4)})[h_{(5)}\cdot_{\rho}a] \\
& =N\sum_{i,j,k,t}d_k [h_{(1)}\cdot_{\rho}\widehat{v}(w_{ij}^k )]\widehat{V}(h_{(2)}w_{jt}^k )\widehat{w}(\tau)
(1\rtimes_{\rho}S(h_{(3)}w_{ti}^k )e_{(1)})\widehat{w}(\tau)\widehat{V^* }(e_{(2)})\widehat{v^* }(e_{(3)}) \\
& \times\widehat{v}(h_{(4)})[h_{(5)}\cdot_{\rho}a] \\
& =N\sum_{i,j,k,t}d_k [h_{(1)}\cdot_{\rho}\widehat{v}(w_{ij}^k )]\widehat{V}(h_{(2)}w_{jt}^k )
\tau(S(h_{(3)}w_{ti}^k )e_{(1)})\widehat{w}(\tau)\widehat{V^* }(e_{(2)})\widehat{v^* }(e_{(3)}) \\
& \times\widehat{v}(h_{(4)})[h_{(5)}\cdot_{\rho}a] \\
& =N\sum_{i,j,k,t, t_1 , t_2 }d_k [h_{(1)}\cdot_{\rho}\widehat{v}(w_{ij}^k )]\widehat{V}(h_{(2)}w_{jt}^k )
\widehat{w}(\tau)\widehat{V^* }(h_{(3)}w_{tt_1}^k S(h_{(4)}w_{t_1 t_2}^k )e_{(2)}) \\
& \times\tau(S(h_{(5)}w_{t_2 i}^k )e_{(1)})\widehat{v^* }(e_{(3)})
\widehat{v}(h_{(6)})[h_{(7)}\cdot_{\rho}a] \\
& =N\sum_{i,j,k,t, t_1 }d_k [h_{(1)}\cdot_{\rho}\widehat{v}(w_{ij}^k )]\widehat{V}(h_{(2)}w_{jt}^k )
\widehat{w}(\tau)\widehat{V^* }(h_{(3)}w_{tt_1}^k )
\tau(S(h_{(4)}w_{t_1 i}^k )e_{(1)})\widehat{v^* }(e_{(2)}) \\
& \times\widehat{v}(h_{(5)})[h_{(6)}\cdot_{\rho}a] \\
& =N\sum_{i,j,k,t, t_1 ,t_2 , t_3 }d_k [h_{(1)}\cdot_{\rho}\widehat{v}(w_{ij}^k )]\widehat{V}(h_{(2)}w_{jt}^k )
\widehat{w}(\tau)\widehat{V^* }(h_{(3)}w_{tt_1}^k ) \\
& \times\widehat{v^* }(h_{(4)}w_{t_1  t_2 }^k S(h_{(5)}w_{t_2 t_3 }^k )e_{(2)})\tau(S(h_{(6)}w_{t_3 i}^k )e_{(1)})
\widehat{v}(h_{(7)})[h_{(8)}\cdot_{\rho}a] \\
& =N\sum_{i,j,k,t, t_1 ,t_2 }d_k [h_{(1)}\cdot_{\rho}\widehat{v}(w_{ij}^k )]\widehat{V}(h_{(2)}w_{jt}^k )
\widehat{w}(\tau)\widehat{V^* }(h_{(3)}w_{tt_1}^k )
\widehat{v^* }(h_{(4)}w_{t_1  t_2 }^k ) \\
& \times\tau(S(h_{(5)}w_{t_2  i}^k )e)
\widehat{v}(h_{(6)})[h_{(7)}\cdot_{\rho}a] \\
& =\sum_{i,j,k,t, t_1 }d_k [h_{(1)}\cdot_{\rho}\widehat{v}(w_{ij}^k )]\widehat{V}(h_{(2)}w_{jt}^k )
\widehat{w}(\tau)\widehat{V^* }(h_{(3)}w_{tt_1}^k )
\widehat{v^* }(h_{(4)}w_{t_1  i}^k )
\widehat{v}(h_{(5)})[h_{(6)}\cdot_{\rho}a] \\
& =\sum_{i,j,k, t_1 }d_k [h_{(1)}\cdot_{\rho}\widehat{v}(w_{ij}^k )][h_{(2)}w_{jt_1 }^k \cdot_{\rho^{\infty}}
\widehat{w}(\tau)]
\widehat{v^* }(h_{(3)}w_{t_1  i}^k )
\widehat{v}(h_{(4)})[h_{(5)}\cdot_{\rho}a] .
\end{align*}
Thus
\begin{align*}
\rho(x)(x^* \otimes 1^0 )v\rho(a) & =\sum_{i,k}d_k (\id\otimes w_{ii}^k )((\rho\otimes\id)(v)
(\id\otimes\Delta^0 )(\rho^{\infty}(\widehat{w}(\tau))v^* ))v\rho(a) \\
& =N(\id\otimes e)((\rho\otimes\id)(v)(\id\otimes\Delta^0 )(\rho^{\infty}(\widehat{w}(\tau))v^* ))v\rho(a) .
\end{align*}
Hence
\begin{align*}
& ||\rho(x)(x^* \otimes 1^0 )v\rho(a)-
N(\id\otimes e)((\rho\otimes\id)(v)(\id\otimes\Delta^0 )(\rho^{\infty}(\widehat{w}(\tau))v^* ))\sigma(a)v|| \\
& =N||(\id\otimes e)((\rho\otimes\id)(v)(\id\otimes\Delta^0 )(\rho^{\infty}(\widehat{w}(\tau))v^* ))(v\rho(a)-\sigma(a)v)|| \\
&\le N ||v\rho(a)-\sigma(a)v||=N||v\rho(a)v^* -\sigma(a)||<N\epsilon .
\end{align*}
\end{proof}

\begin{lemma}\label{lem:key2}With the above notations and assumptions, for any $a\in F$,
\begin{align*}
& ||N(\id\otimes e)((\rho\otimes\id)(v)(\id\otimes\Delta^0 )(\rho^{\infty}(\widehat{w}(\tau))v^* ))\sigma(a)v \\
& -\sum_{i,j,k}d_k (\id\otimes w_{ij}^k )((\rho\otimes\id)(v)(\id\otimes\Delta^0 )(\rho^{\infty}(\widehat{w}(\tau))\rho
([S(w_{ji}^k )\cdot_{\sigma}a])v^* ))v||<L \epsilon,
\end{align*}
where
$L=\sum_{i,j,k} d_k ||\id\otimes w_{ij}^k ||$.
\end{lemma}
\begin{proof}
Since $e=\sum_{i,k}\frac{d_k }{N}w_{ii}^k $,
\begin{align*}
& N(\id\otimes e)((\rho\otimes\id)(v)(\id\otimes\Delta^0 )(\rho^{\infty}(\widehat{w}(\tau))v^* ))\sigma(a)v \\
& =\sum_{i,k}d_k (\id\otimes w_{ii}^k )((\rho\otimes\id)(v)(\id\otimes\Delta^0 )(\rho^{\infty}(\widehat{w}(\tau))v^* ))
\sigma(a)v .
\end{align*}
Thus for any $h\in H$
\begin{align*}
& [N(\id\otimes e)((\rho\otimes\id)(v)(\id\otimes\Delta^0 )(\rho^{\infty}(\widehat{w}(\tau))v^* ))\sigma(a)v]^{\widehat{}}(h) \\
& =\sum_{i,j,k,t_1 }d_k [h_{(1)}\cdot_{\rho}\widehat{v}(w_{ij}^k )][h_{(2)}w_{jt_1 }^k \cdot_{\rho^{\infty}}\widehat{w}(\tau)]
\widehat{v^* }(h_{(3)}w_{t_1 i}^k )[h_{(4)}\cdot_{\sigma}a]\widehat{v}(h_{(5)}) \\
& =\sum_{i,j,k,t_1 , t_2 }d_k [h_{(1)}\cdot_{\rho}\widehat{v}(w_{ij}^k )][h_{(2)}w_{jt_1 }^k \cdot_{\rho^{\infty}}\widehat{w}(\tau)]
\widehat{v^* }(h_{(3)}w_{t_1 t_2 }^k )[h_{(4)}\epsilon(w_{t_2 i}^k )\cdot_{\sigma}a]\widehat{v}(h_{(5)}) \\
& =\sum_{i,j,k,t_1 , t_2 , t_3 }d_k [h_{(1)}\cdot_{\rho}\widehat{v}(w_{ij}^k )]
[h_{(2)}w_{jt_1 }^k \cdot_{\rho^{\infty}}\widehat{w}(\tau)]
\widehat{v^* }(h_{(3)}w_{t_1 t_2 }^k ) \\
& \times [h_{(4)}w_{t_2 t_3 }^k \cdot_{\sigma}[S(w_{t_3 i}^k )\cdot_{\sigma}a]]
\widehat{v}(h_{(5)}) .
\end{align*}
Thus
\begin{align*}
& N(\id\otimes e)((\rho\otimes\id)(v)(\id\otimes\Delta^0 )(\rho^{\infty}(\widehat{w}(\tau))v^* ))\sigma(a)v \\
& =\sum_{i,t_3 ,k}d_k (\id\otimes w_{it_3 }^k )((\rho\otimes\id)(v)(\id\otimes\Delta^0 )(\rho^{\infty}(\widehat{w}(\tau))
v^*\sigma([S(w_{t_3 i}^k )\cdot_{\sigma}a])))v .
\end{align*}
Hence
\begin{align*}
& ||N(\id\otimes e)((\rho\otimes\id)(v)(\id\otimes\Delta^0 )(\rho^{\infty}(\widehat{w}(\tau))v^* ))\sigma(a)v \\
& -\sum_{i,t_3 ,k}d_k (\id\otimes w_{it_3 }^k )((\rho\otimes\id)(v)(\id\otimes\Delta^0 )(\rho^{\infty}(\widehat{w}(\tau))
\rho([S(w_{t_3 i}^k )\cdot_{\sigma}a])v^* ))v|| \\
& =||\sum_{i,t_3 ,k}d_k (\id\otimes w_{it_3 }^k )((\rho\otimes\id)(v)(\id\otimes\Delta^0 )(\rho^{\infty}(\widehat{w}(\tau))
[v^* \sigma([S(w_{t_3 i}^k )\cdot_{\sigma}a]) \\
& -\rho([S(w_{t_3 i}^k )\cdot_{\sigma}a])v^* ))v|| \\
& \le \sum_{i,t_3 ,k}d_k ||\id\otimes w_{it_3 }^k || \, ||v^* \sigma([S(w_{t_3 i}^k )\cdot_{\sigma}a])
-\rho([S(w_{t_3 i}^k )\cdot_{\sigma}a])v^* || \\
& <\sum_{i, t_3 , k}d_k ||\id\otimes w_{it_3}^k ||\epsilon<L \epsilon .
\end{align*}
\end{proof}

\begin{lemma}\label{lem:key3}With the above notations and assumptions, for any $a\in A$,
\begin{align*}
& \sum_{i,j,k}d_k (\id\otimes w_{ij}^k )((\rho\otimes\id)(v)(\id\otimes\Delta^0 )
(\rho^{\infty}(\widehat{w}(\tau))\rho([S(w_{ji}^k )\cdot_{\sigma}a]v^* ))v \\
& =\rho(a)\rho^{\infty}(x)(x^* \otimes1^0 )v .
\end{align*}
\begin{proof}
We shall show the above equation by routine computations. For any $h\in H$
\begin{align*}
& [\sum_{i,t_3 ,k}d_k (\id\otimes w_{it_3 }^k )((\rho\otimes\id)(v)(\id\otimes\Delta^0 )
(\rho^{\infty}(\widehat{w}(\tau))\rho([S(w_{t_3 i}^k )\cdot_{\sigma}a]v^* ))v]^{\widehat{}}(h) \\
& =\sum_{i,j,k,t_1 , t_2 , t_3 }d_k [h_{(1)}\cdot_{\rho}\widehat{v}(w_{ij}^k )]
[h_{(2)}w_{jt_1 }^k \cdot_{\rho^{\infty}}\widehat{w}(\tau)][h_{(3)}w_{t_1 t_2 }^k\cdot_{\rho}[S(w_{t_3 i}^k )\cdot_{\sigma}a]] \\
& \times\widehat{v^* }(h_{(4)}w_{t_2 t_3 }^k )\widehat{v}(h_{(5)}) \\
& =\sum_{i,j,k , t_2 , t_3 }d_k [h_{(1)}\cdot_{\rho}\widehat{v}(w_{ij}^k )]
[h_{(2)}w_{jt_2 }^k \cdot_{\rho^{\infty}}\widehat{w}(\tau)[S(w_{t_3 i}^k )\cdot_{\sigma}a]]
\widehat{v^* }(h_{(3)}w_{t_2 t_3 }^k )\widehat{v}(h_{(4)}) \\
& =\sum_{i,j,k , t_2 , t_3 }d_k [h_{(1)}\cdot_{\rho}\widehat{v}(w_{ij}^k )]
[h_{(2)}w_{jt_2 }^k \cdot_{\rho^{\infty}}[S(w_{t_3 i}^k )\cdot_{\sigma}a]\widehat{w}(\tau)]
\widehat{v^* }(h_{(3)}w_{t_2 t_3 }^k )\widehat{v}(h_{(4)}) \\
& =\sum_{i,j,k , t_2 , t_3 }d_k [h_{(1)}\cdot_{\rho}\widehat{v}(w_{ij}^k )
[w_{jt_2 }^k \cdot_{\rho^{\infty}}[S(w_{t_3 i}^k )\cdot_{\sigma}a]\widehat{w}(\tau)]]
\widehat{v^* }(h_{(2)}w_{t_2 t_3 }^k )\widehat{v}(h_{(3)}) \\
& =\sum_{i,j,k , t_2 , t_3 , j_1 }d_k [h_{(1)}\cdot_{\rho}\widehat{v}(w_{ij}^k )
[w_{jj_1 }^k \cdot_{\rho}[S(w_{t_3 i}^k )\cdot_{\sigma}a]][w_{j_1 t_2 }^k \cdot_{\rho^{\infty}}\widehat{w}(\tau)]]
\widehat{v^* }(h_{(2)}w_{t_2 t_3 }^k ) \\
& \times\widehat{v}(h_{(3)}) \\
& =\sum_{i,j,k , t_2 , t_3 , j_1 }d_k [h_{(1)}\cdot_{\rho}[w_{ij}^k \cdot_{\sigma}
[S(w_{t_3 i}^k )\cdot_{\sigma}a]]\widehat{v}(w_{jj_1}^k )
[w_{j_1 t_2 }^k \cdot_{\rho^{\infty}}\widehat{w}(\tau)]]
\widehat{v^* }(h_{(2)}w_{t_2 t_3 }^k ) \\
& \times\widehat{v}(h_{(3)}) \\
& =\sum_{j, k , t_2 , t_3 , j_1 }d_k [h_{(1)}\cdot_{\rho}[\epsilon(w_{t_3 j}^k )\cdot_{\sigma}a]\widehat{v}(w_{jj_1}^k )
[w_{j_1 t_2 }^k \cdot_{\rho^{\infty}}\widehat{w}(\tau)]]
\widehat{v^* }(h_{(2)}w_{t_2 t_3 }^k )\widehat{v}(h_{(3)}) \\
& =\sum_{k , t_2 , t_3 , j_1 }d_k [h_{(1)}\cdot_{\rho}a\widehat{v}(w_{t_3 j_1}^k )
[w_{j_1 t_2 }^k \cdot_{\rho^{\infty}}\widehat{w}(\tau)]]
\widehat{v^* }(h_{(2)}w_{t_2 t_3 }^k )\widehat{v}(h_{(3)}) \\
& =\sum_{k , t_2 , t_3 , j_1 }d_k [h_{(1)}\cdot_{\rho}a][h_{(2)}\cdot_{\rho}\widehat{v}(w_{t_3 j_1}^k )]
[h_{(3)}w_{j_1 t_2 }^k \cdot_{\rho^{\infty}}\widehat{w}(\tau)]]
\widehat{v^* }(h_{(4)}w_{t_2 t_3 }^k )\widehat{v}(h_{(5)}) .
\end{align*}
On the other hand by Lemma \ref{lem:homo} for any $h\in H$,
\begin{align*}
& [\rho(a)\rho^{\infty}(x)(x^* \otimes 1^0 )v]^{\widehat{}}(h)
=[h_{(1)}\cdot_{\rho}a][h_{(2)}\cdot_{\rho^{\infty}}x]x^* \widehat{v}(h_{(3)}) \\
& =N\sum_{i,j,k}d_k [h_{(1)}\cdot_{\rho}a][h_{(2)}\cdot_{\rho}\widehat{v}(w_{ij}^k )][h_{(3)}w_{ji}^k \cdot_{\rho^{\infty}}
\widehat{w}(\tau)][e_{(1)}\cdot_{\rho^{\infty}}\widehat{w}(\tau)]\widehat{v^* }(e_{(2)})\widehat{v}(h_{(4)}) \\
& =N\sum_{i,j,k, i_1 }d_k [h_{(1)}\cdot_{\rho}a][h_{(2)}\cdot_{\rho}\widehat{v}(w_{ij}^k )]
\widehat{V}(h_{(3)}w_{ji_1 }^k )
\widehat{w}(\tau)(1\rtimes_{\rho}S(h_{(4)}w_{i_1 i}^k ))(1\rtimes_{\rho}e_{(1)}) \\
& \times\widehat{w}(\tau)\widehat{V^* }(e_{(2)})\widehat{v^* }(e_{(3)})\widehat{v}(h_{(5)}) \\
& =N\sum_{i,j,k, i_1 }d_k [h_{(1)}\cdot_{\rho}a][h_{(2)}\cdot_{\rho}\widehat{v}(w_{ij}^k )]
\widehat{V}(h_{(3)}w_{ji_1 }^k )
\tau(S(h_{(4)}w_{i_1 i}^k )e_{(1)})
\widehat{w}(\tau)\widehat{V^* }(e_{(2)}) \\
& \times\widehat{v^* }(e_{(3)})\widehat{v}(h_{(5)}) \\
& =N\sum_{i,j,k, i_1 , i_2 , i_3 }d_k [h_{(1)}\cdot_{\rho}a][h_{(2)}\cdot_{\rho}\widehat{v}(w_{ij}^k )]
\widehat{V}(h_{(3)}w_{ji_1 }^k )\widehat{w}(\tau)
\tau(S(h_{(6)}w_{i_3 i}^k )e_{(1)}) \\
& \times\widehat{V^* }(h_{(4)}w_{i_1 i_2 }^k S(h_{(5)}w_{i_2 i_3 }^k )e_{(2)})
\widehat{v^* }(e_{(3)})\widehat{v}(h_{(7)}) \\
& =N\sum_{i,j,k, i_1 , i_2 }d_k [h_{(1)}\cdot_{\rho}a][h_{(2)}\cdot_{\rho}\widehat{v}(w_{ij}^k )]
\widehat{V}(h_{(3)}w_{ji_1 }^k )\widehat{w}(\tau)
\tau(S(h_{(5)}w_{i_2 i}^k )e_{(1)}) \\
& \times\widehat{V^* }(h_{(4)}w_{i_1 i_2 }^k )
\widehat{v^* }(e_{(2)})\widehat{v}(h_{(6)})
\end{align*}
\begin{align*}
& =N\sum_{i,j,k, i_1 , i_2 , t, t_1 }d_k [h_{(1)}\cdot_{\rho}a][h_{(2)}\cdot_{\rho}\widehat{v}(w_{ij}^k )]
\widehat{V}(h_{(3)}w_{ji_1 }^k )\widehat{w}(\tau)\widehat{V^* }(h_{(4)}w_{i_1 i_2 }^k ) \\
& \times\tau(S(h_{(7)}w_{t_1 i}^k )e_{(1)})
\widehat{v^* }(h_{(5)}w_{i_2 t}^k S(h_{(6)}w_{t t_1 }^k )e_{(2)})\widehat{v}(h_{(8)}) \\
& =N\sum_{i,j,k, i_1 , i_2 , t}d_k [h_{(1)}\cdot_{\rho}a][h_{(2)}\cdot_{\rho}\widehat{v}(w_{ij}^k )]
\widehat{V}(h_{(3)}w_{ji_1 }^k )\widehat{w}(\tau)\widehat{V^* }(h_{(4)}w_{i_1 i_2 }^k ) \\
& \times\tau(S(h_{(6)}w_{t i}^k )e)
\widehat{v^* }(h_{(5)}w_{i_2 t}^k )\widehat{v}(h_{(7)}) \\
& =\sum_{j,k, i_1 , i_2 , t}d_k [h_{(1)}\cdot_{\rho}a][h_{(2)}\cdot_{\rho}\widehat{v}(w_{tj}^k )]
\widehat{V}(h_{(3)}w_{ji_1 }^k )\widehat{w}(\tau)\widehat{V^* }(h_{(4)}w_{i_1 i_2 }^k )
\widehat{v^* }(h_{(5)}w_{i_2 t}^k ) \\
& \times\widehat{v}(h_{(6)}) \\
& =\sum_{j,k, i_1 , i_2 , t}d_k [h_{(1)}\cdot_{\rho}a][h_{(2)}\cdot_{\rho}\widehat{v}(w_{tj}^k )]
[h_{(3)}w_{ji_2 }^k \cdot_{\rho^{\infty}}\widehat{w}(\tau)]\widehat{v^* }(h_{(4)}w_{i_2 t}^k )\widehat{v}(h_{(5)}) .
\end{align*}
Therefore, we obtain the conclusion.
\end{proof}
\end{lemma}

\begin{lemma}\label{lem:key4}With the above notations and assumptions, for any $a\in F$,
$$
||xa-ax||<L\epsilon+L\max_{i,j,k}||\sigma([S(w_{ij}^k )\cdot_{\rho}a])-\rho([S(w_{ij}^k )\cdot_{\rho}a])||,
$$
where $L=\sum_{i,j,k}d_k ||\id\otimes w_{ij}^k ||$.
\end{lemma}
\begin{proof}
For any $a\in F$
\begin{align*}
& xax^* =N\sum_{i,j,k}d_k \widehat{v}(w_{ij}^k )[w_{ji}^k \cdot_{\rho^{\infty}}\widehat{w}(\tau)]
a[e_{(1)}\cdot_{\rho^{\infty}}\widehat{w}(\tau)]\widehat{v^* }(e_{(2)}) \\
& =N\sum_{i,j,k,i_1 }d_k \widehat{v}(w_{ij}^k )\widehat{V}(w_{ji_1 })\widehat{w}(\tau)(1\rtimes_{\rho}S(w_{i_1 i}^k ))
(a\rtimes_{\rho}e_{(1)})\widehat{w}(\tau)\widehat{V^* }(e_{(2)})\widehat{v^* }(e_{(3)}) \\
& =N\sum_{i,j,k,i_1 , i_2 }d_k \widehat{v}(w_{ij}^k )\widehat{V}(w_{ji_1 })\widehat{w}(\tau)
([S(w_{i_2 i }^k )\cdot_{\rho}a]\rtimes_{\rho}S(w_{i_1 i_2 }^k )e_{(1)})
\widehat{w}(\tau)\widehat{V^* }(e_{(2)}) \\
& \times\widehat{v^* }(e_{(3)}) \\
& =N\sum_{i,j,k,i_1 , i_2 }d_k \widehat{v}(w_{ij}^k )\widehat{V}(w_{ji_1 })
[S(w_{i_2 i }^k )\cdot_{\rho}a]\tau(S(w_{i_1 i_2 }^k )e_{(1)})
\widehat{w}(\tau)\widehat{V^* }(e_{(2)})\widehat{v^* }(e_{(3)}) \\
& =N\sum_{i,j,k,i_1 , i_2 , t, t_1 }d_k \widehat{v}(w_{ij}^k )\widehat{V}(w_{ji_1 })
[S(w_{i_2 i }^k )\cdot_{\rho}a]
\widehat{w}(\tau)\widehat{V^* }(w_{i_1 t}^k S( w_{t t_1 }^k )e_{(2)}) \\
& \times\tau(S(w_{t_1 i_2 }^k )e_{(1)})\widehat{v^* }(e_{(3)}) \\
& =N\sum_{i,j,k,i_1 , i_2 , t }d_k \widehat{v}(w_{ij}^k )\widehat{V}(w_{ji_1 })
[S(w_{i_2 i }^k )\cdot_{\rho}a]
\widehat{w}(\tau)\widehat{V^* }(w_{i_1 t}^k )\tau(S(w_{t i_2 }^k )e_{(1)})\widehat{v^* }(e_{(2)}) \\
& =N\sum_{i,j,k,i_1 , i_2 , t, s, s_1 }d_k \widehat{v}(w_{ij}^k )\widehat{V}(w_{ji_1 })
[S(w_{i_2 i }^k )\cdot_{\rho}a]
\widehat{w}(\tau)\widehat{V^* }(w_{i_1 t}^k )\tau(S(w_{s_1 i_2 }^k )e_{(1)}) \\
& \times\widehat{v^* }(w_{ts}^k S(w_{ss_1 }^k )e_{(2)}) \\
& =N\sum_{i,j,k,i_1 , i_2 , t, s }d_k \widehat{v}(w_{ij}^k )\widehat{V}(w_{ji_1 })
[S(w_{i_2 i }^k )\cdot_{\rho}a]
\widehat{w}(\tau)\widehat{V^* }(w_{i_1 t}^k )\tau(S(w_{s i_2 }^k )e)\widehat{v^* }(w_{ts}^k ) \\
& =\sum_{i,j,k,i_1 , i_2 , t}d_k \widehat{v}(w_{ij}^k )\widehat{V}(w_{ji_1 })
[S(w_{i_2 i }^k )\cdot_{\rho}a]
\widehat{w}(\tau)\widehat{V^* }(w_{i_1 t}^k )\widehat{v^* }(w_{ti_2 }^k ) \\
& =\sum_{i,j,k,i_1 , i_2 , t}d_k \widehat{v}(w_{ij}^k )\widehat{V}(w_{ji_1 })
\widehat{w}(\tau)[S(w_{i_2 i }^k )\cdot_{\rho}a]
\widehat{V^* }(w_{i_1 t}^k )\widehat{v^* }(w_{ti_2 }^k ) \\
& =\sum_{i,k, i_2 }d_k (\id\otimes w_{ii_2 }^k )(v\rho^{\infty}(\widehat{w}(\tau))\rho([S(w_{i_2 i }^k )\cdot_{\rho}a])v^* ) .
\end{align*}
Hence
\begin{align*}
& ||xax^* -\sum_{i,k, i_2 }d_k (\id\otimes w_{ii_2 }^k )(v\rho^{\infty}(\widehat{w}(\tau))v^*
\sigma([S(w_{i_2 i }^k )\cdot_{\rho}a]) ) || \\
& =|| \sum_{i,k, i_2 }d_k (\id\otimes w_{ii_2 }^k )(v\rho^{\infty}(\widehat{w}(\tau))[\rho([S(w_{i_2 i }^k )\cdot_{\rho}a])v^*
 -v^* \sigma([S(w_{i_2 i}^k )\cdot_{\rho}a ])]) || \\
& \leq \sum_{i,k, i_2 }d_k ||\id\otimes w_{ii_2 }^k || \, ||\rho([S(w_{i_2 i }^k )\cdot_{\rho}a])v^*
 -v^* \sigma([S(w_{i_2 i}^k )\cdot_{\rho}a ])]) || \\
& \leq \sum_{i,k, i_2 }d_k ||\id\otimes w_{ii_2 }^k ||\epsilon=L\epsilon
\end{align*}
Furthermore,
\begin{align*}
& \sum_{i,k, i_2 }d_k (\id\otimes w_{ii_2 }^k )(v\rho^{\infty}(\widehat{w}(\tau))v^* \rho([S(w_{i_2 i}^k )\cdot_{\rho}a])) \\
&=\sum_{i,k, i_2 , t}d_k (v\rho^{\infty}(\widehat{w}(\tau))v^* )^{\widehat{}}(w_{it}^k )
[w_{ti_2 }^k S(w_{i_2 i}^k )\cdot_{\rho}a] \\
& =\sum_{i,k}d_k (v\rho^{\infty}(\widehat{w}(\tau))v^* )^{\widehat{}}(w_{ii}^k )a
=N(\id\otimes e)(v\rho^{\infty}(\widehat{w}(\tau))v^* )a .
\end{align*}
We recall that $\rho_1 =\Ad(v)\circ\rho$, $u=(v\otimes 1^0 )(\rho\otimes\id)(v)(\id\otimes\Delta^0 )(v^* )$
and that $(\rho_1 , u)$ is a twisted coaction of $H^0 $ on $A$ which is exterior equivalent to
$\rho$. Then by Lemmas \ref{lem:projection2} and \ref{lem:projection},
$$
N(\id\otimes e)(v\rho^{\infty}(\widehat{w}(\tau))v^* )=N[e\cdot_{\rho_1 ,u}\widehat{w}(\tau)]=1 .
$$
Hence
$$
\sum_{i,k, i_2 }d_k (\id\otimes w_{ii_2 }^k )(v\rho^{\infty}(\widehat{w}(\tau))v^* 
\rho([S(w_{i_2 i}^k )\cdot_{\rho}a]))=a .
$$
It follows that
\begin{align*}
& ||xax^* -a|| \\
& =||xax^* -\sum_{i,k, i_2 }d_k (\id\otimes w_{ii_2 }^k )(v\rho^{\infty}(\widehat{w}(\tau))v^* 
\sigma([S(w_{i_2 i}^k )\cdot_{\rho}a]))|| \\
& +\sum_{i,k, i_2 }d_k (\id\otimes w_{ii_2 }^k )(v\rho^{\infty}(\widehat{w}(\tau))v^* 
\sigma([S(w_{i_2 i}^k )\cdot_{\rho}a])) \\
& -\sum_{i,k, i_2 }d_k (\id\otimes w_{ii_2 }^k )(v\rho^{\infty}(\widehat{w}(\tau))v^* 
\rho([S(w_{i_2 i}^k )\cdot_{\rho}a])) || \\
& <L \epsilon +\sum_{i,k,i_2 }d_k ||\id\otimes w_{ii_2 }^k || \, ||\sigma([S(w_{i_2 i}^k )\cdot_{\rho}a])
-\rho([S(w_{i_2 i}^k )\cdot_{\rho}a]) || \\
& <L \epsilon +\sum_{i,j,k }d_k ||\id\otimes w_{ij }^k || \max_{i,j,k}||\sigma([S(w_{ij}^k )\cdot_{\rho}a])
-\rho([S(w_{ij}^k )\cdot_{\rho}a]) || \\
& <L \epsilon +L \max_{i,j,k}||\sigma([S(w_{ij}^k )\cdot_{\rho}a])
-\rho([S(w_{ij}^k )\cdot_{\rho}a]) ||,
\end{align*}
where $L=\sum_{i,j,k }d_k ||\id\otimes w_{ij }^k ||$.
Then we obtain the conclusion.
\end{proof}
By Lemmas \ref{lem:key1}, \ref{lem:key2}, \ref{lem:key3} and  \ref{lem:key4}, we obtain Lemma \ref{lem:key0}.
We note that the constant positive number $L$ in the above proofs, does not depend on
coactions $\rho$ and $\sigma$ but depends on only $H^0$. Also, we note that if a coaction $\rho$
of $H^0$ on $A$ has the Rohlin property, then a coaction $(\alpha\otimes\id)\circ\rho\circ\alpha^{-1}$
of $H^0$ on $A$ has also the Rohlin property for any automorphism $\alpha$ of $A$.

\begin{thm}\label{thm:main}Let $A$ be a separable unital $C^*$-algebra and let
$\rho$ and $\sigma$ be coactions of a finite dimensional $C^*$-Hopf algebra $H^0$ on $A$
with the Rohlin property . We suppose that $\rho$ and $\sigma$ are approximately
unitarily equivalent. Then there is an approximately inner automorphism $\theta$ such that
$$
\sigma=(\theta\otimes\id)\circ\rho\circ\theta^{-1} .
$$
\end{thm}
\begin{proof}We shall show this theorem in the same strategy as in the proof of \cite [Theorem 3.5]{Izumi:group}.
We choose an increasing family $\{F_n \}_{n=0}^{\infty}$ of finite subsets of $A$ whose union is dense in $A$.
By induction using Lemma \ref{lem:key0}, we can construct an increasing family $\{G_n \}_{n=0}^{\infty}$ of
finite subsets of $A$ whose union is dense in $A$, a sequence $\{x_n \}$ of unitary elements in $A$ and
a family of coactions $\rho_{2n}$, $\sigma_{2n+1}$, $n=0, 1,2,\dots$, of $H^0$ on $A$ satisfying the following
conditions:
\begin{align*}
& \rho_0 =\rho, \quad\sigma_1 =\sigma \\
& \rho_{2n+2}=\Ad(x_{2n}\otimes 1^0 )\circ\rho_{2n}\circ\Ad(x_{2n}^* ), \quad n=0,1,2, \dots, \\
& \sigma_{2n+1}=\Ad(x_{2n-1}\otimes 1^0 )\circ\sigma_{2n-1}\circ\Ad(x_{2n-1}^* ),  \quad n=1,2,\dots, \\
& F_{2n}^1 =\bigcup_{i,j,k}[S(w_{ij}^k )\cdot_{\sigma_{2n+1}}F_{2n}] ,\quad n=0,1,\dots, \\
& F_{2n+1}^1 =\bigcup_{i,j,k}[S(w_{ij}^k )\cdot_{\sigma_{2n+2}}F_{2n+1}] ,\quad n=0,1,\dots, \\
& G_0 =F_0 \cup F_0^1 \\
& G_{2n+1}=G_{2n}\cup F_{2n+1}\cup F_{2n+1}^1 , \quad n=0,1,\dots, \\
& G_{2n+2}=G_{2n+1}\cup F_{2n+2}\cup F_{2n+2}^1 , \quad n=0,1,\dots, \\
& ||\sigma_{2n+1}(a)-\rho_{2n+2}(a)||<\frac{1}{2^{2n}}, \quad a\in G_{2n} , \quad n=0,1,\dots, \\
& ||\sigma_{2n+3}(a)-\rho_{2n+2}(a)||<\frac{1}{2^{2n+1}} , \quad a\in G_{2n+1}, \quad n=0,1,\dots, \\
& ||x_{2n+1}a-ax_{2n+1}||<\frac{1}{2^{2n+1}}+L\max_{i,j,k}||\rho_{2n+2}([S(w_{ij}^k )\cdot_{\sigma_{2n+1}}a]) \\
& -\sigma_{2n+1}([(w_{ij}^k )\cdot_{\sigma_{2n+1}}a])|| <\frac{2+L}{2^{2n}}, \quad a\in G_{2n}, \quad n=0,1,\dots,\\
& ||x_{2n}a-ax_{2n}||<\frac{1}{2^{2n}}+L\max_{i,j,k}||\sigma_{2n+1}([S(w_{ij}^k )\cdot_{\rho_{2n}}a]) \\
& -\rho_{2n}([S(w_{ij}^k )\cdot_{\rho_{2n}}a])||<\frac{2+L}{2^{2n-1}}, \quad a\in G_{2n-1}, \quad n=1,2\dots,
\end{align*}
In the same way as in the proof of \cite [Theorem 3.5]{Izumi:group}, we can obtain the conclusion.
\end{proof}

In the rest of this section, we shall study coactions having the Rohlin property
of a finite dimensional $C^*$-Hopf algebra on
a UHF-algebra of type $N^{\infty}$. Let $A$ be a UHF-algebra of type $N^{\infty}$.
Let $M_n (\BC)$ be the $n\times n$-matrix algebra over $\BC$ and $\{f_{ij} \}$
a system of matrix units of $M_n (\BC)$.

\begin{lemma}\label{lem:homo1}Let $\rho$ be a unital homomorphism of $A$
to $A\otimes M_n (\BC)$ and $\rho_* $ the homomorphism of $K_0 (A)$
to $K_0 (A\otimes M_n (\BC))$ induced by $\rho$.
Then $\rho_* ([\frac{1}{N^l }])=n[\frac{1}{N^l }]$ for any $l\in\BN\cup\{0\}$.
\end{lemma}
\begin{proof}Since $\rho(1)=1\otimes I_n $,
$\rho_* ([1])=[1\otimes I_n ]=n[1\otimes f_{11} ]=n[1]$.
Hence $N\rho_* ([\frac{1}{N}])=\rho_* ([1])=n[1]$. Since $K_0 (A)=\BZ[\frac{1}{N}]$
is torsion-free, $\rho_* ([\frac{1}{N}])=n[\frac{1}{N}]$.
\end{proof}

\begin{lemma}\label{lem:homo2}Let $\rho$ be a unital homomorphism of $A$ to
$A\otimes M_n (\BC)$. Then there is a sequence $\{u_k \}$ of unitary elements in
$A\otimes M_n (\BC)$ such that for any $x\in A$
$$
\rho (x)=\lim_{k\to\infty} u_k (x\otimes I_n )u_k^* .
$$
\end{lemma}
\begin{proof}Modifying the proof of Blackadar \cite[7.7 Exercises and Problems]
{Blackadar:K-Theory} we can prove this lemma. 
Let $\{A_k \}$ be a increasing sequence of full matrix algebras
over $\BC$ with $\overline{\cup_k A_k }=A$. Let $\{e_{ij} \}$ be a system of matrix
units of $A_k $. Since $A$ has the cancellation property, by Lemma \ref{lem:homo1},
$\rho(e_{11})\sim e_{11}\otimes I_n $ in
$A\otimes M_n (\BC)$. Hence there is a
partial isometry $w\in A \otimes M_n (\BC)$ such that
$$
w^* w=E_{11}, \quad ww^* =\rho(e_{11}) ,
$$
where $E_{ij}=e_{ij}\otimes I_n $ for any $i,j$.
Let $u_k =\sum_i \rho(e_{i1})wE_{1i}$. Then $u_k $ is a unitary element in
$A \otimes M_n (\BC)$ by easy computations. Let $x\in A_k $. Then we can write
that $x=\sum_{i,j}\lambda_{ij}e_{ij}$, where $\lambda_{ij}\in\BC$.
Hence by easy computations, we can see that $\rho(x)=u_k (x\otimes I_n )u_k^* $.
Since $\overline{\cup_k A_k }=A$, we obtain that for any $x\in A$,
$\rho(x)=\lim_{k\to\infty}u_k (x\otimes I_n )u_k^* $ by routine computations.
\end{proof}

\begin{lemma}\label{lem:homo3}Let $\rho$ be a unital homomorphism
of $A$ to $A\otimes H^0$, where $H^0$ is a finite dimensional $C^*$-algebras.
Then there is a sequence $\{u_k \}$ of unitary elements in
$A\otimes H^0$ such that for any $x\in A$
$$
\rho(x)=\lim_{k\to\infty}u_k (x\otimes 1)u_k^* .
$$
\end{lemma}
\begin{proof}Let $\{p_l \}$ be a family of minimal central projections in $H^0$.
For any $l$ and $x\in A$, let
$$
\rho_l (x)=\rho(x)(1\otimes p_l  ) .
$$
Then by Lemma \ref{lem:homo2}, there is a sequence $\{u_k^{(l)}\}$ of unitary elements
in $A\otimes p_l H^0$ such that $\rho_l (x)=\lim_{k\to\infty}u_k^{(l)}(x\otimes p_l )u_k^{(l)}$
for any $x\in A$. Let $u_k =\oplus_l u_k^{(l)} $. Then  we can see that
$\{u_k \}$ is a desired sequence by easy computations.
\end{proof}

\begin{cor}\label{cor:unique}Let $H$ be a finite dimensional $C^*$-Hopf algebra with
dimension $N$ and let $A$ be a UHF-algebra of type $N^{\infty}$. Let $\rho$ be a coaction
of $H^0$ on $A$ with the Rohlin property constructed in Section \ref{sec:example}.
Then for any coaction $\sigma$ of $H^0$ on $A$ with the Rohlin property,
there is an approximately inner automorphism $\theta$ of $A$ such that
$$
\sigma= (\theta\otimes\id)\circ\rho\circ\theta^{-1} .
$$
\end{cor}
\begin{proof}By Lemma \ref{lem:homo3}, $\sigma$ is approximately unitarily equivalent to
$\rho$. Hence by Theorem \ref{thm:main}, we obtain the conclusion.
\end{proof}

\section{Appendix}\label{sec:appendix}
In the previous paper \cite{OKT:rohlin}, we introduced the Rohlin property for weak coactions of a finite dimensional
$C^*$-Hopf algebra on a unital $C^*$-algebra. In this section,
we shall show that if there is a weak coaction with the Rohlin property in the sense of
\cite{OKT:rohlin} of a finite dimensional $C^*$-Hopf algebra $H$ on
a unital $C^*$-algebra $A$, then $H$ is commutative.
Recall that a weak coaction $\rho$ of $H$ on $A$ has the Rohlin property in the sense of
\cite{OKT:rohlin} if there is 
a monomorphism
$\pi$ of $H$ into $A_{\infty}$
such that for any $h\in H$, $\rho^{\infty}(\pi(h))=\pi(h_{(1)})\otimes h_{(2)}$.
Let $\{w_{ij}^k \}$ be a system of comatrix units of $H$.

\begin{lemma}\label{lem:compact}With the above notations,
$(H\otimes 1)\Delta(H)=H\otimes H$.
\end{lemma}
\begin{proof}For any $i, j, k$, $\Delta(w_{ij}^k )=\sum_t w_{it}^k \otimes w_{tj}^k$.
Since $\sum_i w_{it}^{k*} w_{is}^k=\begin{cases} 1 & \text{if $s=t$} \\
0 & \text{if $s\ne t$}\end{cases}$ for any $k$ by \cite [Theorem 2.2, 2]{SP:saturated}, we can obtain that
$$
\sum_i (w_{it}^{k*}\otimes 1)\Delta(w_{ij}^k ) =\sum_{i, s} w_{it}^{k*} w_{is}^k \otimes w_{sj}^k
=1\otimes w_{tj}^k .
$$
Thus we obtain the conclusion.
\end{proof}

\begin{lemma}\label{lem:effective}With the above notations, let $\rho$ be a weak
coaction of $H$ on $A$ with the Rohlin property in the sense
of \cite{OKT:rohlin}. Then $\overline{(A\otimes 1)\rho(A)}=A\otimes H$.
\end{lemma}
\begin{proof}Since $\rho$ has the Rohlin property in the sense of \cite{OKT:rohlin},
there is a monomorphism $\pi$ of $H$ into $A_{\infty}$. First, we show that
$(A^{\infty}\otimes 1)\rho^{\infty}(A^{\infty})=A^{\infty}\otimes H$.
Since $\rho^{\infty}\circ\pi=(\pi\otimes\id)\circ\Delta$,
\begin{align*}
(\pi(H)\otimes 1)\rho^{\infty}(\pi(H))& =(\pi(H)\otimes 1)(\pi\otimes\id)(\Delta(H)) \\
& =(\pi\otimes\id)((H\otimes 1)\Delta(H))=\pi(H)\otimes H
\end{align*}
by Lemma \ref{lem:compact}. Since $1\otimes w_{ij}^k \in \pi(H)\otimes H$,
$1\otimes w_{ij}^k \in (A^{\infty}\otimes 1)\rho^{\infty}(A^{\infty})$.
Thus we can see that $(A^{\infty}\otimes 1)\rho^{\infty}(A^{\infty})=A^{\infty}\otimes H$.
For any $x\in A\otimes H$, there are $a_1, \dots, a_n , b_1, \dots b_n \in A^{\infty}$
such that $x=\sum_{i=1}^n (a_i \otimes 1)\rho^{\infty}(b_i )$. That is,
$$
||x-\sum_{i=1}^n (a_i^{(k)} \otimes 1)\rho^{\infty}(b_i^{(k)} )||\to 0 \quad (k\to\infty),
$$
where $a_i =(a_i^{(k)}), b_i =(b_i^{(k)})$ and $a_i^{(k)}, b_i^{(k)}\in A$ for any $k, i$.
Hence $x\in \overline{(A\otimes 1)\rho(A)}$.
\end{proof}

\begin{prop}\label{prop:commutant}Let $\rho$ be a weak coaction of $H$ on $A$ with the Rohlin property
in the sense of \cite{OKT:rohlin}
and $\pi$ a monomorphism of $H$ to $A_{\infty}$. Then
$\rho^{\infty}(\pi(H))\subset (A\otimes H)' \cap(A^{\infty}\otimes H)$.
\end{prop}
\begin{proof}Let $a, b\in A$ and $h\in H$. Then
\begin{align*}
\rho^{\infty}(\pi(h))(a\otimes 1)\rho(b) & =(\pi(h_{(1)})\otimes h_{(2)})(a\otimes 1)\rho(b)
=(a\pi(h_{(1)})\otimes h_{(2)})\rho(b) \\
& =(a\otimes 1)\rho^{\infty}(\pi(h))\rho(b)=(a\otimes 1)\rho^{\infty}(\pi(h)b) \\
& =(a\otimes 1)\rho(b)\rho^{\infty}(\pi(h)).
\end{align*}
Therefore we obtain the conclusion by Lemma \ref{lem:effective}.
\end{proof}

\begin{prop}\label{prop:commutant2}Let $\rho$ be a weak coaction of $H$ on $A$ with the Rohlin property.
Let $x$ be any element in $A\otimes H$.
Then for any $h\in H$, $(1\otimes h)x=x(1\otimes h )$.
\end{prop}
\begin{proof}Let $\pi$ be a monomorphism of $H$ into $A_{\infty}$
such that $\rho^{\infty}(\pi(h))=\pi(h_{(1)})\otimes h_{(2)}$ for any $h\in H$.
By the proof of Lemma \ref{lem:effective}, we can see that
$$
1\otimes H\subset \pi(H)\otimes H=(\pi(H)\otimes 1)\rho^{\infty}(\pi(H)) .
$$
Hence it suffices to show that for any $h\in H$,
\newline
(1) $(\pi(h)\otimes 1)x=x(\pi(h)\otimes 1)$,
\newline
(2) $\rho^{\infty}(\pi(h))x=x\rho^{\infty}(\pi(h))$.
\newline
Indeed, since $x\in A\otimes H$ and $\pi(h)$ commute with any element in $A$ for any
$h\in H$, we obtain (1). Also, we can obtain (2) by Proposition \ref{prop:commutant}
\end{proof}

\begin{cor}\label{cor:commutative}Let $A$ be a unital $C^*$-algebra and $H$ a finite
dimensional $C^*$-Hopf algebra. If there is a weak coaction of $H$ on $A$ with the Rohlin property
in the sense of \cite{OKT:rohlin}, then $H$ is commutative.
\end{cor}
\begin{proof}This is immediate by Proposition \ref{prop:commutant2}.
\end{proof}

\end{document}